\documentclass[11pt]{amsart}
\usepackage{cite}
\usepackage[width=1.2\textwidth]{caption}
\usepackage{wrapfig, lipsum,booktabs}
\usepackage{graphicx}
\usepackage{amssymb}
\usepackage{epstopdf}
\usepackage{enumitem} 
\usepackage{verbatim}
\usepackage{bm}
\usepackage{multicol}
\usepackage{multirow}
\usepackage{subfigure}
\usepackage{fancyhdr} 
\usepackage{float}
\usepackage{stmaryrd}
\usepackage{scalerel} 
\usepackage{color}
\usepackage{soul}
\usepackage{tikz}
\usepackage{todonotes}
\usetikzlibrary{patterns}
\usepackage{pgfplots}
\usepackage{cancel}
\usepackage{amsmath}
\usepackage{mathtools}
\usepackage{calrsfs}
\DeclareMathAlphabet{\pazocal}{OMS}{zplm}{m}{n}
\newtheorem{theorem}{Theorem}[section]
\newtheorem{lemma}{Lemma}[section]
\newtheorem{corollary}{Corollary}[section]
\newtheorem{remark}{Remark}[section]

\usepackage{algorithm} 
\usepackage[noend]{algpseudocode} 
\usepackage{xpatch}
\makeatletter
\xpatchcmd{\algorithmic}{\itemsep\z@}{\itemsep=0.4ex plus2pt}{}{}
\makeatother 
\usepackage{multirow}
\usepackage{amsmath,amssymb,eucal}
\usepackage{graphicx,subfigure,epsfig}
\usepackage{psfrag}
\usepackage{url}
\usepackage[top=1in, bottom=1.in, left=1in, right=1in]{geometry}

\newcommand{\Th}{\mathcal{T}_h}

\newcommand{\Eh}{\mathcal{E}_h}

\setulcolor{red}

\newcommand{\poly}{\mathcal{P}}
\newcommand{\trans}{\mathrm{T}}
\newcommand{\CR}{\scaleto{CR}{5pt}}
\newcommand{\T}{\scaleto{T}{5pt}}
\def\dunderline#1{\underline{\underline{#1}}}
\makeatletter
\renewcommand*\env@matrix[1][\arraystretch]{%
	\edef\arraystretch{#1}%
	\hskip -\arraycolsep
	\let\@ifnextchar\new@ifnextchar
	\array{*\c@MaxMatrixCols c}}
\makeatother

\newcommand\red[1]{{\color{black}{#1}}}

\pgfplotsset{compat=1.15}
\begin{document}
\title[MG for {\sf HDG-P0}]{Optimal Geometric Multigrid Preconditioners for {\sf HDG-P0} Schemes for the reaction-diffusion equation and the Generalized Stokes equations}

\author{Guosheng Fu}
\address{Department of Applied and Computational Mathematics and 
Statistics, University of Notre Dame, USA.}
\email{gfu@nd.edu}
\thanks{We gratefully acknowledge the partial support of this work
	by the U.S. National Science Foundation through grant DMS-2012031.}
\author{Wenzheng Kuang}
\address{Department of Applied and Computational Mathematics and 
	Statistics, University of Notre Dame, USA.}
\email{wkuang1@nd.edu}

\begin{abstract}
We present the lowest-order hybridizable discontinuous Galerkin schemes with \red{numerical integration (quadrature)}, denoted as {\sf HDG-P0}, for the reaction-diffusion equation and the generalized Stokes equations 
\red{on conforming simplicial meshes} in two- and three-dimensions. Here by lowest order, we mean that the (hybrid) finite element space for the global HDG facet degrees of freedom (DOFs) is the space of piecewise constants on the mesh skeleton.
\red{A discontinuous piecewise linear space is used for the approximation of the local primal unknowns.}
We give the optimal a priori error analysis of the proposed {\sf HDG-P0} schemes, which hasn't appeared in the literature yet for HDG discretizations as far as numerical integration is concerned.

\red{Moreover, we propose optimal geometric multigrid preconditioners for the statically condensed {\sf HDG-P0} linear systems on conforming simplicial meshes.}
In both cases, we first establish the equivalence of the statically condensed HDG system with a
(slightly modified) nonconforming Crouzeix-Raviart (CR) discretization, where the global (piecewise-constant) HDG finite element space on the mesh skeleton has a natural one-to-one correspondence to the nonconforming CR (piecewise-linear) finite element space that live on the whole mesh.
This equivalence then allows us to use the well-established nonconforming geometry multigrid theory to precondition the condensed HDG system. Numerical results in two- and three-dimensions are presented to verify our theoretical findings.
\end{abstract}
\maketitle

\section{Introduction}
Since \red{their} first unified introduction for the second order elliptic equation \cite{CockburnGopalakrishnanLazarov09}, hybridizable discontinuous Galerkin (HDG) methods have been gaining popularity for numerically solving partial differential equations (PDEs), and have been successfully applied in computational fluid dynamics\cite{cockburn2018discontinuous,qiu2016superconvergent}, wave propagation\cite{cockburn2016hdg, fernandez2018hybridized} and continuum mechanics\cite{nguyen2012hybridizable,fu2021locking}. Besides succeeding attractive features from the discontinuous Galerkin (DG) schemes including 
local conservation, allowing unstructured meshes with hanging nodes, and ease of $hp$-adaptivity, linear systems of HDG schemes can be statically condensed such that only global DOFs on the mesh skeleton remain, resulting in increased sparsity, decreased matrix size, and computational cost\cite{Cockburn16}. 
One HDG technique, known as \textit{projected jumps} \cite[Remark 1.2.4]{lehrenfeld2010hybrid},  further reduces the size of the condensed HDG scheme by making the polynomial spaces of the facet unknowns one order lower than those of the primal variables without loss of accuracy. Thus superconvergence is obtained for the primal variables from the point of view of the globally coupled DOFs. \red{The superconvergence result} was rigorously proved for the primal HDG schemes with projected jumps 
for diffusion and Stokes problems in \cite{oikawa2015hybridized, oikawa2015reduced}, and also for the mixed HDG schemes with projected jumps for linear elasticity \cite{qiu2018hdg}, convection-diffusion \cite{qiu2016hdg}, and incompressible Navier-Stokes \cite{qiu2016superconvergent}. However,  no numerical integration effects were considered in these works.

Large-scale simulations based on HDG schemes still face the challenge of constructing scalable and efficient solvers for the condensed linear systems, for which geometric multigrid techniques have been previously explored as either linear system solvers  or preconditioners for iterative methods \cite{CockburnDubois14,lu2021homogeneous}.
The main difficulty in designing geometric multigrid algorithms for HDG is the construction of intergrid transfer operators between  the coarse and fine meshes, since the spaces of global unknowns on hierarchical meshes live (only) on mesh skeletons and are non-nested. 

In the literature, two techniques have been used to overcome this difficulty. The first approach, which is in the spirit of auxiliary space preconditioning \cite{Xu96}, uses the conforming piecewise linear finite element method on the same mesh as the coarse grid solver and applies the standard geometric multigrid for conforming finite elements from then on. Hence, the difficulty of constructing integergrid transfer operators between global HDG facet spaces on fine and coarse grids was completely bypassed as the coarse grid HDG space was never used in the multigrid algorithm.
This technique was first introduced for HDG schemes in \cite{CockburnDubois14} for the diffusion problem, and similarly used in \cite{wgchen,chen2014robust, betteridge2021multigrid, fabien2019manycore} for diffusion and other equations.

The second approach still keeps the HDG facet spaces on coarser meshes to construct the multigrid algorithms, which is referred to as {\it homogeneous multigrid} in \cite{lu2021homogeneous} since it uses the same HDG discretization scheme on all mesh levels. 
Here the issue of stable intergrid transfer operators between coarse and fine grid HDG facet spaces  has to be addressed. 
Several ``{obvious}" HDG prolongation operators were numerically tested in Tan's 2009 PhD thesis \cite{Tan09}, which, however, failed to be optimal.
The work \cite{wildey2019unified} proposed an intergrid transfer operator based on Dirichlet-to-Neumann maps where operators on coarser levels were recursively changed for energy preservation.
Numerical results in \cite{wildey2019unified} supported the robustness of their multigrid algorithm, but no theoretical analysis was presented.
More recently, a three-step procedure was used in \cite{lu2021homogeneous} to construct a robust prolongation operator:  first, define a continuous extension operator from coarse grid HDG facet space to an $H^1$-conforming finite element space of the same degree on the same mesh by averaging; next, use the natural injection from the coarse grid conforming finite element space to the fine grid conforming finite element space;
lastly, restrict the fine grid conforming finite elements data to the fine grid mesh skeleton to recover the fine grid HDG facet data (see more details in \cite{lu2021homogeneous}).
The optimal convergence result of the standard V-cycle algorithm was 
proven in \cite{lu2021homogeneous} for an HDG scheme with stabilization parameter  $\tau=O(1/h)$ for polynomial degree $k\ge 1$ for the diffusion problem, where $h$ is the mesh size.
Therein, taking stabilization parameter $\tau=O(1/h)$ is crucial in the optimal multigrid analysis in \cite{lu2021homogeneous}, which however is not usually preferred in practice as the resulting scheme loses the important property of superconvergence due to too strong stabilization. 
Nevertheless, the numerical results presented in \cite{lu2021homogeneous} suggested the multigrid algorithm was still optimal when the stabilization parameter $\tau=O(1)$, which has the superconvergence property. 
\red{We note that in these cited works, the finite element spaces in the HDG schemes use the same polynomial degree for all the variables.}
Intergrid operators based on 
(superconvergent) post-processing and smoothing was
also presented in \cite{di2021h,di2021towards}
for the hybrid high-order (HHO) methods.
We finally note that the above cited works in the second approach only focus on the pure diffusion problem.


In this study, we work on HDG discretizations for the reaction-diffusion equation and the generalized Stokes equations \red{on conforming simplicial meshes}.
Specifically, we construct lowest-order HDG schemes with projected jumps and numerical integration, denoted as {\sf HDG-P0}, for these two sets of equations and present their optimal (superconvergent) a priori error analysis.
\red{We use the space of piecewise constants for the approximation of the (global) hybrid facet unknowns, and piecewise linears for the (local) primal unknowns, and take stabilization parameter $\tau=O(1/h)$. Due to the use of piecewise linears for the element-wise primal unknowns and the projected jumps in the stabilization, the proposed schemes still enjoy the superconvergence property.}
We further provide their optimal multigrid preconditioners for the global HDG linear systems with easy-to-implement intergrid transfer operators.
The key of our optimal multigrid preconditioner is the establishment of the equivalence of the condensed {\sf HDG-P0} schemes with the (slightly modified) nonconforming CR discretizations.
Such equivalence enables us to build HDG multigrid preconditioners
easily following the rich literature 
on nonconforming multigrid theory for the reaction-diffusion equation \cite{brenner1989optimal, braess1990multigrid, bramble1991analysis, chen1994analysis, brenner1999convergence, brenner2004convergence, duan2007generalized} and the generalized Stokes equations \cite{brenner1990nonconforming, turek1994multigrid, stevenson1998cascade, brenner1993nonconforming, schoberl1999multigrid, schoberl1999robust}.
We specifically note that equivalence of the lowest order Raviart-Thomas mixed method with certain nonconforming methods
was well-known in the literature \cite{arnold1985mixed, marini1985inexpensive}, and multigrid algorithms for hybrid-mixed methods have been designed  based on this fact \cite{brenner1992multigrid, chen1996equivalence}. 
Moreover, an equivalence between the lowest-order primal HDG schemes with projected jumps and the nonconforming CR schemes
for pure diffusion and Stokes problems have already been established by Oikawa in \cite{oikawa2015hybridized, oikawa2015reduced}. Here we use the mixed HDG formulation with projected jumps and establish the stronger equivalence of the two methods for reaction-diffusion and generalized Stokes equations, which has application in time-dependent diffusion and time-dependent incompressible flow problems. As far as the low-order terms are concerned, we find the mixed HDG formulation more robust to formulate and easier to analyze compared with their primal HDG counterparts. In particular, our final algorithm does not involve any (sufficiently large) tunable stabilization parameters.  
\red{We note that our {\sf HDG-P0} multigrid methods rely on conforming and shape-regular simplicial meshes, which can not handle hanging nodes. Besides, the approach of building multigrid methods based on the equivalence to the CR discretizations is specific to {\sf HDG-P0}, which 
can not be directly applied to higher order HDG schemes.
However, it can be used as a building block when constructing robust $hp$-multigrid methods for higher order HDG schemes, which is our ongoing research.
}

The rest of the paper is organized as follows. Basic notations and the finite element spaces to be applied in {\sf HDG-P0} are introduced in Section \ref{sec:notation}. In Section \ref{sec:MG4ellip}, we propose the {\sf HDG-P0} scheme for the mixed formulation of the reaction-diffusion equation, give the optimal a priori error analysis, and present the optimal multigrid preconditioner for the static condensed linear system. We then focus on the {\sf HDG-P0} scheme for the generalized Stokes equations
in Section~\ref{sec:MG4Stokes} and present the optimal a priori error estimates and multigrid preconditioner. Numerical results in two- and three-dimensions are presented in Section \ref{sec:numExp} to verify the theoretical findings in Section \ref{sec:MG4ellip}--\ref{sec:MG4Stokes}, and we conclude in Section~\ref{sec:conclude}.

\section{Notations and Preliminaries}
\label{sec:notation}
Let $\Omega\subset \mathbb{R}^d$, $d\in\{2,3\}$,
be a bounded polygonal/polyhedral domain with boundary $\partial \Omega$.
We denote $\Th$ as a conforming, \red{shape-regular}, and quasi-uniform simplicial triangulation of  $\Omega$.
Let $\Eh$ be the collection of the facets (edges in 2D, faces in 3D) of the triangulation $\Th$, which is also referred to as the \textit{mesh skeleton}.
We split the mesh skeleton $\Eh$ into the boundary contribution 
$\Eh^\partial := \{F\in\Eh :\; F\subset\partial\Omega \}$, and the interior contribution $\Eh^o := \Eh \backslash \Eh^\partial$.
For any simplex $K\in\Th$ with boundary $\partial K$, we denote the measure of $K$ by $|K|$, the $L_2$-inner product on $K$ by $(\cdot,\;\cdot)_K$, and 
the $L_2$-inner product on $\partial K$ by $\langle\cdot,\; \cdot\rangle_{\partial K}$.
We denote $\underline{n}_K$ as the unit normal vector on the element boundaries $\partial K$ pointing outwards.
Furthermore, we denote $h_{K}$ as the mesh size of $K$ defined only on the facets $\partial K$, whose restriction on a facet $F\subset\partial K$ is given as $h_K|_{F}:=|K|/|F|$, \red{where $|F|$ is the measure of $F$,} and set 
\[
h := \max_{K\in\Th}\max_{F\subset \partial K}h_K|_{F}
\]
as the maximum mesh size of the triangulation $\Th$.
To further simplify the notation, we define the discrete $L_2$-inner product on the whole mesh as $(\cdot,\; \cdot)_{\Th} := \sum_{K\in\Th}(\cdot,\; \cdot)_K$ and the $L_2$-inner product on all element boundaries as $\langle\cdot,\; \cdot\rangle_{\partial\Th} := \sum_{K\in\Th}\langle\cdot,\; \cdot\rangle_{\partial K}$.

As usual, we denote $\|\cdot\|_{m,p,S}$ and $|\cdot|_{m,p,S}$ as the norm and semi-norm of the Sobolev spaces $W^{m,p}(S)$ for the domain $S\subset \mathbb{R}^d$, with $p$ omitted when $p=2$, and $S$ omitted when $S=\Omega$ is the whole domain. 
To this end, we define some numerical integration rules that will be used in this work.
Given a simplex $K\subset \mathbb{R}^d$, we define 
the following two integration rules for the volume integral $\int_K g\,\mathrm{dx}$:
\begin{align*}
    Q_K^0(g):= &\; |K|\,g(m_K),\\
    Q_K^1(g):= &\;\frac{|K|}{d + 1}\sum_{i = 1}^{d + 1}g(m_K^i),
\end{align*}
where $m_K$ is the barycenter of the element $K$, and
$m_K^i$ is the barycenter of the $i$-th facet of $K$ for $i\in[1, d+1]$.
The one-point integration rule $Q_K^0$ is exact for linear functions, while the $(d+1)$-point integration rule $Q_K^1$
is exact for linear functions when $d=3$ and exact for quadratic functions when $d=2$.
We also define a one-point quadrature rule for the facet integral $\int_Fg\,\mathrm{ds}$ over a facet $F\subset \mathbb{R}^{d-1}$:
\begin{align*}
    Q_{F}^0(g):= |F|\, g(m_F),
\end{align*}
where $m_F$ is the barycenter of the facet $F$, which is exact for linear functions. 
Then we denote the numerical integration on the element boundary $\partial K$ as
\begin{align*}
    Q_{\partial K}^0(g):= \sum_{i=1}^{d+1}
    Q_{F^i}(g) = 
    \sum_{i=1}^{d+1}  |F^i|\, g(m_{F^i}),
\end{align*}
where $\{F^i\}_{i=1}^{d+1}$ is the collection of facets of $K$ with $F^i\subset \partial K$.

Moreover, we will use the following finite element spaces:
\begin{subequations}
\label{fespace}
\begin{alignat}{2}
    &W_h &&:= \{
        q_h \in L_2(\Omega):\; q_h|_K \in \poly^0(K),\; \forall K\in\mathcal{T}_h \},
    \\
    &W_h^0 &&:= \{
        q_h \in \red{W_h}:\; \int_{\Omega}q_h\mathrm{dx} = 0\},
    \\
    &V_h &&:= \{
        v_h \in L_2(\Omega):\; v_h|_K \in \poly^1(K),\; \forall K\in\mathcal{T}_h\},
    \\
    &V_h^{\CR} && := \{v_h \in V_h:\; \text{$v_h$ is continuous at the barycenter $m_F$ of $F$, $\forall F\in\Eh^o$}
        \},
    \\
    &V_h^{\CR,0} && := \{v_h \in V_h^{\CR}:\; \text{$v_h(m_F) = 0$, $\forall F\in\Eh^\partial$}\},
    \\
    &M_h &&:= \{
        \widehat{v}_h \in L_2(\mathcal{E}_h):\; \widehat{v}_h|_F \in \poly^0(F),\; \forall F\in\mathcal{E}_h\},
    \\
    &M_h^0 &&:= \{
        \widehat{v}_h \in M_h:\;\widehat{v}_h|_{F} = 0,\; \forall F\subset\partial\Omega\}.
\end{alignat}
\end{subequations}
We use underline to denote the vector version of the spaces $\underline{\;\cdot\;}:={[\;\cdot\;]}^d$, and double-underline to denote the matrix version $\dunderline{\;\cdot\;}:={[\;\cdot\;]}^{d\times d}$; e.g., 
$\dunderline{W}_h:=[W_h]^{d\times d}$ is the $d\times d$ tensor space of piecewise constant functions.
 
For two positive constants $a$ and $b$, we denote $a\lesssim b$ if there exists a positive constant $C$ independent of  mesh size and model parameters 
such that $a\le Cb$. We denote $a \simeq b$ when $a\lesssim b$ and $b\lesssim a$.

\section{{\sf HDG-P0} for the reaction-diffusion equation}
\label{sec:MG4ellip}
\subsection{The model problem}
Let $f\in L_2(\Omega)$, $\alpha\in C^1(\bar{\Omega})$ and 
$\beta\in C^0(\bar{\Omega})$. 
We assume there exist positive constants $\alpha_0$, $\alpha_1$ and $\beta_1$ with
$\alpha_1\lesssim \alpha_0$
such that $\alpha_1\ge \alpha\ge \alpha_0>0$ and 
$\beta_1\ge \beta\ge 0$. We consider the following model problem with a homogeneous Dirichlet boundary condition:
\begin{align}
    \label{ellipOrig}
- \nabla\cdot(\alpha \nabla u)+\beta u =& f \quad \text{in $\Omega$},
    \\
    \nonumber
    u =& 0 \quad \text{on $\partial \Omega$}.
\end{align}
To define the {\sf HDG-P0} scheme, we shall reformulate equation \eqref{ellipOrig} as the following first-order system by introducing the flux $\underline{\sigma}:=-\alpha\nabla u$ as a new variable:
\begin{subequations}
\label{ellipModel}
    \begin{alignat}{2}
        \alpha^{-1}\underline{\sigma} + \nabla u = &\; 0 &&\quad \text{ in }\Omega,\\
        \nabla\cdot\underline{\sigma} + \beta u = &\; f &&\quad \text{ in }\Omega,\\
        u = &\; 0 &&\quad \text{ on }\partial\Omega.
    \end{alignat}
\end{subequations}

\subsection{The {\sf HDG-P0} scheme}
With the mesh and finite element spaces given in Section \ref{sec:notation}, we are ready to present the {\sf HDG-P0} scheme 
for the system \eqref{ellipModel}.
A defining feature of hybrid finite element schemes is the use of finite element spaces to approximate the primal variable $u$ both in the mesh $\Th$ and on the mesh skeleton $\Eh$.
We use the (discontinuous) piecewise linear finite element space $V_h$ to approximate $u$ in the mesh $\Th$, and 
the (discontinuous) piecewise constant space $M_h^0$ to approximate $u$ on the  mesh skeleton $\Eh$, where the homogeneous Dirichlet boundary condition has been built-in to the space $M_h^0$. For the flux variable $\underline{\sigma}$, we approximate it using the vectorial piecewise constant function space $\underline{W}_h$.
Using these finite element spaces and the numerical integration rules introduced in Section \ref{sec:notation}, 
the weak formulation of the {\sf HDG-P0} scheme is now given as follows:
find $(\underline{\sigma}_h,\; u_h,\; \widehat{u}_h)\in\underline{W}_h \times V_{h} \times M_{h}^0$ such that:
\begin{subequations}
\label{ellipWeak0}
\begin{align}
\label{wkA}
    \sum_{K\in\Th}\left(Q_K^0(\alpha^{-1}_h\underline{\sigma}_h\cdot\underline{r}_h) + 
    Q_{\partial K}^0(\widehat{u}_h\, \underline{r}_h\cdot\underline{n}_K)\right)
    &= 0,
    \\
    \label{wk1}
        \sum_{K\in\Th}\left(Q_{\partial K}^0(\tau_K(u_h-\widehat{u}_h)v_h)
   +Q_K^1(\beta u_h v_h)\right)
    &=     \sum_{K\in\Th}Q_K^1(f v_h),\\
    \label{wkB}
\sum_{K\in\Th}
  \left(  Q_{\partial K}^0(\underline{\sigma}_h\cdot\underline{n}_K\widehat{v}_h)
    +
    Q_{\partial K}^0(\tau_K(u_h-\widehat{u}_h)\widehat{v}_h)
    \right)
    &= 0,
\end{align}
\end{subequations}
for all $(\underline{r}_h,\; v_h,\; \widehat{v}_h)\in\underline{W}_h \times V_h \times M_{h}^0$, where ${\alpha}_h^{-1}\in W_h$ is the $L_2$-projection of $\alpha^{-1}$ onto the piecewise constant space $W_h$, 
and $\tau_K:=\frac{\alpha_h}{h_K}$ is the stabilization parameter, with $\alpha_h:=(\alpha_h^{-1})^{-1}$.
Taking test functions $\underline{r}_h=\underline{\sigma}_h$, 
$v_h=u_h$, and
$\widehat{v}_h=-\widehat{u}_h$ in \eqref{ellipWeak0}, and adding, we have the following energy identity:
\begin{align*}
 \sum_{K\in\Th}\left(Q_K^0(\alpha^{-1}_h|\underline{\sigma}_h|^2)+
 Q_{\partial K}^0(\tau_K(u_h-\widehat{u}_h)^2)
 +Q_{K}^1(\beta u_h^2)
 \right) = 
  \sum_{K\in\Th} Q_K^1(fu_h).
\end{align*}
\red{We note that the quadratures used in \eqref{ellipWeak0}
guarantees the scheme is locally mass-conservative with the numerical flux defined as 
\[
    \widehat{\underline{\sigma}}_h\cdot\underline{n}_K:= \underline{\sigma}_h\cdot\underline{n}_K+\tau_K\varPi_0(u_h-\widehat{u}_h),
\]
where $\varPi_0$ is the $L_2$-projection onto the skeleton space $M_h$, since the equation \eqref{wkB}
is equivalent to 
\[    \left\langle
\widehat{\underline{\sigma}}_h\cdot\underline{n}_K, \widehat{v}_h\right\rangle_{\partial\Th}
    = 0, \quad\forall \widehat{v}_h\in M_h^0.
\]
See more discussion in Lemma \ref{lemma:ellipEq} below.
}

\subsection{An a priori error analysis for {\sf HDG-P0}}

\red{To simplify the a priori error analysis in this subsection, we assume the exact solution $(\underline{\sigma}, u) \in \underline{H}^1(\Omega)\times (H^2(\Omega)\cap H^1_0(\Omega))$ for the model problem \eqref{ellipModel}, together with the following full elliptic regularity result:
\begin{align}
\label{reg}
    \alpha_0^{-1/2} \|\underline{\sigma}\|_1 
    + (\alpha_1^{1/2}+\beta_1^{1/2}) \|u\|_2 
    \lesssim c_{reg} \|f\|_0,
\end{align}
which holds when the domain $\Omega$ is convex.
}

To perform an a priori error analysis of the scheme~\eqref{ellipWeak0} with numerical  integration, we follow the standard convention \cite{ciarlet2002finite} by assuming the following stronger regularity for $\beta$ and $f$:
\[
\beta\in W^{1,\infty}(\bar\Omega), \quad
f\in W^{1,\infty}(\bar\Omega).
\]
We use the following approximation results of the quadrature rule $Q_K^1$, which is adapted from Ciarlet
\cite[Theorem 4.1.4--4.1.5]{ciarlet2002finite}.
\begin{lemma}
    \label{lem:numIntErr}
    For any $K\in\mathcal{T}_h$, $f, u\in W^{1, \infty}(\bar K)$, and $v\in\poly^1(K)$, we have:
    \begin{subequations}
    \begin{align}
        \label{numIntErr1}
        |(f,\; v)_K - Q_K^1(fv)| &\lesssim 
        h^{1+d/2}\|f\|_{1,\infty,K}\|v\|_{1,K},
        \\
        \label{numINtErr2}
        |(fu,\; v)_K - Q_K^1(fuv)| &\lesssim
        h\|f\|_{1,\infty,K}\|u\|_{1,K}\|v\|_{0,K},
    \end{align}
    \end{subequations}
\end{lemma}
We will compare the solution to \eqref{ellipWeak0}
with the solution to a similar scheme with exact integration.
To this end, let  $(\underline{\sigma}_h^1,\; u_h^1,\; \widehat{u}_h^1)\in\underline{W}_h \times V_{h} \times M_{h}^0$ be the solution to the following system:
\begin{subequations}
\label{ellipWeak}
\begin{align}
\label{EE0}
    \sum_{K\in\Th}\left(Q_K^0(\alpha^{-1}_h\underline{\sigma}_h^1\cdot\underline{r}_h) + 
    Q_{\partial K}^0(\widehat{u}_h^1\, \underline{r}_h\cdot\underline{n}_K)\right)
    &= 0,
    \\
    \label{wk2}
        \sum_{K\in\Th}\left(Q_{\partial K}^0(\tau_K(u_h^1-\widehat{u}_h^1)v_h)
   +(\beta u_h^1, v_h)_K\right)
    &=     \sum_{K\in\Th}(f, v_h)_K,\\
\label{EE1}
    \sum_{K\in\Th}
\left(Q_{\partial K}^0(\underline{\sigma}_h^1\cdot\underline{n}_K\widehat{v}_h)
    +
    Q_{\partial K}^0(\tau_K(u_h^1-\widehat{u}_h^1)\widehat{v}_h)
    \right)
    &= 0,
\end{align}
\end{subequations}
for all $(\underline{r}_h,\; v_h,\; \widehat{v}_h)\in\underline{W}_h \times V_h \times M_{h}^0$.
Note that the only difference between {\sf HDG-P0} \eqref{ellipWeak0} and 
the above system \eqref{ellipWeak} is in equations \eqref{wk1} and \eqref{wk2}, where the former uses $Q_K^1$ as the  volumetric numerical integration rule and the latter uses exact volumetric integration.
The following result shows that the scheme \eqref{ellipWeak}
 is precisely the lowest-order (mixed) HDG scheme with projected jumps (and exact integration) analyzed in \cite{qiu2018hdg,qiu2016hdg} for linear elasticity and convection-diffusion problems, hence its a priori error estimates can be readily adapted from \cite{qiu2018hdg, qiu2016hdg}.
\begin{lemma}
\label{lemma:ellipEq}
    Let $(\underline{\sigma}_h^1, u_h^1, \widehat{u}_h^1)
    \in \underline{W}_h \times V_h \times M_{h}^0
    $
    be the solution to \eqref{ellipWeak}, then 
    $(\underline{\sigma}_h^1, u_h^1, \widehat{u}_h^1) 
    $ satisfies
 \begin{subequations}
\label{ellipWeakX}
\begin{align}
\label{ex1}
        (\alpha^{-1}\underline{\sigma}_h^1,\,\underline{r}_h)_{\Th}
    -
    ({u}_h^1,\,\nabla\cdot\underline{r}_h)_{\Th}
    +\langle\widehat{u}_h^1,\, \underline{r}_h\cdot\underline{n}_K\rangle_{\partial\Th}
    &= 0,
    \\
\label{ex2}
    (\nabla\cdot\underline{\sigma}_h^1, v_h)_{\Th}
    +\langle\tau_K\varPi_0(u_h^1-\widehat{u}_h^1), v_h\rangle_{\partial\Th}
   +(\beta u_h^1, v_h)_{\Th}
    &=(f, v_h)_{\Th},\\
    \label{ex3}
    \left\langle\underline{\sigma}_h^1\cdot\underline{n}_K+\tau_K\varPi_0(u_h^1-\widehat{u}_h^1), \widehat{v}_h\right\rangle_{\partial\Th}
    &= 0,
\end{align}
for all
 $(\underline{r}_h,\; v_h,\; \widehat{v}_h)\in\underline{W}_h \times V_h \times M_{h}^0$,
 where $\varPi_0$ is the $L_2$-projection onto piecewise constants on the mesh skeleton.
\end{subequations}
\end{lemma}
\begin{proof}
    Using the fact that $\underline{\sigma}^1_h, \underline{r}_h\in \underline{W_h}$ are constant functions on each element $K$, we have 
    \[
    (\alpha^{-1}\underline{\sigma}_h^1,\,\underline{r}_h)_K
    = \underline{\sigma}_h^1(m_K^0)\cdot \underline{r}_h(m_K^0)\int_K\alpha^{-1}\,\mathrm{dx}
    = \underline{\sigma}_h^1(m_K^0)\cdot \underline{r}_h(m_K^0) \alpha_h^{-1} |K|
    =Q_K^0(\alpha_h^{-1}\underline{\sigma}_h^1\cdot \underline{r}_h).
    \]
    Similarly, there holds 
    \begin{align*}
    \langle\widehat{u}_h^1,\, \underline{r}_h\cdot\underline{n}_K\rangle_{\partial K}
    = Q_{\partial K}^0(\widehat{u}_h^1\,\underline{r}_h\cdot\underline{n}_K).
    \end{align*}
    Using the definition of the projection $\varPi_0$ and the
    fact that the restriction of  $\tau_K$ 
    on each facet $F^i\subset \partial K$ is a constant (denoted as $\tau_K^i$), we have
    \[
\langle
\tau_K\varPi_0(u_h^1-\widehat{u}_h^1), {v}_h\rangle_{\partial K}
    = 
    \sum_{i=1}^{d+1} 
    \tau_K^i |F^i|\left(u_h^1(m_{F^i})-\widehat{u}_h^1(m_{F^i})\right)
    v_h(m_{F^i})
    = Q_{\partial K}^0(\tau_K
    \left(u_h^1-\widehat{u}_h^1\right)
    v_h).
    \]
Combining these identities with the fact that $\nabla\cdot \underline r_h=\nabla\cdot \underline \sigma_h^1\equiv 0$, 
we conclude that the systems~\eqref{ellipWeak}
and \eqref{ellipWeakX} are exactly the same.
\end{proof}
We have the following results on the a priori error estimates of the scheme \eqref{ellipWeak}, whose proof is omitted for simplicity. We refer to the works \cite{qiu2018hdg, qiu2016hdg} for a similar analysis.
\begin{corollary}
Let $(\underline\sigma, u)$ be the exact solution to the model problem \eqref{ellipModel}.
 Let $(\underline{\sigma}_h^1, u_h^1, \widehat{u}_h^1)
    \in \underline{W}_h \times V_h \times M_{h}^0
    $
    be the solution to \eqref{ellipWeak}. Then there holds
\begin{subequations}
\begin{align}
\label{ener-ex1}
        \|{\alpha}^{-1/2}(\underline{\sigma} - \underline{\sigma}_h^1)\|_{0}
        +\|\beta^{1/2}(u-u_h^1)\|_{0}
        + \|\tau_K^{1/2}\varPi_0(u_h^1-\widehat{u}_h^1)\|_{0,\partial \Th}
                &\lesssim
        h\,\Theta,\\
\label{ener-ex2}
\|u-u_h^1\|_0\;\lesssim \;       \|\nabla (u - u_h^1)\|_{0}
+\|{h_K^{-1/2}}(u_h^1-\widehat{u}_h^1)\|_{0,\partial \Th}
                &\lesssim
        \alpha_0^{-1/2}h\,\Theta,
    \end{align}
where 
\begin{align}
\label{theta1}
\Theta:=\alpha_0^{-1/2}|\underline \sigma|_{1} 
        + (\alpha_1^{1/2}+\beta_1^{1/2}h)|u|_2,
\end{align}
and
$\|\phi\|_{0,\partial\Th}:=\sqrt{\langle\phi, \phi\rangle_{\partial\Th}}$.
\red{Moreover, with the full elliptic regularity assumed in \eqref{reg},} we have the following optimal $L_2$-convergence for $u_h^1$:
\begin{align}
\label{dual-ex}
    \|u-u_h^1\|_0 \lesssim c_{reg} h^2\Theta.
\end{align}
    \end{subequations}
\end{corollary}

Combining the results in this subsection, 
we are now ready to give the error estimates for the solution to the {\sf HDG-P0} scheme \eqref{ellipWeak0}. 
\begin{theorem}
    \label{theo:ellipErr}
  Let $(\underline\sigma, u)$ be the exact solution to the model problem \eqref{ellipModel}.
 Let $(\underline{\sigma}_h, u_h, \widehat{u}_h)
    \in \underline{W}_h \times V_h \times M_{h}^0
    $
    be the solution to \eqref{ellipWeak0}. Then there holds
    \begin{subequations}
\begin{align}
\label{ener-q1}
        \|{\alpha}^{-1/2}(\underline{\sigma} - \underline{\sigma}_h)\|_{0}
        + \|\tau_K^{1/2}\varPi_0(u_h-\widehat{u}_h)\|_{0,\partial \Th}&\le h (\Theta+\Xi),
        \\
\label{ener-q2}
\|u-u_h\|_0
\lesssim \|\nabla (u - u_h)\|_{0}
+\|{h_K^{-1/2}}(u_h-\widehat{u}_h)\|_{0,\partial \Th}
                &\lesssim
        \alpha_0^{-1/2}h(\Theta+\Xi),
\end{align}
where $\Theta$ is given in \eqref{theta1} and 
\[
\Xi:=\alpha_0^{-1/2}\|\beta\|_{1,\infty}\left(\|u\|_0+\alpha_0^{-1/2}h\Theta\right)
        + \alpha_0^{-1/2} h^{d/2}\|f\|_{1,\infty}.
\]
Moreover, assuming full elliptic regularity \eqref{reg}, we have the following optimal $L_2$-convergence for $u_h$:
\begin{align}
\label{dual-q}
    \|u-u_h\|_0 \lesssim h^2 \Psi,
    \end{align}
where 
\[
\Psi:= c_{reg}(\Theta+\Xi)+\frac{c_{reg}\alpha_0^{-1/2}\beta_1}{\alpha_1^{1/2}+\beta_1^{1/2}h}\Xi
+\|\beta\|_{1,\infty}\alpha_0^{-3/2}(\Theta+\Xi)
+h^{d/2-1}\alpha_0^{-1}\|f-\beta u\|_{1,\infty}.
\]
    \end{subequations}
\end{theorem}
\begin{proof}
The proof of \eqref{ener-q1}--\eqref{ener-q2} follows from a standard energy argument, while that of \eqref{dual-q} follows from a 
duality argument. 
Here we prove the energy estimates \eqref{ener-q1}--\eqref{ener-q2}, and postpone the (slightly more technical) proof of \eqref{dual-q} to the Appendix below. 

We denote the difference of the solution 
$(\underline{\sigma}_h, u_h, \widehat{u}_h)$
to the system \eqref{ellipWeak0}
and the solution 
$(\underline{\sigma}_h^1, u_h^1, \widehat{u}_h^1)$
to the system \eqref{ellipWeak} as
\begin{align*}
        \underline{e}_{\sigma}:= \underline{\sigma}_h - \underline{\sigma}_h^1,
        \quad
        e_{u}:= u_h - u_h^1,
        \quad
        e_{\widehat{u}}:= \widehat{u}_h - \widehat{u}_h^1.
\end{align*}
Subtracting equations \eqref{ellipWeak0} from \eqref{ellipWeak} and using Lemma \ref{lemma:ellipEq}, we get the following error equation: 
    \begin{subequations}
        \label{ellipErrEqns}
        \begin{align}
            \label{ellipErrEqn1}
                (\alpha^{-1}\underline{e}_\sigma,\,\underline{r}_h)_{\Th}
    -
    (e_{u},\,\nabla\cdot\underline{r}_h)_{\Th}
    +\langle e_{\widehat{u}},\, \underline{r}_h\cdot\underline{n}_K\rangle_{\partial\Th}
    &= 0,
    \\
            \label{ellipErrEqn2}
    (\nabla\cdot\underline{e}_{\sigma}, v_h)_{\Th}
    +\langle\tau_K\varPi_0(e_u-e_{\widehat{u}}), v_h\rangle_{\partial\Th}
   +\sum_{K\in\Th}Q_K^1(\beta e_uv_h)
    &=T_1(v_h)-T_2(v_h),\\
                \label{ellipErrEqn3}
\left\langle\underline{e}_\sigma\cdot\underline{n}_K+\tau_K\varPi_0(e_u-e_{\widehat{u}}), \widehat{v}_h\right\rangle_{\partial\Th}
    &= 0,
        \end{align}
    \end{subequations}
    for all  $(\underline{r}_h,\; v_h,\; \widehat{v}_h)\in\underline{W}_h \times V_h \times M_{h}^0$
    where 
    \begin{align*}
    T_1(v_h):=&\;(\beta u_h^1, v_h)_{\Th}-\sum_{K\in\Th}Q_K^1(\beta u_h^1 v_h),\\
    T_2(v_h):=&\; (f, v_h)_{\Th}-\sum_{K\in\Th}Q_K^1(f v_h).
    \end{align*}
    Using Lemma \ref{lem:numIntErr}, we have 
\begin{align}
\label{xx1}
T_1(v_h)-T_2(v_h)\lesssim &\; \sum_{K\in \Th}\left(h\|\beta\|_{1,\infty,K}\|u_h^1\|_{0,K}
+h^{1+d/2}\|f\|_{1,\infty,K}
\right)\|v_h\|_{1,K}
\end{align}

\red{
Taking test function $\underline r_h:=\nabla e_u$ in equation (11a), by integration by parts, reordering terms, and applying the Cauchy-Schwarz inequality and inverse inequality, we get 
\begin{align*}
    (\nabla e_u, \nabla e_u)_{\Th} =&\; -(\alpha^{-1}\underline{e}_{\sigma}, \nabla e_u)
+\langle \varPi_0(e_u-e_{\widehat{u}}),\, \nabla e_u \cdot\underline{n}_K\rangle_{\partial\Th}\\
\le &\; 
\alpha_0^{-1/2}\left(\|\alpha^{-1/2}\underline{e}_{\sigma}\|_0 \; \|\nabla e_u\|_0
+h_K^{1/2}\|\tau_K^{1/2}\varPi_0(e_u-e_{\widehat{u}})\|_{0,\partial\Th}
 \; \|\nabla e_u\|_{0,\partial\Th}\right) \\
\lesssim &\;
\alpha_0^{-1/2}\left(\|\alpha^{-1/2}\underline{e}_{\sigma}\|_0
+\|\tau_K^{1/2}\varPi_0(e_u-e_{\widehat{u}})\|_{0,\partial\Th}\right)
\|\nabla e_u\|_0.
\end{align*}
}
Hence, 
\[
\|\nabla e_u\|_0\lesssim \alpha_0^{-1/2}\left(\|\alpha^{-1/2}\underline{e}_{\sigma}\|_0
+\|\tau_K^{1/2}\varPi_0(e_u-e_{\widehat{u}})\|_{0,\partial\Th}\right).
\]
Moreover, by the definition of $\varPi_0$, the trace-inverse inequality and Poincar\'e inequality, we have
\begin{align*}
\|h_K^{-1/2}(e_u-e_{\widehat{u}})\|_{0,\partial\Th}
\le&\; 
\|h_K^{-1/2}(e_u-\varPi_0 e_u)\|_{0,\partial\Th}
+\|h_K^{-1/2}\varPi_0(e_u-e_{\widehat{u}})\|_{0,\partial\Th}\\
\le &\;
\|h_K^{-1/2}(e_u-P_0{e_u})\|_{0,\partial\Th}
+\|h_K^{-1/2}\varPi_0(e_u-e_{\widehat{u}})\|_{0,\partial\Th}\\
\lesssim &\;
\|\nabla e_u\|_{0}
+\|h_K^{-1/2}\varPi_0(e_u-e_{\widehat{u}})\|_{0,\partial\Th}\\
\lesssim &\;
\alpha_0^{-1/2}\left(\|\alpha^{-1/2}\underline{e}_{\sigma}\|_0
+\|\tau_K^{1/2}\varPi_0(e_u-e_{\widehat{u}})\|_{0,\partial\Th}\right),
\end{align*}
where $P_0(e_u)$ is the $L_2$-projection of $e_u$ onto the piecewise constant space $W_h$.
Combining the above two inequalities with the discrete Poincar\'e
inequality for piecewise $H^1$ functions \cite[Remark 1.1]{brenner2003poincare}, we get 
\begin{align}
\label{xx2}
\|e_u\|_0 \lesssim 
\|\nabla e_u\|_0+
\|h_K^{-1/2}(e_u-e_{\widehat{u}})\|_{0,\partial\Th}
\lesssim 
\;
\alpha_0^{-1/2}\left(\|\alpha^{-1/2}\underline{e}_{\sigma}\|_0
+\|\tau_K^{1/2}\varPi_0(e_u-e_{\widehat{u}})\|_{0,\partial\Th}\right).
\end{align}
Now taking test functions $\underline{r}_h=\underline e_{\sigma}, 
v_h=e_u, \widehat{v}_h = -e_{\widehat{u}}$ in \eqref{ellipErrEqns} and adding, we get 
\begin{align*}
&\|\alpha^{-1/2}\underline{e}_{\sigma}\|_0^2
+\|\tau_K^{1/2}\varPi_0(e_u-e_{\widehat{u}})\|_{0,\partial\Th}^2
+\|\beta^{1/2}e_u\|_{0,h}
 = \;T_1(e_u)-T_2(e_u)\\
 &\lesssim 
 h \left(\alpha_0^{-1/2}\|\beta\|_{1,\infty}\|u_h^1\|_{0}
+ \alpha_0^{-1/2}h^{d/2}\|f\|_{1,\infty}
\right)\left(\|\alpha^{-1/2}\underline{e}_{\sigma}\|_0
+\|\tau_K^{1/2}\varPi_0(e_u-e_{\widehat{u}})\|_{0,\partial\Th}\right).
\end{align*}
where $\|\phi\|_{0,h}:=\sqrt{\sum_{K\in\Th}Q_K^1(\phi^2)}$. Combining the above inequality with
the triangle inequality
\[
\|u_h^1\|_0 \le \|u-u_h^1\|_0+\|u\|_0,
\]
and 
the estimates \eqref{ener-ex1}--\eqref{ener-ex2}  we get the energy estimate \eqref{ener-q1}.
The estimate \eqref{ener-q2} is a direct consequence of 
\eqref{xx2}, \eqref{ener-ex2}, and
\eqref{ener-q1}.
\end{proof}

\subsection{Equivalence to a nonconforming discretization}
\label{subsec:ellEq}
We define an interpolation operator 
$\Pi^{\CR}_h:M_{h}\rightarrow V_{h}^{\CR}$
from the piecewise constant facet space
$M_h$ to the nonconforming (piecewise linear) CR space $V_h^{CR}$:
\begin{align}
    \label{picr}
    \Pi^{\CR}_h\widehat{v}_h(m_F) = \widehat{v}_h(m_F),    \quad \forall\widehat{v}_h\in M_h, F\in\Eh.
\end{align}
It is clear that $\Pi^{CR}_h$ is an isomorphic map between these two spaces, which share the same set of DOFs (one DOF per facet).
The following result builds a link between the {\sf HDG-P0} scheme \eqref{ellipWeak0} and a (slightly modified) nonconforming CR discretization.
\begin{theorem}
\label{theo:HDGP0=CR-ellip}
Let $u_h^{CR}\in V_h^{CR,0}$ be the solution to the following nonconforming scheme 
\begin{align}
    \label{cr}
    \sum_{K\in\Th}\left(Q_K^0(\alpha_h\nabla u_h^{CR}\cdot\nabla v_h^{CR})
    +Q_K^1(\gamma_h\beta u_h^{CR}v_h^{CR})
    \right)
    = \sum_{K\in\Th}
    Q_K^1(\gamma_h fv_h^{CR}), \quad\forall v_h^{CR}\in V_h^{CR,0},
\end{align}
where $\gamma_h:= \frac{\alpha_h}{\alpha_h+ \frac{h_K^2}{d+1}\beta}$.
Then the solution $(\underline\sigma_h, u_h, \widehat u_h)$ to the {\sf HDG-P0} scheme \eqref{ellipWeak0}
satisfies
\begin{subequations}
\label{crc}
\begin{align}
\label{cr0}
    \underline \sigma_h =&\; -\alpha_h\nabla\Pi_h^{CR} \widehat u_h, \\
\label{cr1}
u_h(m_K^i) =&\; 
\gamma_h(m_K^i) \left(\widehat u_h(m_K^i)
+ \frac{(h_K^i)^2}{(d+1)\alpha_h}f(m_K^i)\right),\quad
\forall i\in\{1, \dots, d+1\}, \forall K\in \Th,\\
\label{cr2}
\Pi_h^{CR} \widehat u_h =&\; u_h^{CR},
\end{align}
where recall that $m_K^i$ is the barycenter of the $i$-th facet $F^i$ of the element $K$, and $h_K^i=\frac{|K|}{|F^i|}$.
\end{subequations}
\end{theorem}

\begin{proof}
By \eqref{wkA} and the definition of $\Pi_h^{CR}$ in \eqref{picr}, we have, for any $\underline r_h\in \poly^0(K)$,  
\begin{align*}
    Q_K^0(\alpha^{-1}_h\underline{\sigma}_h\cdot\underline{r}_h) =&\;-
    Q_{\partial K}^0(\Pi_h^{CR}\widehat{u}_h\, \underline{r}_h\cdot\underline{n}_K)\\
    = &\;-\langle \Pi_h^{CR}\widehat{u}_h, 
    \underline{r}_h\cdot\underline{n}_K\rangle_{\partial K}\\
    = &\;-(\nabla \Pi_h^{CR}\widehat{u}_h, 
    \underline{r}_h)_K
    -\underbrace{( \Pi_h^{CR}\widehat{u}_h, 
    \nabla\cdot \underline{r}_h)_K}_{\equiv 0}
\end{align*}
Hence, 
\[
(\alpha_h^{-1}\underline\sigma_h+\nabla \Pi_h^{CR}\widehat{u}_h, 
    \underline{r}_h)_{\Th} = 0,
\]
which implies the identity \eqref{cr0} thanks to the fact that $\alpha_h|_K\in \poly^0(K)$ and 
$\underline \sigma_h, \nabla\Pi_h^{CR}\widehat{u}_h\in
[\poly^0(K)]^d$. 

Next, for $i\in \{1, \dots, d+1\}$, taking test function $v_h$ in \eqref{wk1} to be supported only on
an element $K$, whose values is $1$ on $m_K^i$ and zero
on $m_K^j$ for $j\not=i$, we get
\begin{align*}
    \tau_K^i(u_h(m_K^i)- \widehat{u}_h(m_K^i))|F^i|
    +\frac{|K|}{d+1}\beta(m_K^i)u_h(m_K^i)
    = \frac{|K|}{d+1}f(m_K^i).
\end{align*}
Solving for the value $u_h(m_K^i)$ and using the definition of $\tau_K$ and $h_K$, i.e.,  $\tau_K^i=\frac{\alpha_h}{h_K}$ and $h_K^i=\frac{|K|}{|F^i|}$,
we immediately get the equality \eqref{cr1}.

Finally, let us prove the identity \eqref{cr2}. By the definition of $\Pi_h^{CR}$ and 
\eqref{cr0}--\eqref{cr1}, we have 
\begin{align*}
Q_{\partial K}^0(\underline{\sigma}_h\cdot\underline{n}_K\widehat{v}_h)
    =&\; 
    \langle\underline{\sigma}_h\cdot\underline{n}_K, \Pi_h^{CR}\widehat{v}_h\rangle_{\partial K}=\;
    (\underline{\sigma}_h, \nabla\Pi_h^{CR}\widehat{v}_h)_K=
   -Q_K^0(\alpha_h\nabla{\Pi_h^{CR}\widehat u_h} \nabla\Pi_h^{CR}\widehat{v}_h)_K,
\end{align*}
and 
\begin{align*}
    Q_{\partial K}^0(\tau_K(u_h-\widehat{u}_h)\widehat{v}_h)
    =&\; 
    \sum_{i=1}^{d+1}|F^i|\tau_K^i(u_h(m_K^i)-\widehat u_h(m_K^i))\widehat v_h(m_K^i)\\
    =&\;
    \sum_{i=1}^{d+1}\frac{|K|}{d+1}\gamma_h(m_K^i)\left(-\beta(m_K^i)\widehat u_h(m_K^i)
    +f(m_K^i)\right)
    \widehat v_h(m_K^i)\\
    =&\;
    -Q_K^1(\gamma_h\beta\Pi_h^{CR}\widehat u_h\Pi_h^{CR}\widehat v_h)
    +Q_K^1(\gamma_hf\Pi_h^{CR}\widehat v_h),
\end{align*}
where we skipped 
the algebraic manipulation in the above second equality.
Combining the above two identities with \eqref{wkB}, we get 
\begin{align}
    \label{hdgX}
    \sum_{K\in\Th}\left(Q_K^0(\alpha_h\nabla\Pi_h^{CR}\widehat u_h\cdot\nabla \Pi_h^{CR}\widehat v_h)
    +Q_K^1(\gamma_h\beta \Pi_h^{CR}\widehat u_h
    \Pi_h^{CR}\widehat v_h)\right)
    = \sum_{K\in\Th}
    Q_K^1(\gamma_h f\Pi_h^{CR}\widehat v_h), 
    \end{align}
for all $\widehat v_h\in M_h^0$,
which implies the equivalence \eqref{cr2}. 
\end{proof}

\begin{remark}[Equivalence between {\sf HDG-P0} and CR discretizations]
Theorem \ref{theo:HDGP0=CR-ellip}
implies the equivalence between the {\sf HDG-P0} scheme \eqref{ellipWeak0} and the 
slightly modified CR discretization \eqref{cr} with numerical integration, where the modification is the introduction of the scaling parameter $\gamma_h$.
\red{This result is motivated by the equivalence of the lowest order Raviart-Thomas mixed method and the nonconforming CR method \cite{arnold1985mixed, marini1985inexpensive}.
}
Note that, when $\beta\not=0$, the {\sf HDG-P0} scheme \eqref{ellipWeak0} 
converges to the original CR discretization (with $\gamma_h$ set to 1 in \eqref{cr}) in the asymptotic limit as the mesh size $h$ approaches zero.
Moreover, for the pure diffusion case with $\beta=0$, the scheme \eqref{ellipWeak0} is always equivalent to the original CR discretization for any choice of non-zero stabilization parameter $\tau$ in the sense that the equality \eqref{cr2} always holds.
\end{remark}

\begin{remark}[On {\sf HDG-P0} solution procedure]
In the practical implementation of the {\sf HDG-P0}
scheme, we first 
locally eliminate $\underline \sigma_h$ and $u_h$ from 
\eqref{ellipWeak0} to arrive at the global (condensed) linear system \eqref{hdgX} for 
$\widehat u_h$. 
After solving for $\widehat u_h$ in the system \eqref{hdgX}, we then recover 
$\underline \sigma_h$ and $u_h$.
The next subsection focuses on the efficient linear system solver for \eqref{hdgX} via geometric multigrid. 
Due to the algebraic equivalence between the condensed HDG system \eqref{hdgX} and the (modified) nonconforming system \eqref{cr}, we can simply use the rich multigrid theory for nonconforming methods \cite{brenner1989optimal, braess1990multigrid, bramble1991analysis, chen1994analysis, brenner1999convergence, brenner2004convergence, duan2007generalized}
to precondition the condensed HDG system \eqref{hdgX}.
\end{remark}

\subsection{Multigrid algorithm}
\label{subsec:ellipMG}
In this subsection, we present the detailed multigrid algorithm for the condensed system \eqref{hdgX}, using the nonconforming multigrid theory of Brenner \cite{brenner2004convergence}.

\red{We consider a set of hierarchical meshes for the multigrid algorithm.
Let $\mathcal{T}_1$ be a conforming simplicial triangulation of $\Omega$ and let $\mathcal{T}_l$ be obtained by successive mesh refinements for $l=2, \dots, J$, with the final mesh $\mathcal{T}_J = \Th$.
We denote $h_l$ as the maximum mesh size of the triangulation $\mathcal{T}_l$. We assume the triangulation $\mathcal{T}_l$ is conforming, shape-regular, and quasi-uniform on each level, and the difference of the mesh size between two adjacent mesh levels is bounded, i.e. $h_{l} \lesssim h_{l+1}$. }
Let $\mathcal{E}_l$ be the mesh skeleton of $\mathcal{T}_l$. 
Denote $W_l$, $V_l$, $V_l^{CR}$
as the corresponding finite element spaces on the $l$-th level mesh $\mathcal{T}_l$, and $M_l$
as the corresponding piecewise constant finite element space on the $l$-th level mesh skeleton $\mathcal{E}_l$.
We define the following $L_2$-like inner product on the space $M_l$:
\[
(\widehat{u}_{l},\; \widehat{v}_{l})_{0,l}:=
\sum_{K\in \mathcal{T}_l}
Q_K^1(\Pi_l^{CR}\widehat{u}_{l}\,
\Pi_l^{CR}\widehat{v}_{l})
= 
\sum_{K\in \mathcal{T}_l}
\sum_{i=1}^{d+1}\frac{|K|}{d+1}\widehat{u}_{l}(m_K^i)\widehat{v}_{l}(m_K^i),
\]
with its induced norm as $\|\cdot\|_{0,l}$, 
where  $\Pi_l^{CR}$ is the interpolation operator~\eqref{picr} from $M_l$
to $V_l^{CR}$. 
We define ${A}_l:M_{l}^0\rightarrow M_{l}^0$ as the linear operator satisfying 
\begin{align}
\label{lol}
        ({A}_l\widehat u_l, \widehat v_l)_{0,l}
    := a_l(\widehat u_l, \widehat v_l),\quad \forall 
    \widehat u_l, \widehat v_l\in M_{l}^0,
\end{align}
where $a_l$ is the following bilinear form on the $l$-th level mesh:
\begin{align}
\label{lolX}
    a_l(\widehat u_l, \widehat v_l):=
        \sum_{K\in\mathcal{T}_l}\left(Q_K^0(\alpha_l\nabla\Pi_l^{CR}\widehat u_l\cdot\nabla \Pi_l^{CR}\widehat v_l)
    +Q_K^1(
    \beta_l
    \Pi_l^{CR}\widehat u_l\,
    \Pi_l^{CR}\widehat v_l)\right),\quad \forall 
    \widehat u_l, \widehat v_l\in M_{l}^0,
\end{align}
where $\alpha_l= (\alpha_l^{-1})^{-1}$ with $\alpha_l^{-1}\in W_l$ being the $L_2$-projection of $\alpha^{-1}$ onto the piecewise constant space $W_l$,
and $\beta_l\in V_h$ satisfies 
\[
\beta_l(m_K^i) = \gamma_l(m_K^i)\beta(m_K^i), \quad\forall i\in\{1,\dots, d+1\}, \forall K\in\T_{l}, 
\]
with $\gamma_l:=
 \frac{\alpha_l}{\alpha_l+ 
 \red{\frac{h_{K_l}^2}{d+1}\beta}}$
and 
\red{$h_{K_l}$}
being the mesh size of elements in $\mathcal{T}_l$.
Further denoting $f_l\in M_l^{0}$ such that
\[
(f_l, \widehat{v}_l)_{0,l} = 
\sum_{K\in\mathcal{T}_l} Q_K^1(\gamma_l f\Pi_l^{CR}\widehat v_l), \quad \forall \widehat v_l\in M_l^0,
\]
the operator form of the linear system \eqref{hdgX} is simply
to find $\widehat u_h\in M_J^0$  such that
\begin{align}
\label{op1}
    A_J\widehat u_h = f_J.
\end{align}
A multigrid algorithm for the above system \eqref{op1} needs two ingredients:
an intergrid transfer $I_{l-1}^{l}:M_{l-1}^0\rightarrow M_{l}^0$ operator that connects the spaces $M_{l-1}^{0}$ and $M_{l}^0$ on two consecutive mesh levels, and a smoothing operator $R_l: \; M_{l}^0\rightarrow M_l^0$
that takes care of high-frequency errors.
For the intergrid transfer operator, 
we use the following well-known nonconforming averaging operator 
\cite{brenner1989optimal, braess1990multigrid}: 
\begin{align}
    \label{intergridOpr}
    (I_{l-1}^l \widehat{v}_{l-1})(m_F):= \left\{
        \renewcommand{\arraystretch}{1.5} 
        \begin{array}{ll}
            (\Pi^{\CR}_{l-1}\widehat{v}_{l-1})(m_F),\;& \text{if $F\in \mathcal E_l^0$ lies inside $\mathcal{T}_{l-1}$},  \\[1ex]
            \frac{1}{2}\left((\Pi^{\CR}_{l-1}\widehat{v}_{l-1})^{+}(m_F) + (\Pi_{l-1}^{\CR}\widehat{v}_{l-1})^{-}(m_F)\right),\; & \text{if $F\in \mathcal E_l^0$ lies on $\mathcal{E}_{l-1}^o$},
        \end{array}
    \right.
\end{align}
where $(\Pi^{\CR}_{l-1}\widehat{v}_{l-1})^{+}$ and $(\Pi^{\CR}_{l-1}\widehat{v}_{l-1})^{-}$ are the values of $\Pi^{\CR}_{l-1}\widehat{v}_{l-1}$ on two adjacent elements $K^{+},\; K^{-}\in{\mathcal{T}_{l-1}}$ that share the facet $F$. 
We further denote the restriction operator $I_{l}^{l-1}:\; M^0_{l}\rightarrow M_{l-1}^0$ as the transpose of $I_{l-1}^{l}$ with respect to the inner product $(\cdot,\; \cdot)_{0, l}$:
\[
(I_l^{l-1} \widehat v, \widehat w)_{0, l-1}
= (\widehat v, I_{l-1}^{l} \widehat w)_{0, l}, \quad\forall  \widehat v\in M_l^0, \widehat w\in M_{l-1}^0.
\]
For the smoothing operator $R_l$, we simply take it to be the classical Jacobi or Gauss-Seidel smoother for the operator ${A}_l$.
The proof of multigrid convergence requires another operator $P_l^{l-1}:M_l^0\rightarrow M_{l-1}^0$, which is the transpose of $I_{l-1}^l$ with respect to the bilinear form $a_l$, i.e., 
\[
a_{l-1}(P_l^{l-1}\widehat u_{l}, \widehat v_{l-1})
=
a_{l}(\widehat u_{l}, I_{l-1}^l\widehat v_{l-1}), \quad\forall \widehat u_l\in M_l^0, \widehat v_{l-1}\in M_{l-1}^0.
\]
The operator $P_{l}^{l-1}$ is for multigrid analysis only, and never enters the actual multigrid algorithm.

We now follow \cite[Algorithm 2.1]{brenner2004convergence}
to present the classical symmetric $V$-cycle algorithm for the system $A_l \widehat u_l = f_l\in M_l^{0}$.
\begin{algorithm}[H]
\caption{The $V$-cycle algorithm for $A_l\widehat u_l = f_l$.}
\label{alg:ellipMG}
\begin{algorithmic}
\State The $l$-th level symmetric $V$-cycle algorithm produces $MG_{A_l}(l, f_l, \widehat u_l^0, m)$
as an approximation solution for
$A_l\widehat u_l = f_l$ with initial guess $\widehat u_l^0$, where $m$ denotes the number of pre-smoothing and post-smoothing steps.
\If {$l = 1$}
    \State
    ${MG}_{A_l}(l, \widehat{u}_l, f_l) = ({A}_l)^{-1} f_l$.
    \Else
\State Perform the following three steps:
\State (1) {\it pre-smoothing.}
For $j = 1,\dots,m$, compute $\widehat u_l^j$ by
    \[\widehat{u}_l^{j} = 
    \widehat{u}_l^{j-1} + {R}_l(f_l - {A}_l \widehat{u}_l^{j-1}).\]
\State (2) {\it Coarse grid correction.}
Let $\widehat r_{l-1}= I_l^{l-1}(f_l-A_l\widehat u_l^m)$ and compute $\widehat u_l^{m+1}$ by 
\[
\widehat{u}_l^{m+1} = 
\widehat{u}_l^{m}+I_{l-1}^l MG_{A_{l-1}}(l-1, \widehat r_{l-1}, 0 ,m). 
\]
\State (3) {\it Post-smoothing.} 
For $j=m+2,\dots, 2m+1$, compute $\widehat u_l^j$
by 
\[
\widehat{u}_l^{j} = 
    \widehat{u}_l^{j-1} + {R}_l^T(f_l - {A}_l \widehat{u}_l^{j-1}),
\]
\State where $R_l^T$ is the transpose of $R_l$ with respect to the inner product $(\cdot,\cdot)_{0,l}$.
\State We then define 
$MG_{A_l}(l, f_l,\widehat{u}_l^0, m) = \widehat u_l^{2m+1}$.\EndIf

\end{algorithmic}
\end{algorithm}


By the algebraic equivalence of the nonconforming system \eqref{cr} and the condensed HDG system~\eqref{hdgX}, we immediately have the following optimality result of the above multigrid algorithm
 from \cite{brenner2004convergence}.
 \red{We note that both the equivalence result in Theorem \ref{theo:HDGP0=CR-ellip} and the optimal multigrid theory in \cite{brenner2004convergence} do not require the full elliptic regularity assumption \eqref{reg}. Thus our proposed multigrid algorithm for the {\sf HDG-P0} scheme  is optimal in the low-regularity case, where $u \in H^{1+s}(\Omega)\cap H^1_0(\Omega)$ and $s\in(\frac{1}{2}, 1]$ is the regularity constant for the model problem \eqref{ellipModel}.
}
\begin{theorem}[Theorem 5.2 of \cite{brenner2004convergence}]
\label{thm:mg}
There exist positive mesh-independent constants $C$
and $m_*$ such that 
\[
|\!|\!|\mathbb{E}_{l,m}\widehat v|\!|\!|_{1, l}
\le 
\frac{C}{m^s}|\!|\!|\widehat v|\!|\!|_{1, l},\quad\forall \widehat v\in M_l^{0}, l\ge 1, m\ge m^*, 
\]
where 
$s\in(\frac12,1]$ is the regularity constant such that the solution $u$ to \eqref{ellipOrig}
satisfies 
\[
\|u\|_{1+s}\lesssim \|f\|_{-1+s}, 
\]
$|\!|\!|\cdot|\!|\!|_{1, l}$ is the mesh-dependent norm induced by the linear operator $A_l$ \eqref{lol}, i.e., 
$|\!|\!|\widehat v|\!|\!|_{1, l}:=\sqrt{(A_l\widehat v, \widehat v)_{0,l}}$, 
and $\mathbb{E}_{l,m}: M_l^0\rightarrow M_l^0$
is the operator relating the initial error and the final error of the multigrid $V$-cycle algorithm, i.e., 
\[
\mathbb{E}_{l,m}(\widehat u_l-\widehat u_l^0) := \widehat u_l - MG_{A_l}(l, f_l, \widehat u_l^0, m). 
\]
\end{theorem}
\begin{proof}
    This multigrid algorithm is algebraically equivalent to the $V$-cycle multigrid algorithm for the nonconforming CR scheme \eqref{cr} with diffusion coefficient $\alpha_l\in W_l$
    and reaction coefficient $\beta_l\in V_l$  in each level.
It is easy to verify that the assumptions (3.6)--(3.12) of \cite{brenner2004convergence}
are satisfied for the space $M_l^{0}$ and operators $I_{l-1}^l$ and $P_l^{l-1}$; see, e.g.,  \cite[Section 6]{brenner2004convergence} and \cite{brenner1999convergence}. 
Hence Theorem \ref{thm:mg} is simply a restatement of Theorem 5.2 of \cite{brenner2004convergence}.
\end{proof}
\section{{\sf HDG-P0} for the generalized Stokes equations}
\label{sec:MG4Stokes}
We follow the same procedures as in the previous section to present the {\sf HDG-P0} scheme for the generalized Stokes equations together with its optimal a priori error analysis, and the corresponding multigrid algorithms for the condensed linear system.

\subsection{The model problem}
We consider the following model problem with a homogeneous Dirichlet boundary condition:
\begin{subequations}
    \label{stokesOrig}
    \begin{alignat}{2}
        \beta \underline{u} - \nabla\cdot(\mu\nabla\underline{u}) + \nabla p 
        =& \;\underline{f}, \quad &&\text{in $\Omega$,}
        \\
        \nabla\cdot\underline{u} =& \; 0, \quad&&\text{in $\Omega$,}
        \\
        \underline{u}=& \;\underline{0}, \quad &&\text{on $\partial\Omega$,}
    \end{alignat}
\end{subequations}
where $\underline u$ is the velocity, $p$ is the pressure,
$\underline{f}\in\underline{L}^2(\Omega)$ is the source term, $\mu>0$ is the viscosity constant, and $\beta\ge0$
is the low-order term constant, which typically represents the inverse of time step size in a implicit-in-time discretization of the unsteady Stokes flow.

To present the {\sf HDG-P0} scheme, we introduce the tensor $\dunderline{L}:= - \mu\nabla\underline{u}$ as a new variable and rewrite the equations \eqref{stokesOrig} into a first-order system:
\begin{subequations}
    \label{stokesModel}
    \begin{alignat}{2}
        \mu^{-1}\dunderline{L} + \nabla\underline{u} =& \; 0, \quad&&\text{in $\Omega$,}
        \\
        \nabla\cdot\dunderline{L} + \beta\underline{u} 
        + \nabla p =& \; \underline{f}, \quad&&\text{in $\Omega$,}
        \\
        \nabla\cdot\underline{u} =& \; 0, \quad&&\text{in $\Omega$,}
        \\
        \underline{u} =& \; \underline{0}, \quad&&\text{on $\partial\Omega$.}
    \end{alignat}
\end{subequations}

\subsection{The {\sf HDG-P0} scheme}
The {\sf HDG-P0} discretization for the system 
\eqref{stokesModel} reads as follows:
find $(\dunderline{L}_h, \underline{u}_h, \widehat{\underline{u}}_h, p_h)\in\dunderline{W}_h\times\underline{V}_{h}\times\underline{M}_{h}^{0}\times W_{h}^{0}$ such that
\begin{subequations}
\label{stokesWeak0}
\begin{align}
    \label{stW0-1}
    \sum_{K\in\Th} \left(Q_K^0 (\mu^{-1}\dunderline{L}_h\cdot \dunderline{G}_h) 
    + Q_{\partial{K}}^0 ( \dunderline{G}_h\underline{n} \cdot \widehat{\underline{u}}_h ) \right) &= 0,
    \\
    \label{stW0-2}
    \sum_{K\in\Th}\left(Q_{\partial K}^0( \tau_K(\underline{u}_h - \widehat{\underline{u}}_h ) \cdot \underline{v}_h )
    + Q_K^1(\beta\underline{u}_h\cdot \underline{v}_h) \right) &
    = \sum_{K\in\Th} Q_K^1(\underline{f}\cdot \underline{v}_h),
    \\
    \label{stW0-3}
    \sum_{K\in\Th} Q_{\partial K}^0  (\widehat{\underline{u}}_h\cdot\underline{n}_K q_h)
    &= 0,
    \\
    \label{stW0-4}
    \sum_{K\in\Th} \left( Q_{\partial K}^0 ( (\dunderline{L}_h + p_h\dunderline{I})\underline{n}_K\cdot\widehat{\underline{v}}_h) 
    + Q_{\partial K}^0 ( \tau_K (\underline{u}_h - \widehat{\underline{u}}_h)\cdot \widehat{\underline{v}}_h) \right) &= 0,
\end{align}
\end{subequations}
for all $(\dunderline{G}_h, \underline{v}_h, \widehat{\underline{v}}_h, q_h)\in\dunderline{W}_h\times\underline{V}_{h}\times \underline{M}_{h}^{0}\times W_{h}^{0}$, where $\dunderline{I}$ is the unit diagonal matrix,
and $\tau_K = \frac{\mu}{h_K}$ is the stabilization parameter. Taking test functions $(\dunderline{G}_h, \underline{v}_h, \widehat{\underline{v}}_h, q_h) = (\dunderline{L}_h, \underline{u}_h, -\widehat{\underline{u}}_h, p_h)$ in \eqref{stokesWeak0} and adding, we have the following energy identity:
\[
    \sum_{K\in\Th}\left(
        Q_K^0 (\mu^{-1}|\dunderline{L}_h|^2)
        + Q_{\partial K}^0 (\tau_K (\underline{u}_h -\widehat{\underline{u}}_h)^2) 
        + Q_K^1 (\beta \underline{u}_h^2)
    \right)
    = \sum_{K\in\Th} Q_K^1(\underline{f}\cdot\underline{u}_h).
\]

\subsection{An a priori error analysis for {\sf HDG-P0}}
Similar to the reaction-diffusion case, 
\red{to simplify the a priori error analysis in this subsection, we assume $(\dunderline{L}, \underline{u}, p)\in \dunderline{H}^1(\Omega)\times (\underline{H}^2(\Omega)\cap \underline{H}^1_0(\Omega))\times (H^1(\Omega)\backslash\mathbb{R})$, together with the following elliptic regularity result:
\begin{align}
\label{stReg}
    \mu^{-1/2}\|\dunderline{L}\|_{1} 
    +\mu^{-1/2}\|p\|_1
    + (\mu^{1/2}+\beta^{1/2})\|\underline{u}\|_{2}
    \lesssim c_{reg} \|\underline{f}\|_{0},
\end{align}
which holds when the domain $\Omega$ is convex.}
Furthermore, we assume the following stronger regularity of the source term $\underline{f}$  to perform the a priori error analysis:
\[
    \underline{f}\in\underline{W}^{1,\infty}(\bar{\Omega}),
\]
and compare the solution to the {\sf HDG-P0} scheme \eqref{stokesWeak0} with the solution to a similar HDG scheme with exact integration. Let $(\dunderline{L}_h^1, \underline{u}_h^1, \widehat{\underline{u}}_h^1, p_h^1)\in\dunderline{W}_h\times\underline{V}_{h}\times\underline{M}_{h}^{0}\times W_{h}^{0}$ be the solution to the following scheme:
\begin{subequations}
\label{stokesWeak}
\begin{align}
    \sum_{K\in\Th} \left(Q_K^0 (\mu^{-1}\dunderline{L}_h^1\cdot \dunderline{G}_h) 
    + Q_{\partial{K}}^0 ( \dunderline{G}_h\underline{n} \cdot \widehat{\underline{u}}_h^1 ) \right) &= 0,
    \\
    \label{stW2}
    \sum_{K\in\Th}\left(Q_{\partial K}^0( \tau_K(\underline{u}_h^1 - \widehat{\underline{u}}_h^1 ) \cdot \underline{v}_h )
    + (\beta\underline{u}_h^1,\; \underline{v}_h)_K \right) &
    = \sum_{K\in\Th} (\underline{f},\; \underline{v}_h)_K,
    \\
    \sum_{K\in\Th} Q_{\partial K}^0  (\widehat{\underline{u}}_h^1\cdot\underline{n}_K q_h)
    &= 0,
    \\
    \sum_{K\in\Th} \left( Q_{\partial K}^0 ( (\dunderline{L}_h^1 + p_h^1\dunderline{I})\underline{n}_K\cdot\widehat{\underline{v}}_h) 
    + Q_{\partial K}^0 ( \tau_K (\underline{u}_h^1 - \widehat{\underline{u}}_h^1)\cdot \widehat{\underline{v}}_h) \right) &= 0,
\end{align}
\end{subequations}
for all $(\dunderline{G}_h, \underline{v}_h, \widehat{\underline{v}}_h, q_h)\in\dunderline{W}_h\times\underline{V}_{h}\times\underline{M}_{h}^{0}\times W_{h}^{0}$. 
Note that in two dimension ($d=2$), the above two schemes only differ by the right hand side term as 
$Q_K^1$ is exact for quadratic polynomial integration, while in three dimensions ($d=3$) the quadrature scheme \eqref{stokesWeak0} further introduce numerical integration errors in the left hand side volume term in \eqref{stW0-2}.
We use the following lemma to link the above system \eqref{stokesWeak} with the lowest order (mixed) HDG scheme with projected jumps (and exact integration) that has been analyzed in \cite{qiu2016superconvergent} for the incompressible Navier-Stokes equations. The proof procedure is similar to Lemma \ref{lemma:ellipEq} and is omitted here for simplicity.
\begin{lemma}
    \label{lemma:stEq}
    Let $(\dunderline{L}_h^1, \underline{u}_h^1, \widehat{\underline{u}}_h^1, p_h^1)\in\dunderline{W}_h\times\underline{V}_{h}\times\underline{M}_{h}^{0}\times W_{h}^{0}$ be the solution to \eqref{stokesWeak}, then $(\dunderline{L}_h^1, \underline{u}_h^1, \widehat{\underline{u}}_h^1, p_h^1)$ satisfies
    \begin{subequations}
    \label{stokesWeakX}
    \begin{align}
        \mu^{-1}(\dunderline{L}_h^1,\; \dunderline{G}_h)_{\mathcal{T}_h}  
        - (\underline{u}_h^1,\; \nabla\cdot\dunderline{G}_h)_{\mathcal{T}_h}
        + \langle\widehat{\underline{u}}_h^1,\; \dunderline{G}_h\underline{n}\rangle_{\partial\mathcal{T}_h} &= 0,
        \\
        \left( \nabla\cdot(\dunderline{L}_h^1 + p_h^1\dunderline{I}),\; \underline{v}_h \right)_{\mathcal{T}_h}
        + \left\langle \tau_K\underline{\varPi}_{0}(\underline{u}_h^1 - \widehat{\underline{u}}_h^1),\; \underline{v}_h \right\rangle_{\partial\mathcal{T}_h}
        + \beta(\underline{u}_h^1,\; \underline{v}_h)_{\mathcal{T}_h} &
        = (\underline{f},\; \underline{v}_l)_{\mathcal{T}_h}
        \\
        - (\underline{u}_h^1,\; \nabla q_h)_{\mathcal{T}_h}
        + \langle \widehat{\underline{u}}_h^1\cdot\underline{n}_K,\; q_h\rangle_{\partial\mathcal{T}_h} 
        &= 0,
        \\
        \left\langle (\dunderline{L}_h^1 + p_h^1\dunderline{I})\underline{n}_K + \tau_K \underline{\varPi}_{0}(\underline{u}_h^1 - \widehat{\underline{u}}_h^1),\; \widehat{\underline{v}}_h \right\rangle_{\partial\mathcal{T}_h} &= 0,
    \end{align}
    \end{subequations}
    for all $(\dunderline{G}_h, \underline{v}_h, \widehat{\underline{v}}_h, q_h)\in\dunderline{W}_h\times\underline{V}_{h}\times\underline{M}_{h}^{0}\times W_{h}^{0}$, where $\underline{\varPi}_0$ is the $L_2$-projection on to the piecewise constant vector on the mesh skeletons.
\end{lemma}

With Lemma \ref{lemma:stEq}, we adapt from \cite{qiu2016superconvergent} and get the following a priori error estimates for the HDG scheme \eqref{stokesWeak} with exact integration; 
see \cite{qiu2016superconvergent} for more details.

\begin{corollary}
\label{coro:Stokes}
    Let $(\dunderline{L}, \underline{u}, p)$ be the exact solution to the model problem \eqref{stokesModel}.
    Let $(\dunderline{L}_h^1, \underline{u}_h^1, \widehat{\underline{u}}_h^1, p_h^1)\in\dunderline{W}_h\times\underline{V}_{h}\times\underline{M}_{h}^{0}\times W_{h}^{0}$ be the solution to \eqref{stokesWeak}. Then there holds
    \begin{subequations}
    \begin{align}
    \label{stEner-ex1}
        \|{\mu}^{-1/2}(\dunderline{L} - \dunderline{L}_h^1)\|_{0}
        +\|\beta^{1/2}(\underline{u}-\underline{u}_h^1)\|_{0}
        + \|\tau_K^{1/2}\underline{\varPi}_0(\underline{u}_h^1-\widehat{\underline{u}}_h^1)\|_{0,\partial \Th}
        &\lesssim h\Theta_1,\\
    \label{stEner-ex2}
        \|\underline{u} - \underline{u}_h^1\|_{0}
        \lesssim
        \|\nabla (\underline{u} - \underline{u}_h^1)\|_{0}
        +\|{h_K^{-1/2}}(\underline{u}_h^1-\widehat{\underline{u}}_h^1)\|_{0,\partial \Th}
        &\lesssim
        \mu^{-1/2}h\Theta_1,
    \end{align}
    \end{subequations}
    where
    \begin{align}
    \label{theta2}
        \Theta_1:=
        \mu^{-1/2}|\dunderline{L}|_{1}
        + \mu^{-1/2}|p|_1
        + (\mu^{1/2}+\beta^{1/2}h) |\underline{u}|_2.
    \end{align}
    \red{Moreover, with the full elliptic regularity assumed in \eqref{stReg}, }
    we have the optimal $L_2$-convergence for $\underline{u}_h^1$:
    \begin{align}
    \label{st-dual-ex}
        \|\underline{u}-\underline{u}_h^1\|_0 \lesssim 
        h^2 (c_{reg} \Theta_2 + |\underline{u}|_2),
    \end{align}
    where
    \begin{align}
    \label{theta3}
        \Theta_2:=
        \mu^{-1/2}|\dunderline{L}|_{1}
        + \mu^{-1/2}|p|_1
        + \mu^{1/2} |\underline{u}|_2.
    \end{align}
\end{corollary}


Now we are ready to present the optimal convergence results of the HDG-P0 scheme \eqref{stokesWeak0}. We note the proof procedures are essentially the same as in Theorem \ref{theo:ellipErr} and the detailed proof is omitted here for simplicity.

\begin{theorem}
    \label{theo:stokesErr}
    Let $(\dunderline{L}, \underline{u}, p)$ be the exact solution to the model problem \eqref{stokesModel}.
    Let $(\dunderline{L}_h, \underline{u}_h, \widehat{\underline{u}}_h)
    \in \dunderline{W}_h \times \underline{V}_h \times \underline{M}_{h}^0 \times W_h^0 $
    be the solution to \eqref{stokesWeak0}. Then there holds
    \begin{subequations}
    \begin{align}
    \label{st-ener-q1}
            \|{\mu}^{-1/2}(\dunderline{L} - \dunderline{L}_h)\|_{0}
            + \|\tau_K^{1/2}\underline{\varPi}_0(\underline{u}_h-\widehat{\underline{u}}_h)\|_{0,\partial \Th}&\le 
            h (\Theta_1 + \Xi)
            \\
    \label{st-ener-q2}
            \|\underline{u} - \underline{u}_h\|_0
            \lesssim \|\nabla (\underline{u} - \underline{u}_h)\|_{0}
            +\|{h_K^{-1/2}}(\underline{u}_h-\widehat{\underline{u}}_h)\|_{0,\partial \Th}
            &\lesssim
            \mu^{-1/2}h(\Theta_1 + \Xi), 
    \end{align}
    where $\Theta_1$ is defined in \eqref{theta2} and
    \[
        \Xi:=
        \mu^{-1/2}\beta(\|\underline{u}\|_0 + \mu^{-1/2}h\Theta_1)
        + \mu^{-1/2}h^{d/2}\|\underline{f}\|_{1,\infty}.
    \] 
    Moreover, assuming the full elliptic regularity in \eqref{stReg}, we have the following optimal $L_2$-convergence of $\underline{u}_h$:
    \begin{align}
    \label{st-dual-q}
        \|\underline{u}-\underline{u}_h\|_0 
        \lesssim 
        h^2\Psi,
    \end{align}
    \end{subequations}
    where
    \[
        \Psi:= c_{reg}(\Theta_2+\Xi)
        +|\underline{u}|_2
        +\frac{c_{reg}\mu^{-1/2}\beta}{\mu^{1/2}+\beta^{1/2}h}\Xi
        +\|\beta\|_{1,\infty}\mu^{-3/2}(\Theta_1+\Xi)
        +h^{d/2-1}\mu^{-1}\|\underline{f}-\beta\underline{u}\|_{1,\infty},
    \]
    and $\Theta_2$ is defined in \eqref{theta3}.
\end{theorem}

\begin{remark}[On pressure robustness]
Corollary \ref{coro:Stokes} and
Theorem \ref{theo:stokesErr} implies the 
HDG velocity approximation errors depend on the regularity of the pressure, which is not pressure-robust in the sense of \cite{Volker17}.
A source modification technique 
based on discrete Helmholtz decomposition 
was used in \cite{Linke14a, Linke14b}
 for the nonconforming CR discretization of the Stokes problem ($\beta=0$) to recover pressure-robust 
velocity approximations.
When $\beta>0$, a further modification of the mass term using a BDM interpolation \cite{Linke16x} was needed to recover pressure-robustness. 
Due to the equivalence of 
the proposed {\sf HDG-P0} scheme \eqref{stokesWeak0} and a nonconforming CR  
discretization (see Theorem \ref{theo:HDGP0=CR-stokes} below), 
similar modification can be used for
the {\sf HDG-P0} scheme to render it pressure-robust by changing the two volume integration terms in \eqref{stW0-2} with the following:
\begin{align*}
    Q_K^1(\underline f\cdot\underline v_h)
    \longrightarrow &\;
    Q_K^1(\underline f\cdot\underline{\Pi}_h^{BDM}\underline v_h),\\
    Q_K^1(\beta\underline u_h\cdot\underline v_h)
    \longrightarrow &\;
    Q_K^1\left(\beta\underline{\Pi}_h^{BDM}\left(\underline{\Pi}_h^{CR}\widehat{\underline u}_h\right)\cdot\underline{\Pi}_h^{BDM}\underline v_h\right),
\end{align*}
where $\underline{\Pi}_h^{BDM}$ is the classical BDM interpolation into the space 
$V_h\cap H(\mathrm{div};\Omega)$, and 
$\underline{\Pi}_h^{CR}$ is the interpolation from $\underline{M}_h$ to $\underline{V}_h^{CR}$ given in \eqref{v-picr} below.
Optimal pressure-robust velocity error estimates can be obtained for this modified HDG scheme following the work \cite{Linke14b}.
\end{remark}

\subsection{Equivalence to CR discretization}
Similar to the reaction-diffusion case, a bijective interpolation operator $\underline{\Pi}_{h}^{CR}: \underline{M}_{h}\rightarrow\underline{V}_{h}^{CR}$ from the facet space $\underline{M}_h$ to the CR space $\underline{V}_h^{CR}$ is defined such that:
\begin{align}
\label{v-picr}
    \underline{\Pi}_h^{CR}\widehat{\underline{v}}_h(m_F)
    = \widehat{\underline{v}}_h(m_F), \quad \forall\widehat{\underline{v}}_h\in \underline{M}_h,\; F\in\Eh.
\end{align}
Then we have the following equivalence result between the {\sf HDG-P0} scheme \eqref{stokesWeak0} and a (slightly modified) CR discretization for the generalized Stokes problem:
\begin{theorem}
    \label{theo:HDGP0=CR-stokes}
    Let $(\underline{u}_h^{CR}, p_h^{CR})\in \underline{V}_h^{CR,0}\times W_h^0$ be the solution to the following nonconforming scheme:
    \begin{subequations}
    \label{stCr}
    \begin{align}
        \sum_{K\in\Th}\left(Q_K^0(\mu\nabla \underline{u}_h^{CR}\cdot\nabla \underline{v}_h^{CR})
        + Q_K^1(\gamma_h\beta \underline{u}_h^{CR} \cdot \underline{v}_h^{CR})
        - Q_K^0(p_h^{CR} \nabla\cdot\underline{v}_h^{CR})
        \right)
        =& \sum_{K\in\Th}
        Q_K^1(\gamma_h \underline{f}\cdot\underline{v}_h^{CR}),
        \\
        \sum_{K\in\Th}\left(Q_K^0(q_h\nabla\cdot\underline{u}_h^{CR}) \right)
        =& 0,
    \end{align}
    \end{subequations}
    for all $(\underline{v}_h^{CR}, q_h)\in \underline{V}_h^{CR,0}\times W_h^0$, where $\gamma_h:= \frac{\mu}{\mu+ \frac{h_K^2}{d+1}\beta}$.
    Then the solution $(\dunderline{L}_h, \underline{u}_h, \widehat{\underline{u}}_h, p_h)$ to the {\sf HDG-P0} scheme \eqref{stokesWeak0}
    satisfies
    \begin{subequations}
    \label{stCrc}
    \begin{align}
    \label{stCr0}
        \dunderline{L}_h =&\; -\mu\nabla\underline{\Pi}_h^{CR} \widehat{\underline{u}}_h, \\
    \label{stCr1}
        \underline{u}_h(m_K^i) =&\; 
        \gamma_h(m_K^i) \left(\widehat{\underline{u}}_h(m_K^i)
        + \frac{(h_K^i)^2}{(d+1)\mu}\underline{f}(m_K^i)\right),\quad
        \forall i\in\{1, \dots, d+1\}, \forall K\in \Th,\\
    \label{stCr2}
        \underline{\Pi}_h^{CR} \widehat{\underline{u}}_h =&\; \underline{u}_h^{CR},\\
    \label{stCr3}
        p_h =&\; p_h^{CR},
    \end{align}
    where $m_K^i$ is the barycenter of the $i$-th facet $F^i$ of the element $K$, and $h_K^i=\frac{|K|}{|F^i|}$.
    \end{subequations}
\end{theorem}
\begin{proof}
    The proof procedure is the same as in Theorem \ref{theo:HDGP0=CR-ellip} and we only sketch the main steps here.
    By the definition of $\underline{\Pi}_h^{CR}$ and integration by parts, we get from \eqref{stW0-1}, for all $\dunderline{G}_h\in\dunderline{W}_h$,
    \begin{align*}
        \sum_{K\in\Th} \left(Q_K^0 (\mu^{-1}\dunderline{L}_h\cdot \dunderline{G}_h) 
        + Q_{\partial{K}}^0 ( \dunderline{G}_h\underline{n} \cdot \widehat{\underline{u}}_h ) \right)
        &= \sum_{K\in\Th} Q_K^0 \left(
        (\mu^{-1}\dunderline{L}_h + \nabla\underline{\Pi}_h^{CR}\widehat{\underline{u}}_h ) \cdot \dunderline{G}_h  \right) \\
        &= (\mu^{-1}\dunderline{L}_h + \nabla\underline{\Pi}_h^{CR}\widehat{\underline{u}}_h, \dunderline{G}_h)_{\Th} = 0,
    \end{align*}
    where we used the fact that $Q_K^0$ is exact for $\poly^1(K)$ and $Q_{F}^0$ is exact for $\poly^1(F)$ for all $F\in\partial K$. The identity \eqref{stCr0} follows thanks to $(\mu^{-1}\dunderline{L}_h + \nabla\underline{\Pi}_h^{CR}\widehat{\underline{u}}_h)\in\dunderline{W}_h$.
    
    Next, for $i\in \{1, \dots, d+1\}$, taking test function $\dunderline{v}_h$ in \eqref{stW0-2} to be supported only on
    an element $K$, whose values is $1$ on $m_K^i$ and zero
    on $m_K^j$ for $j\not=i$. By the definition of $Q_{\partial K}^0$ and $Q_K^1$, we get
    \begin{align*}
        \tau_K^i(\underline{u}_h(m_K^i)- \widehat{\underline{u}}_h(m_K^i))|F^i|
        +\frac{|K|}{d+1}\beta(m_K^i)\underline{u}_h(m_K^i)
        = \frac{|K|}{d+1}\underline{f}(m_K^i).
    \end{align*}
    We immediately get the equality \eqref{stCr1} by solving for the value $\underline{u}_h(m_K^i)$ and using the definition of $\tau_K$ and $h_K$.
    
    Finally, by the definition of $\underline{\Pi}_h^{CR}$, $Q_{\partial K}^0$ and $Q_{K}^1$, the identities \eqref{stCr0}-\eqref{stCr1}, and integration by parts, we have:
    \begin{align*}
         Q_{\partial K}^0 ( (\dunderline{L}_h + p_h\dunderline{I})\underline{n}_K\cdot\widehat{\underline{v}}_h)
         =& ((\dunderline{L}_h + p_h\dunderline{I}), \nabla\underline{\Pi}_h^{CR}\widehat{\underline{v}}_h)_K \\
         =& - Q_K^0 (\mu\nabla\underline{\Pi}_h^{CR}\widehat{\underline{u}}_h \cdot \nabla\underline{\Pi}_h^{CR}\widehat{\underline{v}}_h) 
         + Q_K^0 (p_h \nabla\underline{\Pi}_h^{CR}\widehat{\underline{v}}_h),
         \\
         Q_{\partial K}^0 ( \tau_K (\underline{u}_h - \widehat{\underline{u}}_h)\cdot \widehat{\underline{v}}_h)
         =& - Q_K^1(\gamma_h\beta \underline{\Pi}_h^{CR}\widehat{\underline{u}}_h \cdot \underline{\Pi}_h^{CR}\widehat{\underline{v}}_h)
         + Q_K^1(\gamma_h \underline{f}\cdot\underline{\Pi}_h^{CR}\widehat{\underline{v}}_h),
         \\
         Q_{\partial K}^0 ( \widehat{\underline{u}}_h\cdot\underline{n}_K q_h)
         =& \langle \underline{\Pi}_h^{CR}\widehat{\underline{u}}_h\cdot \underline{n}_K, q_h
         \rangle_{\partial K} \\
         =& Q_K^0(q_h \nabla\cdot\underline{\Pi}_h^{CR}\widehat{\underline{u}}_h),
    \end{align*}
    where we used the fact that $Q_K^0$ is exact for $\poly^0(K)$ and $Q_{F}^0$ is exact for $\poly^1(F)$ for all $F\in\partial K$. By plugging the above identities into \eqref{stW0-3}-\eqref{stW0-4}, we get the condensed {\sf HDG-P0} scheme: 
    \begin{subequations}
        \label{stCondense}
        \begin{align}
            \sum_{K\in\Th}\bigl(Q_K^0(\mu\nabla \underline{\Pi}_h^{CR}\widehat{\underline{u}}_h\cdot\nabla \underline{\Pi}_h^{CR}\widehat{\underline{v}}_h)
            + Q_K^1(\gamma_h\beta \underline{\Pi}_h^{CR}\widehat{\underline{u}}_h \cdot \underline{\Pi}_h^{CR}\widehat{\underline{v}}_h)&\\
            \nonumber
            - Q_K^0(p_h \nabla\cdot\underline{\Pi}_h^{CR}\widehat{\underline{v}}_h)
            \bigr)
            &= \sum_{K\in\Th}
            Q_K^1(\gamma_h \underline{f}\cdot\underline{\Pi}_h^{CR}\widehat{\underline{v}}_h),
            \\
            \sum_{K\in\Th}\left(Q_K^0(q_h\nabla\cdot\underline{\Pi}_h^{CR}\widehat{\underline{u}}_h) \right)
            &= 0,
        \end{align}
    \end{subequations}
    for all $(\widehat{\underline{v}}_h, q_h)\in\underline{M}_h^0\times W_h^0$, which implies the result \eqref{stCr2}-\eqref{stCr3}.
\end{proof}

Same as in Section \ref{subsec:ellEq} for the reaction-diffusion equation, the above result demonstrates the equivalence between the condensed linear system of {\sf HDG-P0} for the generalized Stokes equations \eqref{stCondense} and a slightly modified CR discretization \eqref{stCr} with numerical integration, where the modification is introduced by the scaling parameter $\gamma_h$. 
In practice we first locally eliminate $\dunderline{L}_h$ and $\underline{u}_h$ from the linear system of {\sf HDG-P0} \eqref{stokesWeak0} to arrive at the condensed system \eqref{stCondense}, and then recover these local variables after solving for $\widehat{\underline{u}}_h$ and 
$p_h$ from \eqref{stCondense}.

\begin{remark}[On multigrid algorithm for \eqref{stCondense}]
\label{rk:saddle}
The saddle-point structure \eqref{stCondense} creates extra difficulty in design robust multigrid solvers for the system.
In the literature, there are three approach to construct a convergent multigrid algorithm for the nonconforming scheme 
\eqref{stCr}, which is equivalent to the condensed system \eqref{stCondense}.
The first approach \cite{brenner1990nonconforming, turek1994multigrid, stevenson1998cascade}
explores the fact that the CR discretization produces
a cell-wise divergence free velocity approximation
\[
\underline u_h^{CR}\in \underline{Z}_h^{CR,0}:=\{\underline{v}_h\in \underline V_h^{CR,0}: \; \nabla\cdot\underline v_h|_{K} = 0, \forall K\in \Th\},
\]
and applies multigrid algorithms for the positive definite system on the divergence-free kernel space $\underline{Z}_h^{CR,0}$.
This approach is restricted to two dimensions only, which is 
closely related to multigrid for the nonconforming Morley scheme for the Biharmonic equation \cite{Brenner89X}.
Its extension to three dimensions is highly nontrivial due to the need of constructing (highly complex) intergrid transfer operators between these divergence-free subspaces.
The second approach proposed by Brenner 
\cite{brenner1993nonconforming, brenner1994nonconforming} directly works with the saddle point system (with a penalty term), which, however, cannot be applied to the positive definite Schur complement subsystem. The third approach proposed by Sch\"oberl \cite{schoberl1999multigrid, schoberl1999robust} also works with a penalty formulation of the saddle point system, where a multigrid theory was applied to the positive definite subsystem for the velocity approximation. 
The key ingredients in \cite{schoberl1999multigrid,schoberl1999robust} are (i) a robust intergrid transfer operator that transfer coarse-grid divergence-free functions to fine-grid (nearly) divergence-free functions,  and (ii) a robust block-smoother capturing the divergence free basis functions. 
This approach is attractive in three dimensions as the intergrid transfer operator does not need to directly work with the divergence-free kernel space, and is much easier to realize in practice than the first approach.
Sch\"oberl's approach was originally introduced for the 
$P^2/P^0$ discretization on triangles, 
it has been applied to other finite element schemes in two- and three-dimensions by other researchers; see, e.g.,  \cite{Lee09, hong2016robust, Kanschat15,farrell2019augmented, Farrell20}.
In the next subsection, we follow the last approach to 
provide a robust multigrid algorithm for the system \eqref{stCondense} in combination with an augmented Lagrangian Uzawa iteration method \cite{fortin2000augmented, uzawa1958iterative}.
\end{remark}



\subsection{Multigrid-based augmented Lagrangian Uzawa iteration}
We use the same notations for the hierarchical meshes and finite element spaces as in Section \ref{subsec:ellipMG}. With slight abuse of notations, we define the following $L_2$-like inner product on the vector facet space $\underline{M}_l$:
\[
    (\widehat{\underline{u}}_{l},\; \widehat{\underline{v}}_{l})_{0, l}
    :=\sum_{K\in\mathcal{T}_l}Q_K^1(\underline{\Pi}_l^{CR}\widehat{\underline{u}}_{l},\; \underline{\Pi}_l^{CR}\widehat{\underline{v}}_l)
    = \sum_{K\in\mathcal{T}_l}\sum_{i=1}^{d+1}\frac{|K|}{d+1}\widehat{\underline{u}}_l(m_K^i)\;\widehat{\underline{v}}_l(m_K^i),
\]
with the induced norm $\|\cdot\|_{0, l}$, where $\underline{\Pi}_l^{CR}:\underline{M}_l\rightarrow\underline{V}_l^{CR}$ is the interpolation operator \eqref{v-picr} on $l$-th mesh level.
The pressure space $W_l$ is equipped with the standard $L_2$-norm, which we denote as 
\[
[ p_l ,q_l ]_{0,l} := 
\sum_{K\in\mathcal{T}_l}Q_K^0(p_lq_l)\quad \forall p_l,q_l\in W_l.
\]
We define the following linear operators ${\underline{A}}_{l}:\underline{M}_l^0\rightarrow\underline{M}_l^0$, and
$\underline{B}_l:\underline{M}_l^0 \rightarrow W_h$ 
\begin{alignat*}{2}
    (\underline{A}_l \widehat{\underline{u}}_l, \widehat{\underline{v}}_l)_{0,l}
    :=& \underline{a}_l(\widehat{\underline{u}}_l, \widehat{\underline{v}}_l), 
    \quad&& \forall \widehat{\underline{u}}_l, \widehat{\underline{v}}_l\in\underline{M}_l^0,
    \\
        [ \underline{B}_l \widehat{\underline{u}}_l, {{q}}_l]_{0,l}
    :=& \underline{b}_l(\widehat{\underline{u}}_l, q_l), 
    \quad&& \forall \widehat{\underline{u}}_l\in\underline{M}_l^0,
    q_l\in W_l^0, 
\end{alignat*}
where $\underline{a}_l, \underline{b}_l$ are the following bilinear forms on the $l$-th mesh level:
\begin{alignat*}{1}
    \underline{a}_l(\widehat{\underline{u}}_l, \widehat{\underline{v}}_l)
    :=& \sum_{K\in\mathcal{T}_l}\left(Q_K^0(\mu\nabla \underline{\Pi}_l^{CR}\widehat{\underline{u}}_l\cdot\nabla \underline{\Pi}_l^{CR}\widehat{\underline{v}}_l)
    + Q_K^1(\beta_l \underline{\Pi}_l^{CR}\widehat{\underline{u}}_l \cdot \underline{\Pi}_l^{CR}\widehat{\underline{v}}_l)
    \right),
    \quad \forall \widehat{\underline{u}}_l, \widehat{\underline{v}}_l\in\underline{M}_l^0, 
    \\
        \underline{b}_l(\widehat{\underline{u}}_l, q_l)
    :=& \sum_{K\in\mathcal{T}_l} Q_K^0(q_l \nabla\cdot\underline{\Pi}_l^{CR}\widehat{\underline{u}}_l), 
    \quad \forall \widehat{\underline{u}}_l\in\underline{M}_l^0, 
    q_l\in W_l,
\end{alignat*}
with $\beta_l(m_K^i):= \gamma_l(m_K^i)\beta(m_K^i)$
for all $K\in \mathcal{T}_l$.
Further denoting 
$\underline{B}^\ast_l:W_l\rightarrow \underline{M}_l^0$ as the transpose of $\underline{B}_l$ with respect to the $L_2$-inner products:
\[
(\underline{B}^\ast_l p_l, \widehat{\underline{v}}_l)_{0,l}
:=[p_l, \underline{B}_l\widehat{\underline{v}}_l]_{0,l}, \quad 
\forall p_l\in W_l, \widehat{\underline{u}}_l\in\underline{M}_l^0,
\]
and $\underline{f}_l\in\underline{M}_l^0$ such that
\[
    (\underline{f}_l,\;\widehat{\underline{v}}_l)_{0,l}
    := \sum_{K\in\mathcal{T}_l}
    Q_K^1(\gamma_l \underline{f}\cdot\underline{\Pi}_l^{CR}\widehat{\underline{v}}_l), \quad\forall\widehat{\underline{v}}_l\in\underline{M}_l^0.
\]
Then the operator form of the system \eqref{stCondense} is to find $(\widehat{\underline{u}}_J,p_J)\in\underline{M}_J^0\times W_J^0$ 
satisfying:
\begin{subequations}
    \label{stOpEq}
    \begin{align}
        \underline{A}_J \widehat{\underline{u}}_J + \underline{B}_J^\ast p_J =& \underline{f}_J,
        \\
        \label{stOpEq-2}
        \underline{B}_J \widehat{\underline{u}}_J =& 0.
\end{align}
\end{subequations}

Similar to \cite{lee2007robust,hong2016robust}, we apply the augmented Lagrangian Uzawa iteration method \cite{uzawa1958iterative, fortin2000augmented} to the above saddle-point system \eqref{stOpEq}, which solves the following  (equivalent) augmented Lagrangian formulation of \eqref{stOpEq} iteratively using the Uzawa method
\begin{subequations}
    \label{stOpEq-AL}
    \begin{align}
        \underbrace{(\underline{A}_J+\epsilon^{-1} 
        \underline{B}_J^\ast\underline{B}_J)}_{\underline{A}_J^\epsilon} \widehat{\underline{u}}_J + \underline{B}_J^\ast p_J =& \underline{f}_J,
        \\
        \underline{B}_J \widehat{\underline{u}}_J =& 0.
\end{align}
\end{subequations}
The Uzawa method for \eqref{stOpEq-AL} with damping parameter $\epsilon^{-1}\gg 1$ reads as follows:
Start with $p_J^0=0$,  for $k = 1, 2, \cdots$, find 
$(\widehat{\underline{u}}_J^{k},p_J^{k})\in\underline{M}_J^0\times W_J^0$ such that 
\begin{subequations}
    \label{stOpEq-ALU}
    \begin{align}
    \label{stOpEq-ALU1}
\underline{A}_J^\epsilon\, \widehat{\underline{u}}_J^{k}  =& \;\underline{f}_J- \underline{B}_J^\ast p_J^{k-1},
        \\
    \label{stOpEq-ALU2}
p_J^{k} =& \;   p_J^{k-1}- \epsilon^{-1}  \underline{B}_J \widehat{\underline{u}}_J^{k}.
\end{align}
\end{subequations}

Here the singularly perturbed operator 
$\underline{A}_J^\epsilon$ is associated with the bilinear form
\[
    \underline{a}^\epsilon_l(\widehat{\underline{u}}_l,\; \widehat{\underline{v}}_l)
    := \underline{a}_l(\widehat{\underline{u}}_l,\; \widehat{\underline{v}}_l)
    + \sum_{K\in\mathcal{T}_l}Q_K^0(
    \epsilon^{-1}\nabla\cdot\underline{\Pi}_l^{CR}\widehat{\underline{u}}_l \;  \nabla\cdot\underline{\Pi}_l^{CR}\widehat{\underline{v}}_l).
\]

We quote a convergence result 
from \cite[Lemma 2.1]{lee2007robust}
for the above Uzawa method \eqref{stOpEq-ALU}.
\begin{lemma}[Lemma 2.1 of \cite{lee2007robust}]
    \label{lem:uzawaConverge}
    Let $(\widehat{\underline{u}}_J, p_J)\in\underline{M}_J^0\times W_J^0$ be the solution of \eqref{stOpEq}, and let 
    $(\widehat{\underline{u}}_J^k, p_J^k)$
    be the $k$-th Uzawa iteration solution to \eqref{stOpEq-ALU}. 
Then the following estimate holds:
    \begin{align*}
        \|\widehat{\underline{u}}_J^{k} - \widehat{\underline{u}}_J\|_{A_J}
        \lesssim&
        \sqrt{\epsilon}\|p_J^k - p_J\|_0\;
        \lesssim\;
        \sqrt{\epsilon}\bigl(\frac{\epsilon}{\epsilon+\mu_0}\bigr)^{k}
        \|p_J\|_{0},
    \end{align*}
    where $\mu_0$ is the 
    minimal eigenvalue of the Schur complement
    $S_J=\underline B_J\underline A_J^{-1}\underline B^\ast_J$, which is independent of the mesh size.
\end{lemma}
\begin{remark}[On practical choice of $\epsilon$]
    \label{remark:uzawaEps}
    The first step Uzawa iteration solution 
    $(\widehat{\underline{u}}_J^1, p_J^1)$
    is simply the penalty method applied to \eqref{stOpEq}
    where a small mass term $[\epsilon p_l, q_l]_{0,l}$
    is subtracted from the continuity equation \eqref{stOpEq-2}.
    The above convergence result indicates to take 
    $0<\epsilon \ll 1$ for faster convergence in terms 
    of Uzawa iteration counts.
    However, taking $\epsilon$ extremely small will leads to
    round-off issues as the non-zero matrix entries in 
    $\underline A_J$ is of $\mathcal{O}(1)$ while that of 
    $\epsilon^{-1}\underline B_J^\ast \underline B_J$ is of $\mathcal{O}(\epsilon^{-1})$.
    In our numerical experiments, the machine precision is $10^{-16}$, and we use $\epsilon=10^{-8}$  and perform one Uzawa iteration.
\end{remark}

Since the pressure space $W_J$ is a discontinuous piecewise constant space, there is no equation solve needed for \eqref{stOpEq-ALU2}. Hence the major computational cost of a Uzawa iteration is in solving velocity from \eqref{stOpEq-ALU1}, for which we use a multigrid
following \cite{schoberl1999multigrid, schoberl1999robust}.
As mentioned in Remark \ref{rk:saddle}, two key ingredients of a robust multigrid algorithm for the system \eqref{stOpEq-ALU1} are a robust intergrid transfer operator and a robust block smoother.

It turns out that a vectorial version of the averaging 
intergrid transfer operator, denoted as $\underline{I}_{l-1}^l$,  used in \eqref{intergridOpr}
is not robust for $\epsilon\ll 1$.
Here we stabilize this averaging operator with a local correction using discrete harmonic extensions \cite{schoberl1999multigrid}, which takes care of the coarse-grid divergence.
Denoting $M_{l,T}^0$ as the local subspace of 
$M_l^0$ whose DOFs vanishes on the $(l-1)$-th level mesh skeleton $\mathcal{E}_{l-1}$, the integer grid transfer operator $\underline{\mathcal{I}}_{l-1}^l : \underline{M}_{l-1}^0\rightarrow\underline{M}_l^0$
is defined as follows:
\begin{align}
\label{div-transfer}
    \underline{\mathcal{I}}_{l-1}^{l}:=(id - \underline{P}^{T}_{\mathcal{A}_{l}^\epsilon})\underline{I}_{l-1}^l,
\end{align}
where $id$ is the identity operator and  the local projection
$\underline{P}_{\underline{A}_{l}^\epsilon}^{T}: {\underline{M}}_{l}^0\rightarrow {\underline{M}}_{l,T}^{0}$ satisfies
\begin{align}
\label{harmonic-ex}
    \underline{a}_{l}^\epsilon(\underline{P}_{\mathcal{A}_{l}^{\epsilon}}^{T}\widehat{\underline{u}}_l,\; \widehat{\underline{v}}_l^T)
    =
    \underline{a}_{l}^\epsilon(\widehat{\underline{u}}_l,\; \widehat{\underline{v}}_l^T), 
    \quad \forall \widehat{\underline{u}}_l\in\underline{M}_{l}^0,\; \widehat{\underline{v}}_l^T\in\underline{M}_{l,T}^{0},
\end{align}
which can be solved element-by-element on the coarse mesh $\mathcal{T}_{l-1}$.

A robust smoother for \eqref{stOpEq-ALU1} needs to take care of the discretely divergence-free kernel space. Here the classical block smoothers for 
$H(\mathrm{div})$-elliptic problems from Arnold, Falk and Winther are used \cite{arnold2000multigrid}.
In particular, we use vertex-patch based damped block Jacobi
or block Gauss-Seidel smoother, which is of additive Schwarz type \footnote{Edge-patch based block smoother can also be used in three dimensions \cite{arnold2000multigrid} to reduce memory consumption.}.
For completeness, we formulate the vertex-patch damped block Jacobi smoother below.
Denoting $\EuScript{V}_l$ as the set of vertices of the triangulation $\mathcal{T}_l$.
We define the subset of facets $\mathcal{E}_{l,v}$ 
meeting at the vertex
$v\in \EuScript{V}_l$ as
\[
    \mathcal{E}_{l,v}:=\bigcup\limits_{\substack{F\in\mathcal{E}_l, \; v\in F}} F,
\]
and decompose the finite element space $\underline{M}_{l}^0$
into overlapping subspaces with support on $\mathcal{E}_{l,v}$: 
\begin{align*}
    \underline{M}_{l}^0 = \sum_{v\in \EuScript{V}_l} \underline{M}_{l, v}^0:=
\sum_{v\in \EuScript{V}_l}
    \{
        \widehat{\underline{v}}_l\in\underline{M}_{l}^0:\;
        \mathrm{supp}\;\widehat{\underline{v}}_l\subset\mathcal{E}_{l,v}
    \},
\end{align*}
Further defining $\underline{P}_{\underline{A}_{l}^\epsilon}^v:\; \underline{M}_{l}^0\rightarrow\underline{M}_{l,v}^{0}$ as the local projection onto the subspace $\underline{M}_l^{0,v}$ with respect to the bilinear form $a_l^\epsilon$ such that:
\[
    \underline{a}_{l}^\epsilon (\underline{P}_{\underline{A}_{l}^\epsilon}^v\widehat{\underline{u}}_l,\;\widehat{\underline{v}}_{l,i}) 
    = \underline{a}_{l}^\epsilon (\widehat{\underline{u}}_l,\;\widehat{\underline{v}}_{l,v}), \quad
    \forall \widehat{\underline{u}}_l \in\underline{M}_{l}^0,\;
    \widehat{\underline{v}}_{l,v}\in\underline{M}_{l,v}^{0},\;
    v\in\EuScript{V}_l,
\]
the damped block Jacobi smoother is then given as:
\begin{align}
\label{bjac}
    \underline{R}_l:=
    \varsigma\sum_{v\in \EuScript{V}_l}\underline{P}_{\underline{A}_{l}^\epsilon}^v (\underline{A}_{l}^{\epsilon})^{-1},
\end{align}
where $\varsigma$ is the damping parameter that is small enough (independent of $\epsilon$) to ensure 
the operator $(id - \underline{R}_l\underline{A}_l^\epsilon)$
is a positive definite contraction.

With these two ingredients ready, we are ready to present the W-cycle and variable V-cycle multigrid algorithms for the linear system $\underline{A}_l^\epsilon \underline{u}_l = \underline{g}_l\in \underline{M}_l^0$ 
as below:

\begin{algorithm}[H]
\caption{The multigrid algorithm for $\underline{A}_l^\epsilon\widehat{\underline{u}}_l = \underline{g}_l$.}
\label{alg:stokesMG}
\begin{algorithmic}
\State The $l$-th level multigrid algorithm produces $MG_{\underline{A}_l^\epsilon}(l, \underline{g}_l, \widehat{\underline{u}}_l^0, m(l), q)$
as an approximation solution for
$\underline{A}_l^\epsilon\widehat{\underline{u}}_l = \underline{g}_l$ with initial guess $\widehat{\underline{u}}_l^0$, where $m(l)$ denotes the number of pre-smoothing and post-smoothing steps on the $l$-th mesh level, and $q\in\{1, 2\}$.
Here $q=1$ corresponding to the V-cycle algorithm, and $q=2$ corresponding to the W-cycle algorithm \red{in which constant smoothing steps $m(1)=\cdots=m(l)=m$ is used}.
\If {$l = 1$}
    \State
    ${MG}_{\underline{A}_l^\epsilon}(l, \underline{g}_l,  \widehat{\underline{u}}_l, m(l), q) = (\underline{A}_l^\epsilon)^{-1} \underline{g}_l$.
\Else
\State Perform the following three steps:
\State (1) {\it Pre-smoothing.}
For $j = 1,\dots,m(l)$, compute $\widehat{\underline{u}}_l^j$ by
    \[\widehat{\underline{u}}_l^{j} = 
    \widehat{\underline{u}}_l^{j-1} + \underline{R}_l(\underline{g}_l - \underline{A}_l^\epsilon \widehat{\underline{u}}_l^{j-1}).\]
\State (2) {\it Coarse grid correction.}
Let $\delta{\underline{u}}_{l-1}^0 = \underline{0}$ and $\widehat{\underline{r}}_{l-1}= \underline{\mathcal{I}}_l^{l-1}(\underline{g}_l - \underline{A}_l^\epsilon \widehat{\underline{u}}_l^{m(l)})$. For $k = 1,\dots,q$, \indent compute $\delta{\underline{u}}_{l-1}^k$ by
\[
    \delta{\underline{u}}_{l-1}^k =  
    MG_{\underline{A}_{l-1}^\epsilon}(l-1, \widehat {\underline{r}}_{l-1}, \delta{\underline{u}}_{l-1}^{k-1}, m(l-1)). 
\]
\indent Then we get $\widehat{\underline{u}}_l^{m(l)+1} = \widehat{\underline{u}}_l^{m(l)} + \underline{\mathcal{I}}_{l-1}^l \delta{\underline{u}}_{l-1}^q$,
where $\underline{\mathcal{I}}_l^{l-1}: \underline M_{l}^0\rightarrow \underline M_{l-1}^0$ is the restriction operator \indent satisfying
\[
    (\underline{\mathcal{I}}_{l}^{l-1}\widehat{\underline{u}}_{l},\; \widehat{\underline{v}}_{l-1})_{0,l-1}
    =
    (\widehat{\underline{u}}_{l},\; \underline{\mathcal{I}}_{l-1}^l\widehat{\underline{v}}_l^T)_{0,l}, 
    \quad \forall \widehat{\underline{u}}_l\in\underline{M}_{l}^0,\; \widehat{\underline{v}}_{l-1}\in\underline{M}_{l-1}^{0}.
    \]
\State (3) {\it Post-smoothing.} 
For $j=m(l)+2,\dots, 2m(l)+1$, compute $\widehat{\underline{u}}_l^j$
by 
\[
    \widehat{\underline{u}}_l^{j} = 
    \widehat{\underline{u}}_l^{j-1} + \underline{R}_l^T(\underline{g}_l - \underline{A}_l^\epsilon \widehat{\underline{u}}_l^{j-1}),
\]
\State where $\underline{R}_l^T$ is the transpose of $\underline{R}_l$ with respect to the inner product $(\cdot,\;\cdot)_{0,l}$.
\State We then define 
$MG_{\underline{A}_l}(l, \underline{g}_l,\widehat{\underline{u}}_l^0, m(l), q) = \widehat{\underline{u}}_l^{2m(l)+1}$.\EndIf
\end{algorithmic}
\end{algorithm}

The analysis of the above multigrid algorithm follows directly from Sch\"oberl's work \cite{schoberl1999multigrid, schoberl1999robust}, which is based on classical multigrid theories \cite{bramble1987new, bramble1991analysis, hackbusch2013multi}, and boils down to the verification of three properties of the underlying operators.
To perform the analysis, we denote the following parameter-dependent $L_2$-like norm $\|\cdot\|_{\epsilon,l}$ on $\underline{M}_l^0$:
\[
    \|\widehat{\underline{u}}_l\|_{\epsilon,l}^2:=
    \|\widehat{\underline{u}}_l\|_{0, l}^2
    +
    {\epsilon}^{-1}{h}_l^{2}\|\nabla\cdot\underline{\Pi}_{l}^{CR}\widehat{\underline{u}}_l\|_{0}^2
    +
    \epsilon^{-2}{h}_l^{2}\|\varPi_{l-1}^{W}( \nabla\cdot\underline{\Pi}_l^{CR}\widehat{\underline{u}}_l )\|_{0}^2 ,
\] 
where $\varPi_{l-1}^W$ is the $L_2$-projection onto the piecewise constant space $W_{l-1}$ on the $(l-1)$-th  level mesh. We define $\|\cdot\|_{\underline{A}_l^{\epsilon}}$ as the norm induced by the SPD operator ${\underline{A}_l^{\epsilon}}$, i.e. $\|\cdot\|_{\underline{A}_l^{\epsilon}}:=\sqrt{(\underline{A}_l^{\epsilon}\; \cdot,\; \cdot)}$,
and denote
$\underline{P}_{\mathcal{A}_{l}^\epsilon}^{l-1}:\; \underline{M}_{l}^0\rightarrow\underline{M}_{l-1}^0$ as the transpose of $\underline{\mathcal{I}}_{l-1}^l$ with respect to the bilinear form $\underline{a}_{l}^\epsilon$ such that:
\[
    \underline{a}_{l-1}^\epsilon(\underline{P}_{\mathcal{A}_{l}^\epsilon}^{l-1}\widehat{\underline{v}}_l,\; \widehat{\underline{v}}_{l-1})
    =
    \underline{a}_{l}^\epsilon(\widehat{\underline{v}}_l,\; \underline{\mathcal{I}}_{l-1}^l\widehat{\underline{v}}_{l-1}),\quad
    \forall \widehat{\underline{u}}_l\in\underline{M}_{l}^0,\; 
    \widehat{\underline{v}}_{l-1}\in\underline{M}_{l-1}^0.
\]

We quote the abstract multigrid convergence result from 
\red{\cite[Theorem 3.7]{schoberl1999robust}}.
\begin{theorem}[\red{Theorem 3.7 of \cite{schoberl1999robust}}]
\label{thm:stokesMG}
Let the multigrid procedure be as defined in Algorithm \ref{alg:stokesMG}. Assume there hold the following 
properties:
\begin{itemize}
    \item [-] (A0) Scaling property. The spectrum of $id-\underline{R}_l\underline{A}_l^\epsilon$ is in the interval $(0, 1)$.
    \item [-] (A1) Approximation property. 
    There exists a positive constant $C$ independent of mesh size and the parameter $\epsilon$ such that:
    \label{theo:stokesApprox}
    \[
    \left\|\left(id - \underline{\mathcal{I}}_{l-1}^{l}\underline{P}_{\mathcal{A}_{l}^\epsilon}^{l-1}\right)\widehat{\underline{u}}_l\right\|_{\epsilon, l} 
    \le
    C h_l \|\widehat{\underline{u}}_l\|_{\underline{A}_{l}^\epsilon}, \quad
    \forall \widehat{\underline{u}}_l\in\underline{M}_{l}^0.
    \]
    \item [-] (A2) Smoothing property. There exists a positive constant $C$ independent of mesh size $h_l$ and the parameter $\epsilon$ such that:
    \[
    \left\|
        (id - \underline{R}_l \underline{A}_{l}^\epsilon)^{m(l)}\widehat{\underline{u}}_l
    \right\|_{\underline{A}_{l}^\epsilon}
    \le
    C m(l)^{-1/4} h_l^{-1}\|\widehat{\underline{u}}_l\|_{\epsilon, l},
    \quad
    \forall \widehat{\underline{u}}_l\in\underline{M}_{l}^0.
    \]
\end{itemize}
Then the following multigrid methods lead to optimal solvers:
\begin{itemize}
    \item The W-cycle multigrid algorithm with sufficiently many smoothing steps leads to a convergent method.
    That is, there exists positive constants $m_\ast$ and $C$ independent of mesh size $h_l$ and parameter $\epsilon$, such that with $q = 2$ and $m(1)= \cdots=m(l) =  m$ in Algorithm \ref{alg:stokesMG} we have:
    \[
        \|\underline{\mathbb{E}}_{l,m}\widehat{\underline{v}}_l\|_{\underline{A}_l^\epsilon}
        \le  C m^{-1/4} \|\underline{v}_l\|_{\underline{A}_l^\epsilon},
        \quad
        \forall \widehat{\underline{v}}_l\in\underline{M}_l^0,\;
        l\ge 1,\; m\ge m_\ast,
    \]
    where $\mathbb{\underline{E}}_{l,m}: \underline{M}_l^0\rightarrow \underline{M}_l^0$
    is the operator relating the initial error and the final error of the multigrid $W$-cycle algorithm, i.e., 
    \[
    \mathbb{\underline{E}}_{l,m}(\widehat{\underline{u}}_l-\widehat{\underline{u}}_l^0) := \widehat{\underline{u}}_l - MG_{\underline{A}_l^\epsilon}(l, \underline{g}_l, \widehat{\underline{u}}_l^0, m, 2). 
    \]
    \item The variable V-cycle algorithm leads to a robust preconditioner. That is, 
    with $q=1$ and $\beta_0 m(l) \le m(l-1) \le  \beta_1 m(l)$ in Algorithm \ref{alg:stokesMG}  ($1<\beta_0<\beta_1$), there exists positive constant $C$ independent of mesh size $h_l$ and parameter $\epsilon$ such that
    \[
        \kappa(\underline{\mathbb{B}}_{l,m(l)}\underline{A}_l^\epsilon) \le 1 + C m(l)^{-1/4} ,
    \]
    where $\kappa$ is the condition number, and  $\mathbb{\underline{B}}_{l,m(l)}: \underline{M}_l^0\rightarrow \underline{M}_l^0$
    is the preconditioning operator relating the residual to the correction of the variable $V$-cycle algorithm with a zero initial guess, i.e., 
    \[
        \underline{\mathbb{B}}_{l,m(l)}\widehat{\underline{v}}_l:=
        MG_{\underline{A}_l^\epsilon}(l, \widehat{\underline{v}}_l, 0, m(l),1),
        \quad \forall \widehat{\underline{v}}_l\in\underline{M}_l^0.
    \]
\end{itemize}
\end{theorem}
\begin{remark}[Verifications of the Assumptions in Theorem \ref{thm:stokesMG}]
Assumption (A0) is satisfied by taking the damping parameter $\varsigma$ sufficiently small, which only depends on the (bounded) number of overlapping blocks \cite{arnold2000multigrid}. 
Note that if a block Gauss-Seidel smoother is used, there is no need to add damping.

\red{We note that the approximation property (A1), the smoothing property (A2), and thus our proposed multigrid method for the {\sf HDG-P0} scheme for the generalized Stokes equations require full elliptic regularity assumed in \eqref{stReg}.}
The verification of the approximation property (A1) essentially follows from \red{\cite[Section 4.4]{schoberl1999robust}} (see also \red{\cite[Section 5]{schoberl1999multigrid}}), which  establishes a stable decomposition result for the coarse grid operator $id-\underline{\mathcal{I}}_{l-1}^{l}\underline{P}_{\mathcal{A}_{l}^\epsilon}^{l-1}$, where the analysis was performed on an equivalent mixed formulation. 
Here a key step in the proof is to show the averaging operator 
$\underline I_{l-1}^{l}$
preserves the divergence on the coarse-grid:
\begin{align*}
    \int_{K}
\nabla\cdot \underline{\Pi}_{l-1}^{CR}\underline{\widehat{v}}_{l-1}\mathrm{dx}
=
&\;
\int_{\partial K}
\underline{\widehat{v}}_{l-1}\cdot\underline n_K\mathrm{ds}
=\;
\int_{\partial K}
\underline I_{l-1}^{l}\underline{\widehat{v}}_{l-1}\cdot\underline n_{K}\mathrm{ds}\\
=&\;
\sum_{j=1}^s\int_{\partial K^j}
\underline I_{l-1}^{l}\underline{\widehat{v}}_{l-1}\cdot\underline n_{K^j}\mathrm{ds}
=
\sum_{j=1}^s\int_{K^j}
\nabla\cdot \underline{\Pi}_{l}^{CR}
\underline I_{l-1}^{l}\underline{\widehat{v}}_{l-1}\mathrm{dx},
\end{align*}
for all $K\in \mathcal{T}_{l-1}$, where 
the elements $K^j$ are children of $K$ such that $\bar K=\cup_{j=1}^s \bar{K}^j$.
Hence, the coarse-grid divergence correction only need to be locally performed on the bubble spaces $M_{l,T}^0$ in \eqref{harmonic-ex}.

The verification of the 
smoothing property (A2) follows from \red{\cite[Section 6]{schoberl1999multigrid}}, where a key step is to establish the following bound of the interpolation norm (Lemma 7 in \red{\cite{schoberl1999multigrid}})
\begin{align}
\label{interp}
    \|\widehat{\underline u}_{l,1}\|_{[\underline D_l^\epsilon, \underline A_l^\epsilon]}\lesssim 
\|\widehat{\underline u}_{l,1}\|_{\epsilon, l},
\end{align}
where $\widehat{\underline u}_{l,1}\in 
\underline{Z}_{l}^0:=
\{\widehat {\underline{v}}_{l}\in
\underline{M}_{l}^0:\; 
\nabla\cdot \underline{\Pi}_l^{CR}\widehat {\underline{v}}_{l}|_{K} = 0,\; \forall K\in \mathcal{T}_l
\}
$
is a discretely divergence-free function, 
$\|\cdot\|_{\underline D_l^\epsilon}$ is the additive Schwarz norm  expressed by 
\[
\|\widehat{\underline u}_l\|_{\underline D_l^\epsilon}^2:=
\mathrm{inf}_{u=\sum_v u_{l,v}, u_{l,v}\in \underline M_{l,v}}
\sum_{v\in \EuScript{V}_l}\|u_{l,v}\|_{\underline A_l^\epsilon}^2,
\]
and 
$\|\cdot\|_{[\underline D_l^\epsilon, \underline A_l^\epsilon]}$ is the interpolation norm between 
$\|\cdot\|_{\underline D_l^\epsilon}$ and 
$\|\cdot\|_{\underline A_l^\epsilon}$.
The proof of \eqref{interp} relies on the explicit characterization of discretely divergence-free functions in $\underline{Z}_l^0$
using the nonconforming Stokes complex \cite{brenner2015forty, huang2020nonconforming}, which can be expressed as the discrete curl of the 
Morley finite element \cite{morley1968triangular} when $d=2$, and the discrete curl of the 
$H$(grad curl)-nonconforming finite element \cite{Xucurl, huang2020nonconforming} when $d=3$.
Using such explicit characterization, one can prove
\[
    \|\widehat{\underline u}_{l,1}\|_{[\underline D_l^\epsilon, \underline A_l^\epsilon]}\lesssim 
        \|\widehat{\underline u}_{l,1}\|_{\underline D_l^\epsilon}\lesssim
\|\widehat{\underline u}_{l,1}\|_{\epsilon, l},
\]
see similar derivation in \cite[Lemma 9]{hong2016robust}.
\end{remark}

\section{Numerical experiments}
\label{sec:numExp}
In this section, we present numerical experiments to verify the optimal convergence rates of the {\sf HDG-P0} schemes and the optimality of the multigrid algorithms. 
\red{We tested the proposed multigrid algorithms both as the iteration solvers and as the preconditioners for the preconditioned conjugate gradient (PCG) method to solve the condensed systems \eqref{op1} and \eqref{stOpEq-ALU1}, with a relative tolerance of $10^{-8}$ used as the stopping criterion.
The bisection algorithm is used for (local) mesh refinements.}
All results are obtained by using the NGSolve software\cite{Schoberl16}. 
Source code is 
available at \url{https://github.com/WZKuang/MG4HDG-P0.git}.

\subsection{Reaction-diffusion equation}
\subsubsection{Manufactured solution}
\label{subsec:diffConvergeNum}
We first verify the optimal convergence rates of the {\sf HDG-P0} scheme \eqref{ellipWeak0} for the reaction-diffusion equation with known solution. We set the domain as a unit square/cube $\Omega=[0, 1]^d$ with homogeneous Dirichlet boundary conditions on all sides. We let the coefficients $\alpha = \beta = 1 + \frac{1}{2}sin(x)sin(y)$ when $d=2$, and $\alpha = \beta = 1 + \frac{1}{2}sin(x)sin(y)sin(z)$ when $d=3$. The exact solution is set as:
\begin{align*}
    u = \left\{
    \begin{tabular}{l l}
        $(x-x^2)(y-y^2),$ & $\text{when $d = 2$,}$
        \\
        $(x-x^2)(y-y^2)(z-z^2),$ & $\text{when $d = 3$.}$
    \end{tabular} 
    \right.
\end{align*}
Table \ref{tab:poissonRate} reports the estimated order of convergence (EOC) of the $L_2$ norms $\|u_h - u\|_{0}$ and $\|\underline{\sigma_h} - \underline{\sigma}\|_{0}$ of the {\sf HDG-P0} scheme \eqref{ellipWeak0}, and optimal convergence rates are obtained as expected.

\red{
Standard V-cycle multigrid in Algorithm \ref{alg:ellipMG} is tested both as the iteration solver and as the preconditioner for the PCG algorithm for the the condensed {\sf HDG-P0} scheme \eqref{hdgX}.
}
The coarsest mesh is a triangulation of $\Omega$ with the maximum element diameter less than $1/4$ in both two-dimensional and three-dimensional cases, \red{followed by uniform refinements.}
\red{
The iteration counts of the multigrid iteration and the PCG method are reported respectively in Table \ref{tab:poissonSolver} and Table \ref{tab:poissonSm}
with different smoothers and smoothing steps in the multigrid algorithms,
where we denote P-JAC as the damped point Jacobi smoother with damping parameter set as 0.5,
and P-GS as the point Gauss-Seidel smoother.
For the multigrid iteration solver, we find that the standard V-cycle fails when the smoothing steps are not large enough.
The needed smoothing steps to make the multigrid iteration converge in three-dimensional cases are larger than those in two-dimensional cases.
For the PCG method, the iteration counts and condition numbers are bounded independent of the mesh size in two dimensions for all cases.}
There is a very mild iteration count growth in three dimensions when we only performing one or two smoothing steps, where the situation is further improved when we take 4 smoothing steps.

\begin{table}[ht]
    \begin{tabular}{|c|cccc|cccc|}
    \hline
    \multirow{2}{*}{Level} & \multicolumn{4}{c|}{$d = 2$} & \multicolumn{4}{c|}{$d = 3$} \\ \cline{2-9} 
     & $\|u_h - u\|_{0}$ & \multicolumn{1}{c|}{EOC} & $\|\underline{\sigma}_h - \underline{\sigma}\|_{0}$ & EOC & $\|u_h - u\|_{0}$ & \multicolumn{1}{c|}{EOC} & $\|\underline{\sigma}_h - \underline{\sigma}\|_{0}$ & EOC \\ \hline
    1 & $1.17 e{-1}$ & \multicolumn{1}{c|}{$\backslash$} & $8.29 e{-1}$ & $\backslash$ & $1.67 e{-2}$ & \multicolumn{1}{c|}{$\backslash$} & $1.72 e{-1}$ & $\backslash$ \\ \cline{1-1}
    2 & $2.96 e{-2}$ & \multicolumn{1}{c|}{1.98} & $4.22 e{-1}$ & 0.97 & $4.25 e{-3}$ & \multicolumn{1}{c|}{1.97} & $9.22 e{-2}$ & 0.90 \\ \cline{1-1}
    3 & $7.44 e{-3}$ & \multicolumn{1}{c|}{1.99} & $2.12 e{-1}$ & 0.99 & $1.07 e{-3}$ & \multicolumn{1}{c|}{1.99} & $4.68 e{-2}$ & 0.98 \\ \cline{1-1}
    4 & $1.86 e{-3}$ & \multicolumn{1}{c|}{2.00} & $1.06 e{-1}$ & 1.00 & $2.66 e{-4}$ & \multicolumn{1}{c|}{2.00} & $2.35 e{-2}$ & 1.00 \\ \cline{1-1}
    5 & $4.66 e{-4}$ & \multicolumn{1}{c|}{2.00} & $5.31 e{-2}$ & 1.00 & $6.64 e{-5}$ & \multicolumn{1}{c|}{2.00} & $1.17 e{-2}$ & 1.00 \\ \hline
    \end{tabular}
    \caption{Estimated convergence rates of the {\sf HDG-P0} scheme in Example \ref{subsec:diffConvergeNum}.}
    \label{tab:poissonRate}
\end{table}

\begin{table}[ht]
    \red{
    \begin{tabular}{|cccccccc|}
    \hline
    \multicolumn{8}{|c|}{$d=2$}                                                                                                                  \\ \hline
    \multicolumn{2}{|c|}{\multirow{2}{*}{Parameters}}           & \multicolumn{3}{c|}{P-JAC}                       & \multicolumn{3}{c|}{P-GS}   \\ \cline{3-8} 
    \multicolumn{2}{|c|}{}    & $m = 1$ & $m = 2$ & \multicolumn{1}{c|}{$m = 4$} & $m = 1$ & $m = 2$ & $m = 4$ \\ \hline
    \multicolumn{1}{|c|}{$J$} & \multicolumn{1}{c|}{Facet DOFs} & \multicolumn{6}{c|}{Iteration Count}                                           \\ \hline
    \multicolumn{1}{|c|}{2}   & \multicolumn{1}{c|}{$2.20 e2$}  & 43      & 23      & \multicolumn{1}{c|}{13}      & 20      & 10      & 7       \\ \cline{1-2}
    \multicolumn{1}{|c|}{3}   & \multicolumn{1}{c|}{$8.48 e2$}  & 97      & 25      & \multicolumn{1}{c|}{14}      & 22      & 12      & 8       \\ \cline{1-2}
    \multicolumn{1}{|c|}{4}   & \multicolumn{1}{c|}{$3.33 e3$}  & N/A     & 29      & \multicolumn{1}{c|}{16}      & 25      & 15      & 9       \\ \cline{1-2}
    \multicolumn{1}{|c|}{5}   & \multicolumn{1}{c|}{$1.32 e4$}  & N/A     & 33      & \multicolumn{1}{c|}{17}      & 29      & 17      & 10      \\ \cline{1-2}
    \multicolumn{1}{|c|}{6}   & \multicolumn{1}{c|}{$5.25 e4$}  & N/A     & 37      & \multicolumn{1}{c|}{18}      & 34      & 19      & 11      \\ \cline{1-2}
    \multicolumn{1}{|c|}{7}   & \multicolumn{1}{c|}{$2.09 e5$}  & N/A     & 42      & \multicolumn{1}{c|}{18}      & 43      & 21      & 11      \\ \cline{1-2}
    \multicolumn{1}{|c|}{8}   & \multicolumn{1}{c|}{$8.37 e5$}  & N/A     & 50      & \multicolumn{1}{c|}{19}      & 56      & 23      & 12      \\ \hline\hline
    \multicolumn{8}{|c|}{$d=3$}                                                                                                                  \\ \hline
    \multicolumn{2}{|c|}{\multirow{2}{*}{Parameters}}           & \multicolumn{3}{c|}{P-JAC}                       & \multicolumn{3}{c|}{P-GS}   \\ \cline{3-8} 
    \multicolumn{2}{|c|}{}    & $m = 1$ & $m = 4$ & \multicolumn{1}{c|}{$m = 8$} & $m = 1$ & $m = 4$ & $m = 8$ \\ \hline
    \multicolumn{1}{|c|}{$J$} & \multicolumn{1}{c|}{Facet DOFs} & \multicolumn{6}{c|}{Iteration Count}                                           \\ \hline
    \multicolumn{1}{|c|}{2}   & \multicolumn{1}{c|}{$3.45 e3$}  & 160     & 21      & \multicolumn{1}{c|}{13}      & 39      & 9       & 6       \\ \cline{1-2}
    \multicolumn{1}{|c|}{3}   & \multicolumn{1}{c|}{$1.05 e4$}  & N/A     & 30      & \multicolumn{1}{c|}{17}      & 63      & 11      & 7       \\ \cline{1-2}
    \multicolumn{1}{|c|}{4}   & \multicolumn{1}{c|}{$3.02 e4$}  & N/A     & 62      & \multicolumn{1}{c|}{21}      & N/A     & 13      & 8       \\ \cline{1-2}
    \multicolumn{1}{|c|}{5}   & \multicolumn{1}{c|}{$8.42 e4$}  & N/A     & 80      & \multicolumn{1}{c|}{23}      & N/A     & 15      & 9       \\ \cline{1-2}
    \multicolumn{1}{|c|}{6}   & \multicolumn{1}{c|}{$2.26 e5$}  & N/A     & N/A     & \multicolumn{1}{c|}{36}      & N/A     & 22      & 10      \\ \cline{1-2}
    \multicolumn{1}{|c|}{7}   & \multicolumn{1}{c|}{$6.01 e5$}  & N/A     & N/A     & \multicolumn{1}{c|}{48}      & N/A     & 30      & 11      \\ \cline{1-2}
    \multicolumn{1}{|c|}{8}   & \multicolumn{1}{c|}{$1.54 e6$}  & N/A     & N/A     & \multicolumn{1}{c|}{62}      & N/A     & 33      & 11      \\ \hline
    \end{tabular}
    }
\caption{Iteration counts of V-cycle multigrid iteration solver
with different smoothers and smoothing steps in Example \ref{subsec:diffConvergeNum}.}
\label{tab:poissonSolver}
\end{table}

\begin{table}[ht]
    \setlength{\tabcolsep}{6pt}
    \begin{tabular}{|cccccccccccccc|}
    \hline
    \multicolumn{14}{|c|}{$d=2$}                                                                                                                                                                                                                                                                       \\ \hline
    \multicolumn{2}{|c|}{\multirow{2}{*}{Parameters}}           & \multicolumn{6}{c|}{P-JAC}                                                                                            & \multicolumn{6}{c|}{P-GS}                                                                                    \\ \cline{3-14} 
    \multicolumn{2}{|c|}{}    & \multicolumn{2}{c|}{$m = 1$}          & \multicolumn{2}{c|}{$m = 2$}          & \multicolumn{2}{c|}{$m = 4$}          & \multicolumn{2}{c|}{$m = 1$}          & \multicolumn{2}{c|}{$m = 2$}          & \multicolumn{2}{c|}{$m = 4$} \\ \hline
    \multicolumn{1}{|c|}{$J$} & \multicolumn{1}{c|}{Facet DOFs} & \# It & \multicolumn{1}{c|}{$\kappa$} & \# It & \multicolumn{1}{c|}{$\kappa$} & \# It & \multicolumn{1}{c|}{$\kappa$} & \# It & \multicolumn{1}{c|}{$\kappa$} & \# It & \multicolumn{1}{c|}{$\kappa$} & \# It       & $\kappa$       \\ \hline
    \multicolumn{1}{|c|}{2}   & \multicolumn{1}{c|}{$2.20 e2$}  & 19    & \multicolumn{1}{c|}{4.9}      & 13    & \multicolumn{1}{c|}{2.6}      & 9     & \multicolumn{1}{c|}{1.5}      & 12    & \multicolumn{1}{c|}{2.2}      & 8     & \multicolumn{1}{c|}{1.3}      & 6           & 1.1            \\ \cline{1-2}
    \multicolumn{1}{|c|}{3}   & \multicolumn{1}{c|}{$8.48 e2$}  & 22    & \multicolumn{1}{c|}{7.1}      & 14    & \multicolumn{1}{c|}{3.4}      & 10    & \multicolumn{1}{c|}{1.8}      & 13    & \multicolumn{1}{c|}{2.8}      & 9     & \multicolumn{1}{c|}{1.5}      & 6           & 1.1            \\ \cline{1-2}
    \multicolumn{1}{|c|}{4}   & \multicolumn{1}{c|}{$3.33 e3$}  & 23    & \multicolumn{1}{c|}{8.0}      & 15    & \multicolumn{1}{c|}{3.8}      & 11    & \multicolumn{1}{c|}{2.1}      & 14    & \multicolumn{1}{c|}{3.2}      & 9     & \multicolumn{1}{c|}{1.7}      & 7           & 1.2            \\ \cline{1-2}
    \multicolumn{1}{|c|}{5}   & \multicolumn{1}{c|}{$1.32 e4$}  & 25    & \multicolumn{1}{c|}{11}       & 16    & \multicolumn{1}{c|}{4.7}      & 11    & \multicolumn{1}{c|}{2.3}      & 14    & \multicolumn{1}{c|}{3.5}      & 10    & \multicolumn{1}{c|}{1.8}      & 7           & 1.2            \\ \cline{1-2}
    \multicolumn{1}{|c|}{6}   & \multicolumn{1}{c|}{$5.25 e4$}  & 26    & \multicolumn{1}{c|}{12}       & 16    & \multicolumn{1}{c|}{4.9}      & 11    & \multicolumn{1}{c|}{2.3}      & 15    & \multicolumn{1}{c|}{3.7}      & 10    & \multicolumn{1}{c|}{1.9}      & 7           & 1.3            \\ \cline{1-2}
    \multicolumn{1}{|c|}{7}   & \multicolumn{1}{c|}{$2.09 e5$}  & 26    & \multicolumn{1}{c|}{12}       & 16    & \multicolumn{1}{c|}{5.1}      & 11    & \multicolumn{1}{c|}{2.4}      & 15    & \multicolumn{1}{c|}{4.0}      & 10    & \multicolumn{1}{c|}{2.0}      & 7           & 1.3            \\ \cline{1-2}
    \multicolumn{1}{|c|}{8}   & \multicolumn{1}{c|}{$8.37 e5$}  & 26    & \multicolumn{1}{c|}{12}       & 16    & \multicolumn{1}{c|}{5.1}      & 11    & \multicolumn{1}{c|}{2.5}      & 15    & \multicolumn{1}{c|}{4.1}      & 10    & \multicolumn{1}{c|}{2.0}      & 7           & 1.3            \\ \hline\hline
    \multicolumn{14}{|c|}{$d=3$}                                                                                                                                                                                                                                                                       \\ \hline
    \multicolumn{2}{|c|}{\multirow{2}{*}{Parameters}}           & \multicolumn{6}{c|}{P-JAC}                                                                                            & \multicolumn{6}{c|}{P-GS}                                                                                    \\ \cline{3-14} 
    \multicolumn{2}{|c|}{}    & \multicolumn{2}{c|}{$m = 1$}          & \multicolumn{2}{c|}{$m = 2$}          & \multicolumn{2}{c|}{$m = 4$}          & \multicolumn{2}{c|}{$m = 1$}          & \multicolumn{2}{c|}{$m = 2$}          & \multicolumn{2}{c|}{$m = 4$} \\ \hline
    \multicolumn{1}{|c|}{$J$} & \multicolumn{1}{c|}{Facet DOFs} & \# It & \multicolumn{1}{c|}{$\kappa$} & \# It & \multicolumn{1}{c|}{$\kappa$} & \# It & \multicolumn{1}{c|}{$\kappa$} & \# It & \multicolumn{1}{c|}{$\kappa$} & \# It & \multicolumn{1}{c|}{$\kappa$} & \# It       & $\kappa$       \\ \hline
    \multicolumn{1}{|c|}{2}   & \multicolumn{1}{c|}{$3.45 e3$}  & 26    & \multicolumn{1}{c|}{9.7}      & 18    & \multicolumn{1}{c|}{4.6}      & 13    & \multicolumn{1}{c|}{2.4}      & 18    & \multicolumn{1}{c|}{4.7}      & 11    & \multicolumn{1}{c|}{1.8}      & 7           & 1.2            \\ \cline{1-2}
    \multicolumn{1}{|c|}{3}   & \multicolumn{1}{c|}{$1.05 e4$}  & 37    & \multicolumn{1}{c|}{19}       & 25    & \multicolumn{1}{c|}{8.2}      & 17    & \multicolumn{1}{c|}{4.2}      & 23    & \multicolumn{1}{c|}{7.3}      & 14    & \multicolumn{1}{c|}{2.8}      & 9           & 1.6            \\ \cline{1-2}
    \multicolumn{1}{|c|}{4}   & \multicolumn{1}{c|}{$3.02 e4$}  & 38    & \multicolumn{1}{c|}{22}       & 26    & \multicolumn{1}{c|}{10}       & 18    & \multicolumn{1}{c|}{4.8}      & 25    & \multicolumn{1}{c|}{10}       & 15    & \multicolumn{1}{c|}{3.4}      & 10          & 1.7            \\ \cline{1-2}
    \multicolumn{1}{|c|}{5}   & \multicolumn{1}{c|}{$8.42 e4$}  & 44    & \multicolumn{1}{c|}{25}       & 29    & \multicolumn{1}{c|}{12}       & 20    & \multicolumn{1}{c|}{5.5}      & 29    & \multicolumn{1}{c|}{14}       & 16    & \multicolumn{1}{c|}{4.0}      & 10          & 1.9            \\ \cline{1-2}
    \multicolumn{1}{|c|}{6}   & \multicolumn{1}{c|}{$2.26 e5$}  & 46    & \multicolumn{1}{c|}{29}       & 30    & \multicolumn{1}{c|}{13}       & 21    & \multicolumn{1}{c|}{6.3}      & 31    & \multicolumn{1}{c|}{16}       & 18    & \multicolumn{1}{c|}{4.7}      & 11          & 2.1            \\ \cline{1-2}
    \multicolumn{1}{|c|}{7}   & \multicolumn{1}{c|}{$6.01 e5$}  & 49    & \multicolumn{1}{c|}{35}       & 32    & \multicolumn{1}{c|}{18}       & 22    & \multicolumn{1}{c|}{8.6}      & 35    & \multicolumn{1}{c|}{22}       & 19    & \multicolumn{1}{c|}{6.2}      & 12          & 2.7            \\ \cline{1-2}
    \multicolumn{1}{|c|}{8}   & \multicolumn{1}{c|}{$1.54 e6$}  & 50    & \multicolumn{1}{c|}{35}       & 32    & \multicolumn{1}{c|}{17}       & 22    & \multicolumn{1}{c|}{8.1}      & 36    & \multicolumn{1}{c|}{23}       & 19    & \multicolumn{1}{c|}{7.0}      & 12          & 2.8            \\ \hline
    \end{tabular}
\caption{PCG iteration counts for V-cycle multigrid preconditioner
with different smoothers and smoothing steps in Example \ref{subsec:diffConvergeNum}.}
\label{tab:poissonSm}
\end{table}

\subsubsection{\red{Non-convex domain and jump diffusion coefficients on the coarsest mesh}}
\label{subsec:diffJumpNum}
In this example, we consider the reaction-diffusion equation with jump diffusion coefficients \red{on the coarsest mesh in a non-convex domain}. 
When $d = 2$, we denote $\Omega_1^{2D}$ as the domain formed by connecting points (0.5, 0.15), (0.65, 0.3), (0.5, 0.45), and (0.35, 0.3) in order. We then define $\Omega_2^{2D}:=([0,\; 1]\times[0,\; 0.6]) \backslash \Omega_{1}^{2D}$, and $\Omega_3^{2D}:=[0.2,\; 0.8]\times[0.6,\; 0.8]$. The domain $\Omega$ consists of these three subdomains, i.e. $\Omega = \Omega_1^{2D}\cup\Omega_2^{2D}\cup\Omega_3^{2D}$ as illustrated in the left panel of Fig. \ref{fig:chipDiff}. When $d = 3$, $\Omega_1^{3D} := \Omega_1^{2D}\times[0.225,\; 0.375]$, $\Omega_2^{3D} := \Omega_2^{2D}\times[0,\; 0.6]$, $\Omega_3^{3D} := \Omega_3^{2D}\times[0.2,\; 0.4]$, and $\Omega = \Omega_1^{3D}\cup\Omega_2^{3D}\cup\Omega_3^{3D}$. The diffusion coefficient $\alpha$ differs in different subdomains, which is set as 10 in $\Omega_1^{2D}/\Omega_1^{3D}$, 1 in $\Omega_2^{2D}/\Omega_2^{3D}$, and 1000 in $\Omega_3^{2D}/\Omega_3^{3D}$. The source term $f = 1$ in the center subdomain $\Omega_1^{2D}/\Omega_1^{3D}$, and $f = 0$ elsewhere. We assume a homogeneous Dirichlet boundary condition on the bottom and homogeneous Neumann boundary conditions on all other boundaries. We use the standard V-cycle multigrid algorithms with point Gauss-Seidel as smoother to precondition the conjugate gradient solver. The coarsest mesh is a triangulation of $\Omega$ respecting the subdomain boundaries, with the maximum element diameter not exceeding $1/4$ in both two-dimensional and three-dimensional cases.

\red{We first test the algorithms under uniform mesh refinements,}
and 
Table \ref{tab:chipDiffUniform} reports the PCG iteration counts 
with different mesh levels $J$, reaction coefficient $\beta$, and smoothing steps $m$ in both two-dimensional and three-dimensional cases. 
As observed from the results when $d = 3$, one-step and two-step Gauss-Seidel smoother are not enough to let Theorem \ref{thm:mg} hold, and the iteration counts grow with increasing facet DOFs. Other results verify the robustness of our preconditioner with respect to mesh size and mesh levels when the smoothing step $m$ is large enough \red{in the low-regularity cases}.  The iteration counts remain almost unchanged as $\beta$ varies. We note that when smoothing steps are the same, the iteration counts here with jump coefficients in the reaction-diffusion equation are obviously larger than in the previous case with mildly changing coefficients.

\red{
We next test with $\beta=0$ on adaptively refined mesh using a recovery-based error estimator\cite{cai2010recovery}. All other settings are the same. The right panel of Fig. \ref{fig:chipDiff} demonstrates the finest two-dimensional mesh. Fig. \ref{fig:chipDiffAdapt} reports the obtained PCG iteration counts, and similar results as those on uniformly-refined meshes are observed.
}

\begin{figure}[ht]
    \centering
    \includegraphics[height = .15\textheight]{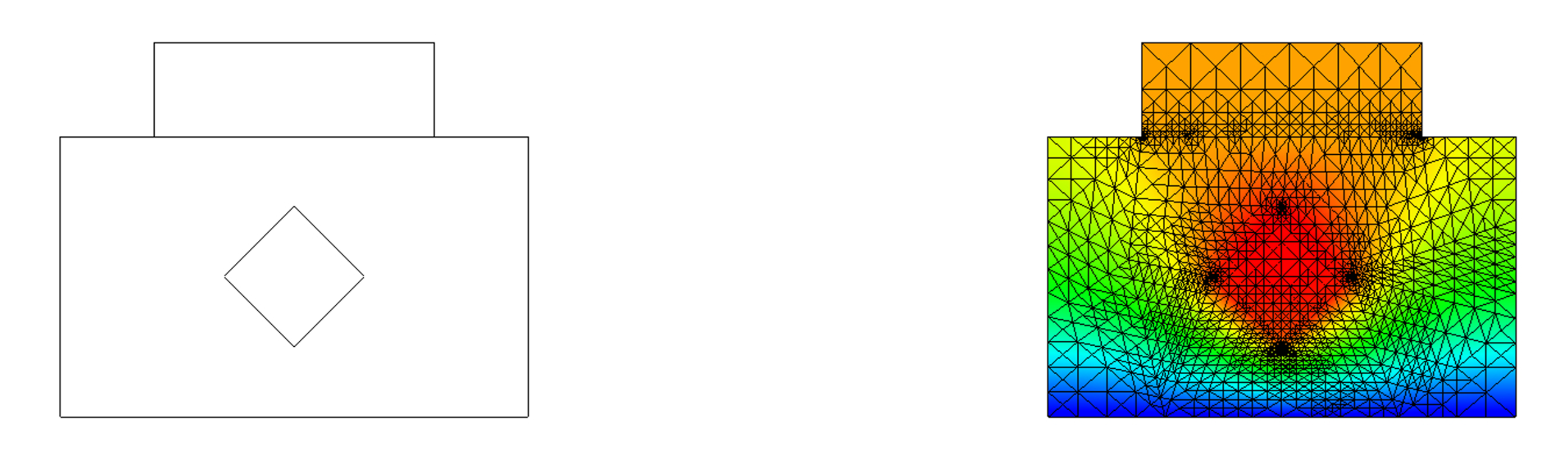}
    \caption{The domain of reaction-diffusion equation with jump diffusion coefficients on the coarsest mesh in Example \ref{subsec:diffJumpNum} (left), 
    \red{and the top-level adaptively refined mesh in two dimensions with $\beta=0$ (right).}}
    \label{fig:chipDiff}
\end{figure}

\begin{table}[ht]
    \setlength{\tabcolsep}{6pt}
	\begin{tabular}{|ccccccccccc|}
        \hline
        \multicolumn{11}{|c|}{$d=2$} \\ \hline
        \multicolumn{2}{|c|}{\multirow{2}{*}{Parameters}} & \multicolumn{3}{c|}{$\beta = 1000$} & \multicolumn{3}{c|}{$\beta = 1$} & \multicolumn{3}{c|}{$\beta =0$} \\ \cline{3-11} 
        \multicolumn{2}{|c|}{} & \multicolumn{1}{c|}{$m = 1$} & \multicolumn{1}{c|}{$m = 2$} & \multicolumn{1}{c|}{$m = 4$} & \multicolumn{1}{c|}{$m = 1$} & \multicolumn{1}{c|}{$m = 2$} & \multicolumn{1}{c|}{$m = 4$} & \multicolumn{1}{c|}{$m = 1$} & \multicolumn{1}{c|}{$m = 2$} & $m = 4$ \\ \hline
        \multicolumn{1}{|c|}{$J$} & \multicolumn{1}{c|}{Facet DOFs} & \multicolumn{9}{c|}{Iteration Count} \\ \hline
        \multicolumn{1}{|c|}{2} & \multicolumn{1}{c|}{$2.03 e2$} & 19 & 11 & \multicolumn{1}{c|}{8} & 21 & 13 & \multicolumn{1}{c|}{10} & 21 & 14 & 10 \\ \cline{1-2}
        \multicolumn{1}{|c|}{3} & \multicolumn{1}{c|}{$7.78 e2$} & 28 & 17 & \multicolumn{1}{c|}{10} & 34 & 19 & \multicolumn{1}{c|}{11} & 34 & 19 & 11 \\ \cline{1-2}
        \multicolumn{1}{|c|}{4} & \multicolumn{1}{c|}{$3.04 e3$} & 42 & 24 & \multicolumn{1}{c|}{11} & 44 & 27 & \multicolumn{1}{c|}{13} & 44 & 27 & 13 \\ \cline{1-2}
        \multicolumn{1}{|c|}{5} & \multicolumn{1}{c|}{$1.20 e4$} & 59 & 28 & \multicolumn{1}{c|}{12} & 61 & 31 & \multicolumn{1}{c|}{14} & 61 & 31 & 14 \\ \cline{1-2}
        \multicolumn{1}{|c|}{6} & \multicolumn{1}{c|}{$4.79 e4$} & 67 & 28 & \multicolumn{1}{c|}{12} & 72 & 31 & \multicolumn{1}{c|}{14} & 72 & 31 & 14 \\ \cline{1-2}
        \multicolumn{1}{|c|}{7} & \multicolumn{1}{c|}{$1.91 e5$} & 69 & 28 & \multicolumn{1}{c|}{12} & 73 & 31 & \multicolumn{1}{c|}{14} & 73 & 31 & 14 \\ \cline{1-2}
        \multicolumn{1}{|c|}{8} & \multicolumn{1}{c|}{$7.63 e5$} & 69 & 28 & \multicolumn{1}{c|}{11} & 72 & 31 & \multicolumn{1}{c|}{14} & 73 & 31 & 14 \\ \hline \hline
        \multicolumn{11}{|c|}{$d=3$} \\ \hline
        \multicolumn{2}{|c|}{\multirow{2}{*}{Parameters}} & \multicolumn{3}{c|}{$\beta=1000$} & \multicolumn{3}{c|}{$\beta=1$} & \multicolumn{3}{c|}{$\beta=0$} \\ \cline{3-11} 
        \multicolumn{2}{|c|}{} & \multicolumn{1}{c|}{$m = 1$} & \multicolumn{1}{c|}{$m = 2$} & \multicolumn{1}{c|}{$m = 4$} & \multicolumn{1}{c|}{$m = 1$} & \multicolumn{1}{c|}{$m = 2$} & \multicolumn{1}{c|}{$m = 4$} & \multicolumn{1}{c|}{$m = 1$} & \multicolumn{1}{c|}{$m = 2$} & $m = 4$ \\ \hline
        \multicolumn{1}{|c|}{$J$} & \multicolumn{1}{c|}{Facet DOFs} & \multicolumn{9}{c|}{Iteration Count} \\ \hline
        \multicolumn{1}{|c|}{2} & \multicolumn{1}{c|}{$1.20 e3$} & 16 & 10 & \multicolumn{1}{c|}{7} & 24 & 14 & \multicolumn{1}{c|}{10} & 24 & 14 & 10 \\ \cline{1-2}
        \multicolumn{1}{|c|}{3} & \multicolumn{1}{c|}{$3.73 e3$} & 22 & 12 & \multicolumn{1}{c|}{8} & 29 & 16 & \multicolumn{1}{c|}{11} & 29 & 16 & 11 \\ \cline{1-2}
        \multicolumn{1}{|c|}{4} & \multicolumn{1}{c|}{$1.10 e4$} & 30 & 16 & \multicolumn{1}{c|}{10} & 38 & 21 & \multicolumn{1}{c|}{13} & 38 & 21 & 13 \\ \cline{1-2}
        \multicolumn{1}{|c|}{5} & \multicolumn{1}{c|}{$3.04 e4$} & 45 & 24 & \multicolumn{1}{c|}{14} & 55 & 29 & \multicolumn{1}{c|}{16} & 54 & 28 & 16 \\ \cline{1-2}
        \multicolumn{1}{|c|}{6} & \multicolumn{1}{c|}{$8.10 e4$} & 66 & 30 & \multicolumn{1}{c|}{14} & 73 & 34 & \multicolumn{1}{c|}{17} & 73 & 34 & 17 \\ \cline{1-2}
        \multicolumn{1}{|c|}{7} & \multicolumn{1}{c|}{$2.14 e5$} & 89 & 39 & \multicolumn{1}{c|}{18} & 98 & 44 & \multicolumn{1}{c|}{20} & 98 & 44 & 20 \\ \cline{1-2}
        \multicolumn{1}{|c|}{8} & \multicolumn{1}{c|}{$5.43 e5$} & 107 & 42 & \multicolumn{1}{c|}{16} & 117 & 47 & \multicolumn{1}{c|}{19} & 119 & 46 & 19 \\ \hline
        \end{tabular}
    \caption{PCG iteration counts for V-cycle multigrid preconditioner with Gauss-Seidel smoother \red{and uniform mesh refinement}
    for the reaction-diffusion equation with jump coefficients in Example \ref{subsec:diffJumpNum}.}
    \label{tab:chipDiffUniform}
\end{table}

\begin{figure}[ht]
    \centering
    \includegraphics[width=.75\textwidth]{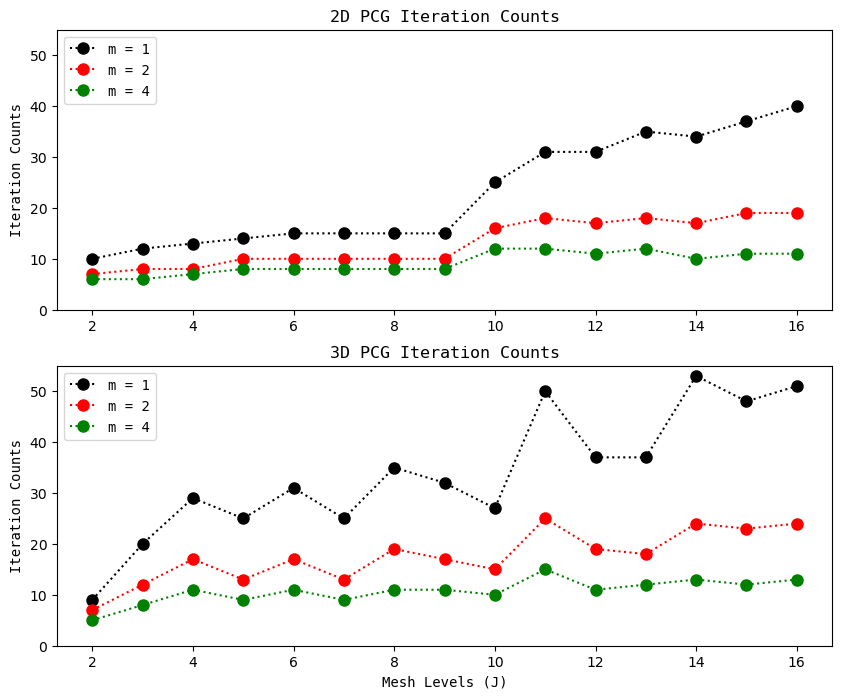}
    \caption{PCG iteration counts for V-cycle multigrid preconditioner with Gauss-Seidel smoother \red{and adaptive mesh refinement}
    for the reaction-diffusion equation with jump coefficients in Example \ref{subsec:diffJumpNum}.}
    \label{fig:chipDiffAdapt}
\end{figure}

\red{
\subsubsection{Jump diffusion coefficient on the finest mesh}
\label{subsec:ellipFineJumpNum}
We consider a reaction-diffusion equation with jump diffusion coefficient on the finest mesh level. For simplicity, we limit ourselves to the two-dimensional case and set the domain as a unit square $\Omega=[0, 1]^2$ with homogeneous Dirichlet boundary conditions on all sides. The source term is set as $f = 1$. The coarsest mesh is a structured triangulation of the domain $\Omega$ with $h=1/4$ and the hierarchical meshes are obtained by successive uniform refinement. We let the reaction coefficient $\beta = 1$, and the diffusion coefficient $\alpha$ changes alternatively on the finest mesh level $\mathcal{T}_J$ between the minimum value $\alpha_0$ ($\alpha_0 = 1$) and the maximum value $\alpha_1$, as demonstrated in the left panel of Fig. \ref{fig:diffChess}. We define the ratio $\rho:= \alpha_1/\alpha_0$. We note that the $\alpha_l^{-1}$ in the {\sf HDG-P0} scheme \eqref{ellipWeak0} requires the projection of $\alpha^{-1}$ onto piecewise constant space $W_l$ on each mesh level, therefore in this numerical experiment $\alpha_l^{-1}$ becomes global constant on coarser mesh levels for $l = 1,\dots, J-1$ and equals $2/(\rho + 1)$. 

We use the standard V-cycle multigrid algorithms with two-step point Gauss-Seidel as smoother to precondition the conjugate gradient solver, and Table \ref{tab:diffChess} reports the PCG iteration counts and the condition number with different ratio $\rho$ and different mesh levels $J$. The robustness of our proposed multigrid method with respect to mesh size and mesh levels is observed. However, as the ratio $\rho$ increases, the iteration counts and the conditioner numbers grow and are not bounded. We refer to \cite{zhu2014analysis, kolev2016multilevel} for further studies on multigrid for CR discretization with jump coefficients, where conforming piecewise linear finite element was used as the auxiliary space to construct multigrid methods.
}

\begin{figure}[ht]
    \centering
    \includegraphics[height = .15\textheight]{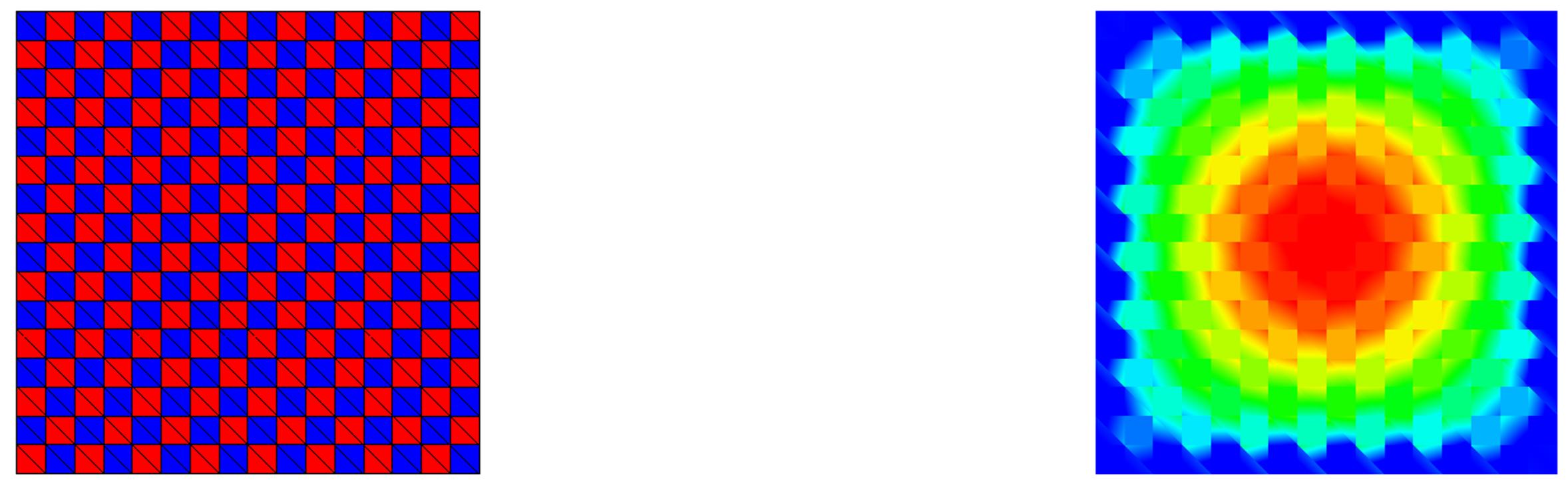}
    \caption{\red{The diffusion coefficient on the finest mesh level in Example \ref{subsec:ellipFineJumpNum} (left), 
    and the numerical solution with $J=3$ and $\rho=1e4$ (right).}}
    \label{fig:diffChess}
\end{figure}

\begin{table}[ht]
    \red{
    \begin{tabular}{|c|cccccccccc|}
    \hline
    \multirow{2}{*}{Parameters} & \multicolumn{10}{c|}{$\rho$}                                                                                                                                                             \\ \cline{2-11} 
                                & \multicolumn{2}{c|}{2}                & \multicolumn{2}{c|}{1e1}              & \multicolumn{2}{c|}{1e2}              & \multicolumn{2}{c|}{1e4}              & \multicolumn{2}{c|}{1e6} \\ \hline
    $J$                         & \# It & \multicolumn{1}{c|}{$\kappa$} & \# It & \multicolumn{1}{c|}{$\kappa$} & \# It & \multicolumn{1}{c|}{$\kappa$} & \# It & \multicolumn{1}{c|}{$\kappa$} & \# It     & $\kappa$     \\ \hline
    2                           & 9     & \multicolumn{1}{c|}{1.6}      & 16    & \multicolumn{1}{c|}{3.8}      & 22    & \multicolumn{1}{c|}{3.1e1}    & 32    & \multicolumn{1}{c|}{2.6e3}    & 37        & 2.4e5        \\ \cline{1-1}
    3                           & 10    & \multicolumn{1}{c|}{1.7}      & 17    & \multicolumn{1}{c|}{4.1}      & 28    & \multicolumn{1}{c|}{3.3e1}    & 39    & \multicolumn{1}{c|}{2.8e3}    & 49        & 2.4e5        \\ \cline{1-1}
    4                           & 10    & \multicolumn{1}{c|}{1.8}      & 17    & \multicolumn{1}{c|}{4.3}      & 31    & \multicolumn{1}{c|}{3.5e1}    & 47    & \multicolumn{1}{c|}{3.1e3}    & 58        & 2.4e5        \\ \cline{1-1}
    5                           & 11    & \multicolumn{1}{c|}{1.8}      & 17    & \multicolumn{1}{c|}{4.3}      & 31    & \multicolumn{1}{c|}{3.5e1}    & 54    & \multicolumn{1}{c|}{3.2e3}    & 69        & 2.5e5        \\ \cline{1-1}
    6                           & 11    & \multicolumn{1}{c|}{1.8}      & 17    & \multicolumn{1}{c|}{4.3}      & 31    & \multicolumn{1}{c|}{3.5e1}    & 56    & \multicolumn{1}{c|}{3.2e3}    & 71        & 2.5e5        \\ \cline{1-1}
    7                           & 11    & \multicolumn{1}{c|}{1.8}      & 17    & \multicolumn{1}{c|}{4.3}      & 31    & \multicolumn{1}{c|}{3.5e1}    & 56    & \multicolumn{1}{c|}{3.4e3}    & 72        & 2.5e5        \\ \cline{1-1}
    8                           & 11    & \multicolumn{1}{c|}{1.8}      & 17    & \multicolumn{1}{c|}{4.3}      & 31    & \multicolumn{1}{c|}{3.5e1}    & 56    & \multicolumn{1}{c|}{3.4e3}    & 73        & 2.5e5        \\ \hline
    \end{tabular}}
    \caption{PCG iteration counts and condition numbers for V-cycle multigrid preconditioner with Gauss-Seidel smoother for the reaction-diffusion equation with jump coefficients on the finest mesh level in Example \ref{subsec:ellipFineJumpNum}.}
    \label{tab:diffChess}
\end{table}

\subsection{Generalized Stokes equations}
\subsubsection{Manufactured solution}
\label{subsec:stokesConvergeNum}
We first verify the optimal convergence rates of the {\sf HDG-P0} scheme \eqref{stokesWeak0} for the generalized Stokes equations with known solutions.
We set the coefficients $\mu = 1$, $\beta = 10$ and the exact solution 
\begin{align*}
	&\left.
		\begin{array}{l l}
			u_x & = x^2 (x - 1)^2 2y(1 - y)(2y - 1) \\
			u_y & = y^2 (y - 1)^2 2 x(x - 1)(2x - 1)\\
			p & = x(1 - x)(1 - y) - 1/12
		\end{array}
	\right\} \;\text{when $d = 2$,}
\end{align*}
and
\begin{align*}
	&\left.
		\begin{array}{l l}
			u_x & = x^2(x - 1)^2 (2 y - 6y^2 + 4 y^3)(2 z - 6z^2 + 4 z^3) \\
			u_y & = y^2 (y - 1)^2 (2x - 6x^2 + 4x^3) (2z - 6z^2 + 4z^3)\\
			u_z & = -2 z^2 (z - 1)^2 (2x - 6x^2 + 4x^3) (2y - 6y^2 + 4y^3)\\
			p & = x(1 - x)(1 - y)(1 - z) - 1/24
		\end{array}
	\right\} \;\text{when $d = 3$,}
\end{align*}
with all other settings the same as in Example \ref{subsec:diffConvergeNum}.
\red{Table \ref{tab:stokesRate} reports the values and the corresponding EOC of the discrete $L_2$ norms $\|u_h - u\|_{0}$, $\|\nabla\cdot\underline{u}_h\|_0$, and $\|\dunderline{L_h} - \dunderline{L}\|_{0}$. Optimal convergence rates of $\underline{u}_h$ and $\dunderline{L}_h$ are observed. Since the divergence-free constraint on velocity is imposed weakly in the {\sf HDG-P0} scheme, and the normal components of $u_h$ are not necessarily continuous across mesh facets, the obtained $u_h$ is not exactly divergence-free and the divergence error has first order convergence rate.}

\red{
We have also tested the W-cycle multigrid method in Algorithm \ref{alg:stokesMG} as the iteration solver for the problem. The coarsest mesh is a triangulation of $\Omega$ with the maximum element diameter less than $1/4$ when $d=2$ and less than $1/2$ when $d=3$, followed by uniform refinement. Table \ref{tab:stokesSolver} reports the obtained iteration counts with different smothers and smoothing steps in the multigrid algorithm, where we denote B-JAC as the damped vertex-block Jacobi smoother with the damping parameter set as 0.4, and B-GS as the vertex-block Gauss-Seidel smoother. As observed from the results, when the smoothing steps are not large enough, the operator $\underline{\mathbb{E}}_{l,m}$ in Theorem \ref{thm:stokesMG} is no longer a reducer, and the needed smoothing steps in three dimensions for the W-cycle iteration to converge are larger than those in two dimensions. Other results verify the robustness of the W-cycle iteration solver with respect to the mesh size and mesh level.
}


\begin{table}[ht]
\red{
    \begin{tabular}{|c|cccccc|}
    \hline
    \multirow{2}{*}{Level} & \multicolumn{6}{c|}{$d=2$}                                                                                                                                                                                                                       \\ \cline{2-7} 
                           & $\|\underline{u}_h - \underline{u}\|_{0}$ & \multicolumn{1}{c|}{EOC}          & $\|\nabla\cdot\underline{u}_h\|_{0}$ & \multicolumn{1}{c|}{EOC}          & $\|\dunderline{L}_h - \dunderline{L}\|_{0}$ & EOC          \\ \hline
    1                      & $5.87 e{-3}$                              & \multicolumn{1}{c|}{$\backslash$} & $1.76 e{-2}$                                                    & \multicolumn{1}{c|}{$\backslash$} & $3.37 e{-2}$                                & $\backslash$ \\ \cline{1-1}
    2                      & $1.67 e{-3}$                              & \multicolumn{1}{c|}{1.81}         & $9.56 e{-3}$                                                    & \multicolumn{1}{c|}{0.88}         & $1.94 e{-2}$                                & 0.79         \\ \cline{1-1}
    3                      & $4.42 e{-4}$                              & \multicolumn{1}{c|}{1.92}         & $4.94 e{-3}$                                                    & \multicolumn{1}{c|}{0.95}         & $1.02 e{-2}$                                & 0.93         \\ \cline{1-1}
    4                      & $1.13 e{-4}$                              & \multicolumn{1}{c|}{1.97}         & $2.49 e{-3}$                                                    & \multicolumn{1}{c|}{0.99}         & $5.19 e{-3}$                                & 0.98         \\ \cline{1-1}
    5                      & $2.83 e{-5}$                              & \multicolumn{1}{c|}{1.99}         & $1.25 e{-3}$                                                    & \multicolumn{1}{c|}{1.00}         & $2.61 e{-3}$                                & 0.99         \\ \hline\hline
    \multirow{2}{*}{Level} & \multicolumn{6}{c|}{$d=3$}                                                                                                                                                                                                                       \\ \cline{2-7} 
                           & $\|\underline{u}_h - \underline{u}\|_{0}$ & \multicolumn{1}{c|}{EOC}          & $\|\nabla\cdot\underline{u}_h\|_{0}$ & \multicolumn{1}{c|}{EOC}          & $\|\dunderline{L}_h - \dunderline{L}\|_{0}$ & EOC          \\ \hline
    1                      & $1.74 e{-3}$                              & \multicolumn{1}{c|}{$\backslash$} & $6.53 e{-3}$                                                    & \multicolumn{1}{c|}{$\backslash$} & $1.34 e{-2}$                                & $\backslash$ \\ \cline{1-1}
    2                      & $5.03 e{-4}$                              & \multicolumn{1}{c|}{1.79}         & $3.20e{-3}$                                                     & \multicolumn{1}{c|}{1.03}         & $8.05 e{-3}$                                & 0.74         \\ \cline{1-1}
    3                      & $1.35 e{-4}$                              & \multicolumn{1}{c|}{1.90}         & $1.59 e{-3}$                                                    & \multicolumn{1}{c|}{1.01}         & $4.29 e{-3}$                                & 0.91         \\ \cline{1-1}
    4                      & $3.45 e{-5}$                              & \multicolumn{1}{c|}{1.96}         & $7.94 e{-4}$                                                    & \multicolumn{1}{c|}{1.00}         & $2.19 e{-3}$                                & 0.97         \\ \cline{1-1}
    5                      & $8.69 e{-6}$                              & \multicolumn{1}{c|}{1.99}         & $3.97 e{-4}$                                                    & \multicolumn{1}{c|}{1.00}         & $1.10 e{-3}$                                & 0.99         \\ \hline
    \end{tabular}
    \caption{Estimated convergence rates of the {\sf HDG-P0} scheme solved by one-step augmented Lagrangian Uzawa iteration in Example \ref{subsec:stokesConvergeNum}.}
    \label{tab:stokesRate}}
\end{table}

\begin{table}[ht]
    \red{
    \begin{tabular}{|cccccccc|}
    \hline
    \multicolumn{8}{|c|}{$d=2$}                                                                                                                  \\ \hline
    \multicolumn{2}{|c|}{\multirow{2}{*}{Parameters}}           & \multicolumn{3}{c|}{B-JAC}                       & \multicolumn{3}{c|}{B-GS}   \\ \cline{3-8} 
    \multicolumn{2}{|c|}{}    & $m = 1$ & $m = 2$ & \multicolumn{1}{c|}{$m = 4$} & $m = 1$ & $m = 2$ & $m = 4$ \\ \hline
    \multicolumn{1}{|c|}{$J$} & \multicolumn{1}{c|}{Facet DOFs} & \multicolumn{6}{c|}{Iteration Count}                                           \\ \hline
    \multicolumn{1}{|c|}{2}   & \multicolumn{1}{c|}{$3.76 e2$}  & 53      & 29      & \multicolumn{1}{c|}{20}      & 27      & 17      & 12      \\ \cline{1-2}
    \multicolumn{1}{|c|}{3}   & \multicolumn{1}{c|}{$1.57 e3$}  & 22      & 20      & \multicolumn{1}{c|}{17}      & 17      & 13      & 11      \\ \cline{1-2}
    \multicolumn{1}{|c|}{4}   & \multicolumn{1}{c|}{$6.40 e3$}  & 46      & 36      & \multicolumn{1}{c|}{25}      & 30      & 16      & 11      \\ \cline{1-2}
    \multicolumn{1}{|c|}{5}   & \multicolumn{1}{c|}{$2.59 e4$}  & N/A     & 62      & \multicolumn{1}{c|}{31}      & 26      & 14      & 11      \\ \cline{1-2}
    \multicolumn{1}{|c|}{6}   & \multicolumn{1}{c|}{$1.04 e5$}  & N/A     & 78      & \multicolumn{1}{c|}{36}      & 33      & 15      & 10      \\ \cline{1-2}
    \multicolumn{1}{|c|}{7}   & \multicolumn{1}{c|}{$4.17 e5$}  & N/A     & 72      & \multicolumn{1}{c|}{31}      & 30      & 11      & 8       \\ \cline{1-2}
    \multicolumn{1}{|c|}{8}   & \multicolumn{1}{c|}{$1.67 e6$}  & N/A     & 86      & \multicolumn{1}{c|}{24}      & 21      & 11      & 7       \\ \hline\hline
    \multicolumn{8}{|c|}{$d=3$}                                                                                                                  \\ \hline
    \multicolumn{2}{|c|}{\multirow{2}{*}{Parameters}}           & \multicolumn{3}{c|}{B-JAC}                       & \multicolumn{3}{c|}{B-GS}   \\ \cline{3-8} 
    \multicolumn{2}{|c|}{}    & $m = 1$ & $m = 4$ & \multicolumn{1}{c|}{$m = 8$} & $m = 1$ & $m = 4$ & $m = 8$ \\ \hline
    \multicolumn{1}{|c|}{$J$} & \multicolumn{1}{c|}{Facet DOFs} & \multicolumn{6}{c|}{Iteration Count}                                           \\ \hline
    \multicolumn{1}{|c|}{2}   & \multicolumn{1}{c|}{$1.41 e3$}  & N/A     & 23      & \multicolumn{1}{c|}{9}       & 10      & 5       & 4       \\ \cline{1-2}
    \multicolumn{1}{|c|}{3}   & \multicolumn{1}{c|}{$4.61 e3$}  & N/A     & 8       & \multicolumn{1}{c|}{6}       & 64      & 5       & 4       \\ \cline{1-2}
    \multicolumn{1}{|c|}{4}   & \multicolumn{1}{c|}{$1.44 e4$}  & N/A     & 18      & \multicolumn{1}{c|}{9}       & N/A     & 7       & 6       \\ \cline{1-2}
    \multicolumn{1}{|c|}{5}   & \multicolumn{1}{c|}{$3.89 e4$}  & N/A     & 8       & \multicolumn{1}{c|}{6}       & N/A     & 6       & 6       \\ \cline{1-2}
    \multicolumn{1}{|c|}{6}   & \multicolumn{1}{c|}{$1.06 e5$}  & N/A     & 13      & \multicolumn{1}{c|}{7}       & N/A     & 7       & 6       \\ \cline{1-2}
    \multicolumn{1}{|c|}{7}   & \multicolumn{1}{c|}{$2.74 e5$}  & N/A     & 9       & \multicolumn{1}{c|}{6}       & N/A     & 6       & 5       \\ \cline{1-2}
    \multicolumn{1}{|c|}{8}   & \multicolumn{1}{c|}{$7.00 e5$}  & N/A     & 14      & \multicolumn{1}{c|}{6}       & N/A     & 6       & 5       \\ \hline
    \end{tabular}
    }
\caption{Iteration counts of W-cycle multigrid iteration solver
with different smoothers and smoothing steps in Example \ref{subsec:stokesConvergeNum}.}
\label{tab:stokesSolver}
\end{table}

\subsubsection{Lid-driven cavity}
\label{subsec:stokesLidNum}
In this example, we test the robustness of the proposed multigrid preconditioners for the 
the lid-driven cavity problem, where the computational domain is a unit square/cube $\Omega=[0,1]^d$. We assume an inhomogeneous Dirichlet boundary condition $\underline{u}=[4x(1-x),\; 0]^\trans$ when $d=2$, or $\underline{u}=[16x(1-x)y(1-y),\; 0,\; 0]^\trans$ when $d=3$ on the top side, and no-slip boundary conditions on the remaining domain boundaries. The source term $\underline{f}=\underline{0}$. 
We use the PCG method preconditioned by W-cycle and variable V-cycle  multigrid (with smoothing steps $m(l)=2^{J-l}m(J)$) in Algorithm \ref{alg:stokesMG} to solve the augmented Lagrangian Uzawa iteration for the condensed {\sf HDG-P0} scheme \eqref{stOpEq-ALU1}. We use vertex-patched block Gauss-Seidel as smoothers in multigrid algorithms to avoid damping parameters.
The coarsest mesh is a triangulation of $\Omega$ with the maximum element diameter less than $1/4$ when $d=2$, or less than $1/2$ when $d=3$.

Table \ref{tab:stokesLid} reports the PCG iteration counts \red{on uniformly refined meshes}
with different domain dimensions $d$, mesh levels $J$, the low order term coefficient $\beta$, and the smoothing step on the finest mesh $m(J)$. As observed from the results, when $d=2$ and $\beta=1000$, the solver preconditioned by the W-cycle algorithm with $m(J)=2$ fails. This is due to the fact that the linear operator of the W-cycle multigrid with only two smoothing steps on each level becomes indefinite and cannot be used as a preconditioner, see 
\cite[Section 4]{bramble1991analysis}. On the other hand, variable V-cycle multigrid is more robust and works with $m(J)=1$. Other results verify the robustness of our multigrid preconditioner with respect to mesh size and mesh levels.

\begin{table}[ht]
    \begin{tabular}{|cccccccccccccc|}
    \hline
    \multicolumn{14}{|c|}{$d=2$} \\ \hline
    \multicolumn{2}{|c|}{Multigrid} & \multicolumn{6}{c|}{Variable V cycle} & \multicolumn{6}{c|}{W cycle} \\ \hline
    \multicolumn{2}{|c|}{$\beta$} & \multicolumn{2}{c|}{1000} & \multicolumn{2}{c|}{1} & \multicolumn{2}{c|}{0} & \multicolumn{2}{c|}{1000} & \multicolumn{2}{c|}{1} & \multicolumn{2}{c|}{0} \\ \hline
    \multicolumn{2}{|c|}{$m(J)$} & 1 & \multicolumn{1}{c|}{2} & 1 & \multicolumn{1}{c|}{2} & 1 & \multicolumn{1}{c|}{2} & 2 & \multicolumn{1}{c|}{4} & 2 & \multicolumn{1}{c|}{4} & 2 & 4 \\ \hline
    \multicolumn{1}{|c|}{$J$} & \multicolumn{1}{c|}{Facet DOFs} & \multicolumn{12}{c|}{Iteration Counts} \\ \hline
    \multicolumn{1}{|c|}{2} & \multicolumn{1}{c|}{$3.76 e2$} & 13 & 10 & 12 & 10 & 12 & \multicolumn{1}{c|}{10} & N/A & 8 & 10 & 8 & 10 & 8 \\ \cline{1-2}
    \multicolumn{1}{|c|}{3} & \multicolumn{1}{c|}{$1.57 e3$} & 18 & 14 & 15 & 12 & 15 & \multicolumn{1}{c|}{12} & N/A & 12 & 10 & 9 & 10 & 9 \\ \cline{1-2}
    \multicolumn{1}{|c|}{4} & \multicolumn{1}{c|}{$6.40 e3$} & 20 & 16 & 17 & 13 & 17 & \multicolumn{1}{c|}{13} & N/A & 11 & 11 & 9 & 11 & 9 \\ \cline{1-2}
    \multicolumn{1}{|c|}{5} & \multicolumn{1}{c|}{$2.59 e4$} & 20 & 15 & 18 & 14 & 18 & \multicolumn{1}{c|}{14} & N/A & 9 & 11 & 10 & 11 & 10 \\ \cline{1-2}
    \multicolumn{1}{|c|}{6} & \multicolumn{1}{c|}{$1.04 e5$} & 20 & 15 & 19 & 15 & 19 & \multicolumn{1}{c|}{15} & N/A & 9 & 11 & 9 & 11 & 9 \\ \cline{1-2}
    \multicolumn{1}{|c|}{7} & \multicolumn{1}{c|}{$4.17 e5$} & 20 & 15 & 20 & 15 & 20 & \multicolumn{1}{c|}{15} & N/A & 9 & 11 & 9 & 11 & 9 \\ \cline{1-2}
    \multicolumn{1}{|c|}{8} & \multicolumn{1}{c|}{$1.67 e6$} & 21 & 15 & 21 & 15 & 21 & \multicolumn{1}{c|}{15} & N/A & 9 & 11 & 9 & 11 & 9 \\ \hline \hline
    \multicolumn{14}{|c|}{$d=3$} \\ \hline
    \multicolumn{2}{|c|}{Multigrid} & \multicolumn{6}{c|}{Variable V cycle} & \multicolumn{6}{c|}{W cycle} \\ \hline
    \multicolumn{2}{|c|}{$\beta$} & \multicolumn{2}{c|}{1000} & \multicolumn{2}{c|}{1} & \multicolumn{2}{c|}{0} & \multicolumn{2}{c|}{1000} & \multicolumn{2}{c|}{1} & \multicolumn{2}{c|}{0} \\ \hline
    \multicolumn{2}{|c|}{$m(J)$} & 1 & \multicolumn{1}{c|}{2} & 1 & \multicolumn{1}{c|}{2} & 1 & \multicolumn{1}{c|}{2} & 2 & \multicolumn{1}{c|}{4} & 2 & \multicolumn{1}{c|}{4} & 2 & 4 \\ \hline
    \multicolumn{1}{|c|}{$J$} & \multicolumn{1}{c|}{Facet DOFs} & \multicolumn{12}{c|}{Iteration Counts} \\ \hline
    \multicolumn{1}{|c|}{2} & \multicolumn{1}{c|}{$1.41 e3$} & 10 & 7 & 11 & 8 & 11 & \multicolumn{1}{c|}{8} & 7 & 5 & 8 & 6 & 8 & 6 \\ \cline{1-2}
    \multicolumn{1}{|c|}{3} & \multicolumn{1}{c|}{$4.61 e3$} & 11 & 9 & 11 & 9 & 11 & \multicolumn{1}{c|}{9} & 8 & 7 & 8 & 7 & 8 & 7 \\ \cline{1-2}
    \multicolumn{1}{|c|}{4} & \multicolumn{1}{c|}{$1.44 e4$} & 12 & 10 & 12 & 9 & 12 & \multicolumn{1}{c|}{9} & 9 & 8 & 9 & 8 & 9 & 8 \\ \cline{1-2}
    \multicolumn{1}{|c|}{5} & \multicolumn{1}{c|}{$3.89 e4$} & 13 & 11 & 12 & 10 & 12 & \multicolumn{1}{c|}{10} & 8 & 8 & 8 & 7 & 8 & 7 \\ \cline{1-2}
    \multicolumn{1}{|c|}{6} & \multicolumn{1}{c|}{$1.06 e5$} & 14 & 12 & 12 & 10 & 13 & \multicolumn{1}{c|}{10} & 8 & 8 & 8 & 7 & 8 & 7 \\ \cline{1-2}
    \multicolumn{1}{|c|}{7} & \multicolumn{1}{c|}{$2.74 e5$} & 14 & 12 & 12 & 10 & 12 & \multicolumn{1}{c|}{10} & 8 & 7 & 8 & 7 & 8 & 7 \\ \cline{1-2}
    \multicolumn{1}{|c|}{8} & \multicolumn{1}{c|}{$7.00 e5$} & 14 & 12 & 13 & 10 & 13 & \multicolumn{1}{c|}{10} & 8 & 8 & 9 & 7 & 9 & 7 \\  \hline
    \end{tabular}
    \caption{Iteration counts of the preconditioned conjugate gradient solver for the generalized Stokes equations in lid-driven cavity problems in Example \ref{subsec:stokesLidNum}.}
    \label{tab:stokesLid}
\end{table}

\subsubsection{Backward-facing step flow}
\label{subsec:stokesBackNum}
Finally we test the robustness of the proposed multigrid preconditioners for the 
backward-facing step flow problem, with the \red{non-convex L-shaped domain} $\Omega=([0.5,5]\times[0,0.5])\cup([0,5]\times[0.5,1])$ when $d=2$, or $\Omega=\left(([0.5,5]\times[0,0.5])\cup([0,5]\times[0.5,1])\right)\times[0,1]$ when $d=3$. We assume an inhomogeneous Dirichlet boundary condition $\underline{u}=[16(1-y)(y-0.5),\; 0]^\trans$ when $d=2$, or $\underline{u}=[64(1-y)(y-0.5)z(1-z),\; 0,\; 0]^\trans$ when $d=3$ for the inlet flow on $\{x=0\}$, with do-nothing boundary condition on $\{x=5\}$ and no-slip boundary conditions on the remaining sides.
The maximum element diameter is less than $1/2$ in both two-dimensional and three-dimensional cases. Other settings are the same as in Example \ref{subsec:stokesLidNum}. 

Table \ref{tab:stokesBack} reports the PCG iteration counts
\red{under uniform mesh refinements}.
Similar results are observed as in the previous example.
In particular, variable V-cycle algorithm with $m(J)=1$ leads to a robust preconditioner.
\red{Next we test with $\beta=0$ on adaptively refined meshes using a  recovery-based error estimator as in Example \ref{subsec:diffJumpNum}. 
Fig. \ref{fig:stokesBack-adapt} reports the obtained PCG iteration counts, and similar results as those on the uniformly-refined meshes are observed.
We note that although the multigrid analysis in Theorem \ref{thm:stokesMG} requires full elliptic regularity, our proposed multigrid method works well in this low-regularity case.}



\begin{table}[ht]
    \begin{tabular}{|cccccccccccccc|}
    \hline
    \multicolumn{14}{|c|}{$d=2$} \\ \hline
    \multicolumn{2}{|c|}{Multigrid} & \multicolumn{6}{c|}{Variable V cycle} & \multicolumn{6}{c|}{W cycle} \\ \hline
    \multicolumn{2}{|c|}{$\beta$} & \multicolumn{2}{c|}{1000} & \multicolumn{2}{c|}{1} & \multicolumn{2}{c|}{0} & \multicolumn{2}{c|}{1000} & \multicolumn{2}{c|}{1} & \multicolumn{2}{c|}{0} \\ \hline
    \multicolumn{2}{|c|}{$m(J)$} & 1 & \multicolumn{1}{c|}{2} & 1 & \multicolumn{1}{c|}{2} & 1 & \multicolumn{1}{c|}{2} & 4 & \multicolumn{1}{c|}{6} & 4 & \multicolumn{1}{c|}{6} & 4 & 6 \\ \hline
    \multicolumn{1}{|c|}{$J$} & \multicolumn{1}{c|}{Facet DOFs} & \multicolumn{12}{c|}{Iteration Counts} \\ \hline
    \multicolumn{1}{|c|}{2} & \multicolumn{1}{c|}{$3.92 e2$} & 9 & 7 & 9 & 7 & 9 & \multicolumn{1}{c|}{7} & N/A & 4 & 6 & 5 & 6 & 6 \\ \cline{1-2}
    \multicolumn{1}{|c|}{3} & \multicolumn{1}{c|}{$1.65 e3$} & 13 & 10 & 10 & 8 & 10 & \multicolumn{1}{c|}{8} & N/A & 7 & 6 & 6 & 6 & 6 \\ \cline{1-2}
    \multicolumn{1}{|c|}{4} & \multicolumn{1}{c|}{$6.75 e3$} & 16 & 13 & 11 & 8 & 11 & \multicolumn{1}{c|}{8} & N/A & 10 & 6 & 6 & 6 & 6 \\ \cline{1-2}
    \multicolumn{1}{|c|}{5} & \multicolumn{1}{c|}{$2.73 e4$} & 18 & 14 & 11 & 9 & 11 & \multicolumn{1}{c|}{9} & N/A & 7 & 6 & 6 & 6 & 6 \\ \cline{1-2}
    \multicolumn{1}{|c|}{6} & \multicolumn{1}{c|}{$1.10 e5$} & 17 & 12 & 12 & 9 & 12 & \multicolumn{1}{c|}{9} & N/A & 6 & 6 & 6 & 6 & 6 \\ \cline{1-2}
    \multicolumn{1}{|c|}{7} & \multicolumn{1}{c|}{$4.41 e5$} & 15 & 11 & 12 & 9 & 12 & \multicolumn{1}{c|}{9} & N/A & 6 & 6 & 6 & 6 & 6 \\ \cline{1-2}
    \multicolumn{1}{|c|}{8} & \multicolumn{1}{c|}{$1.77 e6$} & 13 & 10 & 12 & 9 & 12 & \multicolumn{1}{c|}{9} & N/A & 6 & 6 & 6 & 6 & 6 
    \\ \hline \hline
    \multicolumn{14}{|c|}{$d=3$} \\ \hline
    \multicolumn{2}{|c|}{Multigrid} & \multicolumn{6}{c|}{Variable V cycle} & \multicolumn{6}{c|}{W cycle} \\ \hline
    \multicolumn{2}{|c|}{$\beta$} & \multicolumn{2}{c|}{1000} & \multicolumn{2}{c|}{1} & \multicolumn{2}{c|}{0} & \multicolumn{2}{c|}{1000} & \multicolumn{2}{c|}{1} & \multicolumn{2}{c|}{0} \\ \hline
    \multicolumn{2}{|c|}{$m(J)$} & 1 & \multicolumn{1}{c|}{2} & 1 & \multicolumn{1}{c|}{2} & 1 & \multicolumn{1}{c|}{2} & 4 & \multicolumn{1}{c|}{6} & 4 & \multicolumn{1}{c|}{6} & 4 & 6 \\ \hline
    \multicolumn{1}{|c|}{$J$} & \multicolumn{1}{c|}{Facet DOFs} & \multicolumn{12}{c|}{Iteration Counts} \\ \hline
    \multicolumn{1}{|c|}{2} & \multicolumn{1}{c|}{$3.02 e3$} & 7 & 5 & 6 & 5 & 6 & \multicolumn{1}{c|}{5} & 4 & 3 & 4 & 4 & 4 & 4 \\ \cline{1-2}
    \multicolumn{1}{|c|}{3} & \multicolumn{1}{c|}{$8.55 e3$} & 7 & 6 & 6 & 5 & 6 & \multicolumn{1}{c|}{5} & 5 & 4 & 5 & 4 & 5 & 4 \\ \cline{1-2}
    \multicolumn{1}{|c|}{4} & \multicolumn{1}{c|}{$2.47 e4$} & 8 & 7 & 8 & 6 & 8 & \multicolumn{1}{c|}{6} & 6 & 5 & 5 & 5 & 5 & 5 \\ \cline{1-2}
    \multicolumn{1}{|c|}{5} & \multicolumn{1}{c|}{$6.88 e4$} & 9 & 7 & 7 & 6 & 7 & \multicolumn{1}{c|}{6} & 6 & 5 & 5 & 5 & 5 & 5 \\ \cline{1-2}
    \multicolumn{1}{|c|}{6} & \multicolumn{1}{c|}{$1.82 e5$} & 10 & 8 & 8 & 7 & 8 & \multicolumn{1}{c|}{7} & 6 & 6 & 5 & 5 & 5 & 5 \\ \cline{1-2}
    \multicolumn{1}{|c|}{7} & \multicolumn{1}{c|}{$4.78 e5$} & 10 & 8 & 8 & 7 & 8 & \multicolumn{1}{c|}{7} & 6 & 6 & 5 & 5 & 5 & 5 \\ \cline{1-2}
    \multicolumn{1}{|c|}{8} & \multicolumn{1}{c|}{$1.22 e6$} & 10 & 9 & 8 & 7 & 8 & \multicolumn{1}{c|}{7} & 5 & 5 & 5 & 5 & 5 & 5 \\ \hline
    \end{tabular}
        \caption{Iteration counts of the preconditioned conjugate gradient solver for the generalized Stokes equations in backward-facing step flow problems \red{with uniformly refined mesh} in Example \ref{subsec:stokesBackNum}.}
        \label{tab:stokesBack}
\end{table}

\begin{figure}[ht]
    \centering
    \includegraphics[width=.75\textwidth]{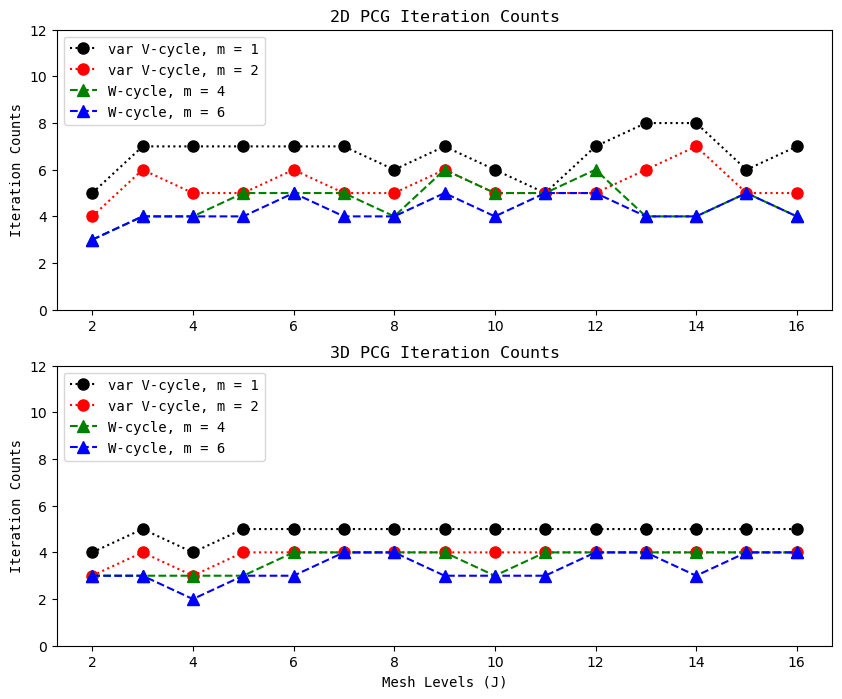}
    \caption{Iteration counts of the preconditioned conjugate gradient solver for the generalized Stokes equations in backward-facing step flow problems \red{with adaptively refined mesh} in Example \ref{subsec:stokesBackNum}.}
    \label{fig:stokesBack-adapt}
\end{figure}

\section{Conclusion}
\label{sec:conclude}
In this study, we present the lowest order HDG schemes with projected jumps and numerical integration ({\sf HDG-P0}) for the reaction-diffusion equation and the generalized Stokes equations, prove their optimal a priori error analysis, and construct optimal multigrid algorithms for the condensed {\sf HDG-P0} schemes for the two sets of equations \red{on conforming simplicial meshes}. The key idea of constructing the optimal multigrid is the equivalence between the condensed {\sf HDG-P0} scheme and the (slightly modified) CR discretization, which enables us to follow the rich literature on multigrid algorithms for CR discretizations to design robust multigrid schemes for {\sf HDG-P0}. Numerical experiments are presented with the proposed multigrid algorithms as \red{both iterative solvers and preconditioners} to solve the condensed {\sf HDG-P0} schemes, where optimal $L_2$ norm convergence rates are obtained, and the condition number of the preconditioned operators are bounded independent of mesh level and mesh size. We further note that the proposed $h$-multigrid algorithms for the condensed {\sf HDG-P0} schemes in this study can be used as a building block to construct robust preconditioners for 
higher-order HDG schemes using the auxiliary space preconditioning technique \cite{Xu96}, which will be our forthcoming research.


\section*{Appendix: Proof of \eqref{dual-q}}
\red{
Let $(\underline{\psi}, \phi)\in \underline{H}^1(\Omega) \times (H^2(\Omega)\cap H^1_0(\Omega))$ be the solution to the following dual problem 
together with the regularity assumption as in \eqref{reg}:
\begin{align}
    \label{DDD}
    \alpha^{-1}\underline{\psi} + \nabla\phi = 0, \quad
    \nabla\cdot\underline{\psi} + \beta\phi = e_u,\quad
    \text{in $\Omega$}.
\end{align}
}
We define $\varPi_V$ and $\varPi_W$ as $L_2$-projection onto the finite element space $V_h^0$ and $\underline{W}_h$. To further simplify notation, we denote
\[
    \phi_h:= \varPi_V\phi,\quad
    \psi_h:= \varPi_W\psi.
\]
Moreover, by continuity of the $L_2$-projection and the continuous dependency result on $\phi$, there holds:
\begin{align}
    \label{cont}
\|\phi_h\|_{1,h}\lesssim \|\phi\|_1\lesssim \alpha_0^{-1} \|e_u\|_0,
\end{align}
where the norm $\|\phi_h\|_{1,h}:=\sqrt{\sum_{K\in\Th}
\|\phi_h\|_{0,K}^2+\|\nabla\phi_h\|_{0,K}^2}$.

Using the definition of the dual problem \eqref{DDD}, we have, for any $\underline\psi_h\in \underline{W}_h$ and $\phi_h\in V_h$,
\begin{align*}
    \|e_u\|_0^2
    =&\;(e_u, \nabla\cdot\underline\psi +\beta \phi)_{\Th}\\ 
    =&\;(e_u, \nabla\cdot(\underline\psi-\underline{\psi}_h))_{\Th} +(e_u,\beta (\phi-\phi_h))_{\Th}
    +
    (e_u, \nabla\cdot \underline{\psi}_h)_{\Th} +(\beta e_u,\phi_h)_{\Th}.
\end{align*}
Taking $\underline r_h=\underline \psi_h$ in \eqref{ellipErrEqn1}, we get 
\begin{align*}
    (e_u, \nabla\cdot \underline{\psi}_h)_{\Th}
= &\;
 (\alpha^{-1}\underline{e}_\sigma,\,\underline{\psi}_h)_{\Th}
    +\langle e_{\widehat{u}},\, \underline{\psi}_h\cdot\underline{n}_K\rangle_{\partial\Th}\\
    =&\;
(\alpha^{-1}\underline{e}_\sigma,\,\underline{\psi}_h-\underline\psi)_{\Th}
-(\underline{e}_\sigma,\,\nabla\phi)_{\Th}
    +\langle e_{\widehat{u}},\, (\underline{\psi}_h-\underline{\psi})\cdot\underline{n}_K\rangle_{\partial\Th},    
\end{align*}
where we used the first equation in \eqref{DDD}, 
single-valuedness of $e_{\widehat{u}}$ on the interior mesh skeleton $\mathcal{E}_h^o$ and $e_{\widehat{u}}|_{\partial\Omega}=0$ 
in the second step.
Taking 
$v_h=\phi_h$ in \eqref{ellipErrEqn2}, we have
\begin{align*}
 (\beta e_u, \phi_h)_{\Th}
 =&\;
\underbrace{(\beta e_u, \phi_h)_{\Th}-
  \sum_{K\in\Th}Q_K^1(\beta e_u \phi_h)}_{:=T_3(\phi_h)}\\
  &\;
+
 (\underline{e}_{\sigma}, \nabla\phi_h)_{\Th}
    -\langle\underline{e}_{\sigma}\cdot\underline n_K+\tau_K\varPi_0(e_u-e_{\widehat{u}}), \phi_h\rangle_{\partial\Th}
    +T_1(\phi_h)-T_2(\phi_h)\\
    =&\;
(\underline{e}_{\sigma}, \nabla\phi_h)_{\Th}
    -\langle\underline{e}_{\sigma}\cdot\underline n_K+\tau_K\varPi_0(e_u-e_{\widehat{u}}), \phi_h-\phi\rangle_{\partial\Th}
    +T_1(\phi_h)-T_2(\phi_h)+T_3(\phi_h),    
\end{align*}
where in the second step we used the single-valuedness of 
$\underline{e}_{\sigma}\cdot\underline n_K+\tau_K\varPi_0(e_u-e_{\widehat{u}})$
on the interior mesh skeleton $\mathcal{E}_h^o$ 
and $\phi|_{\partial\Omega}=0$.
Combing the above three equalities and simplifying, we get 
\begin{align*}
    \|e_u\|_0^2
    = &\;\underbrace{-(\alpha^{-1}\underline{e}_\sigma+
    \nabla e_u, \underline\psi-\underline\psi_h)_{\Th}}_{:=I_1}
    +\underbrace{\langle e_u-\widehat{e}_u, (\underline\psi-\underline\psi_h)\cdot\underline n_K\rangle_{\partial\Th}}_{:=I_2}
    +
    \underbrace{(\beta e_u+
    \nabla\cdot\underline{e}_\sigma, \phi-\phi_h)_{\Th}}_{:=I_3}
\\
&\;
+\underbrace{\langle \tau_K\Pi_0(e_u-\widehat{e}_u), (\phi-\phi_h)\rangle_{\partial\Th}}_{:=I_4}
+T_1(\phi_h)-T_2(\phi_h)+T_3(\phi_h).
\end{align*}
Using approximation properties of the $L_2$-projections and \eqref{cont}, we can bound each of the above right hand side terms by the following:
\begin{align*}
    I_1\le &\;(\|\nabla e_u\|_0
    +\alpha_0^{-1/2}\|\alpha^{-1/2}\underline e_\sigma\|_0)
    \,\|\underline\psi-\underline\psi_h\|_0\lesssim \alpha_0^{-1/2}h^2\Xi\|\psi\|_1,\\
    I_2\lesssim &\;
    \|h_K^{-1/2}(e_u-\widehat e_u)\|_{0,\partial\Th}
\|h_K^{1/2}(\underline\psi-\underline\psi_h)\|_{0,\partial\Th}
\lesssim
\alpha_0^{-1/2}h^2\Xi\|\psi\|_1,\\
I_3\lesssim&\; (\beta_1\|e_u\|_0+\|\nabla\cdot\underline{e}_\sigma\|_0) \|\phi-\phi_h\|_0
\lesssim \left(\beta_1\|e_u\|_0+\alpha_1^{1/2}\Xi\right) h^2\|\phi\|_2
,\\
I_4\lesssim &\; 
    \|\tau_K^{1/2}\Pi_0(e_u-\widehat e_u)\|_{0,\partial\Th}
\|\tau_K^{1/2}(\phi-\phi_h)\|_{0,\partial\Th}
\lesssim \alpha_1^{1/2}\Xi h^2\|\phi\|_2,\\
T_3(\phi_h)\lesssim &\;h\|\beta\|_{1,\infty}\|e_u\|_{1,h}\|\phi_h\|_0
\lesssim
\|\beta\|_{1,\infty}\alpha_0^{-3/2}h^2\Xi\|e_u\|_0,\\ 
T_1(\phi_h)-T_2(\phi_h)\lesssim &\;
h\|\beta\|_{1,\infty}\|u-u_h^1\|_{1,h}\|\phi_h\|_0+
h^{1+d/2}\|f-\beta u\|_{1,\infty}\|\phi_h\|_{1}\\
\lesssim &\; 
\left(\|\beta\|_{1,\infty}\alpha_0^{-3/2}h^2\Theta
+h^{1+d/2}\alpha_0^{-1}\|f-\beta u\|_{1,\infty}\right)\|e_u\|_{0}.
\end{align*}
Combining these inequalities with the regularity assumption \eqref{reg}, we finally get
\begin{align*}
    \|e_u\|_0^2\lesssim&\; h^2\Xi(\alpha_0^{-1/2}\|\underline\psi\|_1+\alpha_1^{1/2}\|\phi\|_2)\\
    &\;+
\left(\beta_1\|\phi\|_2
+\|\beta\|_{1,\infty}\alpha_0^{-3/2}(\Theta+\Xi)
+h^{d/2-1}\alpha_0^{-1}\|f-\beta u\|_{1,\infty}
\right)h^2\|e_u\|_0\\
\lesssim \;
\Bigg(&c_{reg}\Xi+\frac{c_{reg}\alpha_0^{-1/2}\beta_1}{\alpha_1^{1/2}+\beta_1^{1/2}h}\Xi
+\|\beta\|_{1,\infty}\alpha_0^{-3/2}(\Theta+\Xi)
+h^{d/2-1}\alpha_0^{-1}\|f-\beta u\|_{1,\infty}
\Bigg)h^2\|e_u\|_0.
\end{align*}
The estimate \eqref{dual-q} now following from the above inequality and  \eqref{dual-ex} by the triangle inequality. This completes the proof of \eqref{dual-q}.\qed

\bibliography{mgHDG}

\begin{thebibliography}{10}

\bibitem{arnold1985mixed}
{\sc D.~N. Arnold and F.~Brezzi}, {\em Mixed and nonconforming finite element
  methods: implementation, postprocessing and error estimates}, RAIRO
  Mod\'{e}l. Math. Anal. Num\'{e}r., 19 (1985), pp.~7--32.

\bibitem{arnold2000multigrid}
{\sc D.~N. Arnold, R.~S. Falk, and R.~Winther}, {\em Multigrid in {$H({\rm
  div})$} and {$H({\rm curl})$}}, Numer. Math., 85 (2000), pp.~197--217.

\bibitem{betteridge2021multigrid}
{\sc J.~Betteridge, T.~H. Gibson, I.~G. Graham, and E.~H. M\"{u}ller}, {\em
  Multigrid preconditioners for the hybridised discontinuous {G}alerkin
  discretisation of the shallow water equations}, J. Comput. Phys., 426 (2021),
  pp.~Paper No. 109948, 34.

\bibitem{braess1990multigrid}
{\sc D.~Braess and R.~Verf\"{u}rth}, {\em Multigrid methods for nonconforming
  finite element methods}, SIAM J. Numer. Anal., 27 (1990), pp.~979--986.

\bibitem{bramble1987new}
{\sc J.~H. Bramble and J.~E. Pasciak}, {\em New convergence estimates for
  multigrid algorithms}, Math. Comp., 49 (1987), pp.~311--329.

\bibitem{bramble1991analysis}
{\sc J.~H. Bramble, J.~E. Pasciak, and J.~Xu}, {\em The analysis of multigrid
  algorithms with nonnested spaces or noninherited quadratic forms}, Math.
  Comp., 56 (1991), pp.~1--34.

\bibitem{Linke14b}
{\sc C.~Brennecke, A.~Linke, C.~Merdon, and J.~Sch\"{o}berl}, {\em Optimal and
  pressure-independent {$L^2$} velocity error estimates for a modified
  {C}rouzeix-{R}aviart element with {BDM} reconstructions}, in Finite volumes
  for complex applications {VII}. {M}ethods and theoretical aspects, vol.~77 of
  Springer Proc. Math. Stat., Springer, Cham, 2014, pp.~159--167.

\bibitem{brenner1989optimal}
{\sc S.~C. Brenner}, {\em An optimal-order multigrid method for {${\rm P}1$}
  nonconforming finite elements}, Math. Comp., 52 (1989), pp.~1--15.

\bibitem{Brenner89X}
\leavevmode\vrule height 2pt depth -1.6pt width 23pt, {\em An optimal-order
  nonconforming multigrid method for the biharmonic equation}, SIAM J. Numer.
  Anal., 26 (1989), pp.~1124--1138.

\bibitem{brenner1990nonconforming}
\leavevmode\vrule height 2pt depth -1.6pt width 23pt, {\em A nonconforming
  multigrid method for the stationary {S}tokes equations}, Math. Comp., 55
  (1990), pp.~411--437.

\bibitem{brenner1992multigrid}
\leavevmode\vrule height 2pt depth -1.6pt width 23pt, {\em A multigrid
  algorithm for the lowest-order {R}aviart-{T}homas mixed triangular finite
  element method}, SIAM J. Numer. Anal., 29 (1992), pp.~647--678.

\bibitem{brenner1993nonconforming}
\leavevmode\vrule height 2pt depth -1.6pt width 23pt, {\em A nonconforming
  mixed multigrid method for the pure displacement problem in planar linear
  elasticity}, SIAM J. Numer. Anal., 30 (1993), pp.~116--135.

\bibitem{brenner1994nonconforming}
\leavevmode\vrule height 2pt depth -1.6pt width 23pt, {\em A nonconforming
  mixed multigrid method for the pure traction problem in planar linear
  elasticity}, Math. Comp., 63 (1994), pp.~435--460, S1--S5.

\bibitem{brenner1999convergence}
\leavevmode\vrule height 2pt depth -1.6pt width 23pt, {\em Convergence of
  nonconforming multigrid methods without full elliptic regularity}, Math.
  Comp., 68 (1999), pp.~25--53.

\bibitem{brenner2003poincare}
\leavevmode\vrule height 2pt depth -1.6pt width 23pt, {\em
  Poincar\'{e}-{F}riedrichs inequalities for piecewise {$H^1$} functions}, SIAM
  J. Numer. Anal., 41 (2003), pp.~306--324.

\bibitem{brenner2004convergence}
\leavevmode\vrule height 2pt depth -1.6pt width 23pt, {\em Convergence of
  nonconforming {V}-cycle and {F}-cycle multigrid algorithms for second order
  elliptic boundary value problems}, Math. Comp., 73 (2004), pp.~1041--1066.

\bibitem{brenner2015forty}
\leavevmode\vrule height 2pt depth -1.6pt width 23pt, {\em Forty years of the
  {C}rouzeix-{R}aviart element}, Numer. Methods Partial Differential Equations,
  31 (2015), pp.~367--396.

\bibitem{cai2010recovery}
{\sc Z.~Cai and S.~Zhang}, {\em Recovery-based error estimators for interface
  problems: mixed and nonconforming finite elements}, SIAM J. Numer. Anal., 48
  (2010), pp.~30--52.

\bibitem{chen2014robust}
{\sc H.~Chen, P.~Lu, and X.~Xu}, {\em A robust multilevel method for
  hybridizable discontinuous {G}alerkin method for the {H}elmholtz equation},
  J. Comput. Phys., 264 (2014), pp.~133--151.

\bibitem{wgchen}
{\sc L.~Chen, J.~Wang, Y.~Wang, and X.~Ye}, {\em An auxiliary space multigrid
  preconditioner for the weak {G}alerkin method}, Comput. Math. Appl., 70
  (2015), pp.~330--344.

\bibitem{chen1996equivalence}
{\sc Z.~Chen}, {\em Equivalence between and multigrid algorithms for
  nonconforming and mixed methods for second-order elliptic problems},
  East-West J. Numer. Math., 4 (1996), pp.~1--33.

\bibitem{chen1994analysis}
{\sc Z.~Chen and D.~Y. Kwak}, {\em The analysis of multigrid algorithms for
  nonconforming and mixed methods for second order elliptic problems}, in IMA
  Preprints Series, Institute for Mathematics and Its Applications, University
  of Minnesota, 1994.

\bibitem{ciarlet2002finite}
{\sc P.~G. Ciarlet}, {\em The finite element method for elliptic problems},
  vol.~40 of Classics in Applied Mathematics, Society for Industrial and
  Applied Mathematics (SIAM), Philadelphia, PA, 2002.
\newblock Reprint of the 1978 original [North-Holland, Amsterdam; MR0520174 (58
  \#25001)].

\bibitem{Cockburn16}
{\sc B.~Cockburn}, {\em Static condensation, hybridization, and the devising of
  the {HDG} methods}, in Building bridges: connections and challenges in modern
  approaches to numerical partial differential equations, vol.~114 of Lect.
  Notes Comput. Sci. Eng., Springer, [Cham], 2016, pp.~129--177.

\bibitem{cockburn2018discontinuous}
\leavevmode\vrule height 2pt depth -1.6pt width 23pt, {\em Discontinuous
  galerkin methods for computational fluid dynamics}, in Encyclopedia of
  Computational Mechanics Second Edition, John Wiley \& Sons, Ltd, 2017,
  pp.~1--63.

\bibitem{CockburnDubois14}
{\sc B.~Cockburn, O.~Dubois, J.~Gopalakrishnan, and S.~Tan}, {\em Multigrid for
  an {HDG} method}, IMA J. Numer. Anal., 34 (2014), pp.~1386--1425.

\bibitem{CockburnGopalakrishnanLazarov09}
{\sc B.~Cockburn, J.~Gopalakrishnan, and R.~Lazarov}, {\em Unified
  hybridization of discontinuous {G}alerkin, mixed, and continuous {G}alerkin
  methods for second order elliptic problems}, SIAM J. Numer. Anal., 47 (2009),
  pp.~1319--1365.

\bibitem{cockburn2016hdg}
{\sc B.~Cockburn, N.~C. Nguyen, and J.~Peraire}, {\em H{DG} methods for
  hyperbolic problems}, in Handbook of numerical methods for hyperbolic
  problems, vol.~17 of Handb. Numer. Anal., Elsevier/North-Holland, Amsterdam,
  2016, pp.~173--197.

\bibitem{di2021h}
{\sc D.~A. Di~Pietro, F.~H\"{u}lsemann, P.~Matalon, P.~Mycek, U.~R\"{u}de, and
  D.~Ruiz}, {\em An {$h$}-multigrid method for hybrid high-order
  discretizations}, SIAM J. Sci. Comput., 43 (2021), pp.~S839--S861.

\bibitem{di2021towards}
\leavevmode\vrule height 2pt depth -1.6pt width 23pt, {\em Towards robust, fast
  solutions of elliptic equations on complex domains through hybrid high-order
  discretizations and non-nested multigrid methods}, Internat. J. Numer.
  Methods Engrg., 122 (2021), pp.~6576--6595.

\bibitem{duan2007generalized}
{\sc H.-Y. Duan, S.-Q. Gao, R.~C.~E. Tan, and S.~Zhang}, {\em A generalized
  {BPX} multigrid framework covering nonnested {V}-cycle methods}, Math. Comp.,
  76 (2007), pp.~137--152.

\bibitem{fabien2019manycore}
{\sc M.~S. Fabien, M.~G. Knepley, R.~T. Mills, and B.~M. Rivi\`ere}, {\em
  Manycore parallel computing for a hybridizable discontinuous {G}alerkin
  nested multigrid method}, SIAM J. Sci. Comput., 41 (2019), pp.~C73--C96.

\bibitem{Farrell20}
{\sc P.~E. Farrell, L.~Mitchell, L.~R. Scott, and F.~Wechsung}, {\em Robust
  multigrid methods for nearly incompressible elasticity using macro elements},
  arXiv preprint arXiv:2002.02051,  (2020).

\bibitem{farrell2019augmented}
{\sc P.~E. Farrell, L.~Mitchell, and F.~Wechsung}, {\em An augmented
  {L}agrangian preconditioner for the 3{D} stationary incompressible
  {N}avier-{S}tokes equations at high {R}eynolds number}, SIAM J. Sci. Comput.,
  41 (2019), pp.~A3073--A3096.

\bibitem{fernandez2018hybridized}
{\sc P.~Fernandez, A.~Christophe, S.~Terrana, N.~C. Nguyen, and J.~Peraire},
  {\em Hybridized discontinuous {G}alerkin methods for wave propagation}, J.
  Sci. Comput., 77 (2018), pp.~1566--1604.

\bibitem{fortin2000augmented}
{\sc M.~Fortin and R.~Glowinski}, {\em Augmented {L}agrangian methods}, vol.~15
  of Studies in Mathematics and its Applications, North-Holland Publishing Co.,
  Amsterdam, 1983.
\newblock Applications to the numerical solution of boundary value problems,
  Translated from the French by B. Hunt and D. C. Spicer.

\bibitem{fu2021locking}
{\sc G.~Fu, C.~Lehrenfeld, A.~Linke, and T.~Streckenbach}, {\em Locking-free
  and gradient-robust {$H(\rm div)$}-conforming {HDG} methods for linear
  elasticity}, J. Sci. Comput., 86 (2021), pp.~Paper No. 39, 30.

\bibitem{hackbusch2013multi}
{\sc W.~Hackbusch}, {\em Multigrid methods and applications}, vol.~4 of
  Springer Series in Computational Mathematics, Springer-Verlag, Berlin, 1985.

\bibitem{hong2016robust}
{\sc Q.~Hong, J.~Kraus, J.~Xu, and L.~Zikatanov}, {\em A robust multigrid
  method for discontinuous {G}alerkin discretizations of {S}tokes and linear
  elasticity equations}, Numer. Math., 132 (2016), pp.~23--49.

\bibitem{huang2020nonconforming}
{\sc X.~Huang}, {\em Nonconforming finite element stokes complexes in three
  dimensions}, arXiv preprint arXiv:2007.14068,  (2020).

\bibitem{Volker17}
{\sc V.~John, A.~Linke, C.~Merdon, M.~Neilan, and L.~G. Rebholz}, {\em On the
  divergence constraint in mixed finite element methods for incompressible
  flows}, SIAM Rev., 59 (2017), pp.~492--544.

\bibitem{Kanschat15}
{\sc G.~Kanschat and Y.~Mao}, {\em Multigrid methods for {$H^{\rm
  div}$}-conforming discontinuous {G}alerkin methods for the {S}tokes
  equations}, J. Numer. Math., 23 (2015), pp.~51--66.

\bibitem{kolev2016multilevel}
{\sc T.~V. Kolev, J.~Xu, and Y.~Zhu}, {\em Multilevel preconditioners for
  reaction-diffusion problems with discontinuous coefficients}, J. Sci.
  Comput., 67 (2016), pp.~324--350.

\bibitem{Lee09}
{\sc Y.-J. Lee, J.~Wu, and J.~Chen}, {\em Robust multigrid method for the
  planar linear elasticity problems}, Numer. Math., 113 (2009), pp.~473--496.

\bibitem{lee2007robust}
{\sc Y.-J. Lee, J.~Wu, J.~Xu, and L.~Zikatanov}, {\em Robust subspace
  correction methods for nearly singular systems}, Math. Models Methods Appl.
  Sci., 17 (2007), pp.~1937--1963.

\bibitem{lehrenfeld2010hybrid}
{\sc C.~Lehrenfeld}, {\em Hybrid {Discontinuous} {Galerkin} methods for solving
  incompressible flow problems}, PhD thesis, RWTH Aachen University, 2010.

\bibitem{Linke14a}
{\sc A.~Linke}, {\em On the role of the {H}elmholtz decomposition in mixed
  methods for incompressible flows and a new variational crime}, Comput.
  Methods Appl. Mech. Engrg., 268 (2014), pp.~782--800.

\bibitem{Linke16x}
{\sc A.~Linke and C.~Merdon}, {\em Pressure-robustness and discrete {H}elmholtz
  projectors in mixed finite element methods for the incompressible
  {N}avier-{S}tokes equations}, Comput. Methods Appl. Mech. Engrg., 311 (2016),
  pp.~304--326.

\bibitem{lu2021homogeneous}
{\sc P.~Lu, A.~Rupp, and G.~Kanschat}, {\em {Homogeneous multigrid for {HDG}}},
  IMA J. Numer. Anal.,  (2021).
\newblock drab055.

\bibitem{marini1985inexpensive}
{\sc L.~D. Marini}, {\em An inexpensive method for the evaluation of the
  solution of the lowest order {R}aviart-{T}homas mixed method}, SIAM J. Numer.
  Anal., 22 (1985), pp.~493--496.

\bibitem{morley1968triangular}
{\sc L.~S.~D. Morley}, {\em The triangular equilibrium element in the solution
  of plate bending problems}, Aeronautical Quarterly, 19 (1968), p.~149–169.

\bibitem{nguyen2012hybridizable}
{\sc N.~C. Nguyen and J.~Peraire}, {\em Hybridizable discontinuous {G}alerkin
  methods for partial differential equations in continuum mechanics}, J.
  Comput. Phys., 231 (2012), pp.~5955--5988.

\bibitem{oikawa2015hybridized}
{\sc I.~Oikawa}, {\em A hybridized discontinuous {G}alerkin method with reduced
  stabilization}, J. Sci. Comput., 65 (2015), pp.~327--340.

\bibitem{oikawa2015reduced}
\leavevmode\vrule height 2pt depth -1.6pt width 23pt, {\em Analysis of a
  reduced-order {HDG} method for the {S}tokes equations}, J. Sci. Comput., 67
  (2016), pp.~475--492.

\bibitem{qiu2018hdg}
{\sc W.~Qiu, J.~Shen, and K.~Shi}, {\em An {HDG} method for linear elasticity
  with strong symmetric stresses}, Math. Comp., 87 (2018), pp.~69--93.

\bibitem{qiu2016hdg}
{\sc W.~Qiu and K.~Shi}, {\em An {HDG} method for convection diffusion
  equation}, J. Sci. Comput., 66 (2016), pp.~346--357.

\bibitem{qiu2016superconvergent}
\leavevmode\vrule height 2pt depth -1.6pt width 23pt, {\em A superconvergent
  {HDG} method for the incompressible {N}avier-{S}tokes equations on general
  polyhedral meshes}, IMA J. Numer. Anal., 36 (2016), pp.~1943--1967.

\bibitem{schoberl1999multigrid}
{\sc J.~Sch{\"o}berl}, {\em Multigrid methods for a parameter dependent problem
  in primal variables}, Numer. Math., 84 (1999), pp.~97--119.

\bibitem{schoberl1999robust}
\leavevmode\vrule height 2pt depth -1.6pt width 23pt, {\em Robust multigrid
  methods for parameter dependent problems}, PhD thesis, Johannes Kepler
  University Linz, 1999.

\bibitem{Schoberl16}
\leavevmode\vrule height 2pt depth -1.6pt width 23pt, {\em {C}++11
  {I}mplementation of {F}inite {E}lements in {NGS}olve}, 2014.
\newblock {ASC Report 30/2014, Institute for Analysis and Scientific Computing,
  Vienna University of Technology}.

\bibitem{stevenson1998cascade}
{\sc R.~Stevenson}, {\em Nonconforming finite elements and the cascadic
  multi-grid method}, Numer. Math., 91 (2002), pp.~351--387.

\bibitem{Tan09}
{\sc S.~Tan}, {\em Iterative solvers for hybridized finite element methods},
  PhD thesis, University of Florida, 2009.

\bibitem{turek1994multigrid}
{\sc S.~Turek}, {\em Multigrid techniques for a divergence-free finite element
  discretization}, East-West J. Numer. Math., 2 (1994), pp.~229--255.

\bibitem{uzawa1958iterative}
{\sc H.~Uzawa}, {\em Iterative methods for concave programming}, in Studies in
  linear and non-linear programming, Stanford University Press, Stanford,
  Calif., 1958, pp.~154--165.

\bibitem{wildey2019unified}
{\sc T.~Wildey, S.~Muralikrishnan, and T.~Bui-Thanh}, {\em Unified geometric
  multigrid algorithm for hybridized high-order finite element methods}, SIAM
  J. Sci. Comput., 41 (2019), pp.~S172--S195.

\bibitem{Xu96}
{\sc J.~Xu}, {\em The auxiliary space method and optimal multigrid
  preconditioning techniques for unstructured grids}, vol.~56, 1996,
  pp.~215--235.
\newblock International GAMM-Workshop on Multi-level Methods (Meisdorf, 1994).

\bibitem{Xucurl}
{\sc B.~Zheng, Q.~Hu, and J.~Xu}, {\em A nonconforming finite element method
  for fourth order curl equations in {$\mathbb{R}^{3}$}}, Math. Comp., 80
  (2011), pp.~1871--1886.

\bibitem{zhu2014analysis}
{\sc Y.~Zhu}, {\em Analysis of a multigrid preconditioner for
  {C}rouzeix-{R}aviart discretization of elliptic partial differential equation
  with jump coefficients}, Numer. Linear Algebra Appl., 21 (2014), pp.~24--38.

\end{thebibliography}
\bibliographystyle{siam}

\end{document}